\documentclass[a4paper,10pt,reqno]{amsart}
\usepackage{amsmath,amsfonts,amsthm,amssymb,color}
\usepackage[T1]{fontenc}
\usepackage{pdfsync}
\usepackage{csquotes}
\usepackage{graphicx}
\usepackage{pstricks}
\usepackage{lmodern}
\usepackage{calc}
\usepackage{mathabx}
\usepackage{mathrsfs}

\usepackage{accents}
\newcommand{\dbtilde}[1]{\accentset{\approx}{#1}}

\usepackage{enumerate}
\usepackage{hyperref}

\numberwithin{equation}{section}
\usepackage{tikz}

\newcommand{\llangle}{\langle\! \langle}
\newcommand{\rrangle}{\rangle\! \rangle}

\newcommand{\frx}{\mathfrak{X}}

\newcommand{\be}{\beta}

\newcommand{\1}{\mathbf{1}}

\newcommand{\sq}{\square}

\newcommand{\lati}{\tilde{\la}}
\newcommand{\betati}{\tilde{\beta}}

\newcommand{\yti}{\tilde{y}}

\newcommand{\cq}{\mathcal{Q}}
\newcommand{\bu}{\mathbf{u}}
\newcommand{\bbu}{\mathbf{\bar{u}}}
\newcommand{\bv}{\mathbf{v}}
\newcommand{\bw}{\mathbf{w}}

\newcommand{\R}{\mathbb R}
\newcommand{\N}{\mathbb N}

\newcommand{\ca}{\mathcal A}
\newcommand{\cb}{\mathcal B}
\newcommand{\cac}{\mathcal C}
\newcommand{\cd}{\mathcal D}
\newcommand{\ce}{\mathcal E}
\newcommand{\cf}{\mathcal F}
\newcommand{\cg}{\mathcal G}
\newcommand{\ch}{\mathcal H}

\newcommand{\cj}{\mathcal J}

\newcommand{\cm}{\mathcal M}
\newcommand{\cn}{\mathcal N}
\newcommand{\cp}{\mathcal P}
\newcommand{\cs}{\mathcal S}
\newcommand{\cw}{\mathcal W}

\newcommand{\al}{\alpha}

\newcommand{\ga}{\gamma}

\newcommand{\ka}{\kappa}
\newcommand{\la}{\lambda}

\newcommand{\si}{\sigma}

\newcommand{\vp}{\varphi}

\newtheorem{theorem}{Theorem}[section]

\newtheorem{definition}[theorem]{Definition}

\newtheorem{lemma}[theorem]{Lemma}

\newtheorem{propdef}[theorem]{Proposition-Definition}
\newtheorem{proposition}[theorem]{Proposition}

\theoremstyle{remark}
\newtheorem{remark}[theorem]{Remark}

\pgfdeclareshape{crosscircle}
{
  \inheritsavedanchors[from=circle] 
  \inheritanchorborder[from=circle]
  \inheritanchor[from=circle]{north}
  \inheritanchor[from=circle]{north west}
  \inheritanchor[from=circle]{north east}
  \inheritanchor[from=circle]{center}
  \inheritanchor[from=circle]{west}
  \inheritanchor[from=circle]{east}
  \inheritanchor[from=circle]{mid}
  \inheritanchor[from=circle]{mid west}
  \inheritanchor[from=circle]{mid east}
  \inheritanchor[from=circle]{base}
  \inheritanchor[from=circle]{base west}
  \inheritanchor[from=circle]{base east}
  \inheritanchor[from=circle]{south}
  \inheritanchor[from=circle]{south west}
  \inheritanchor[from=circle]{south east}
  \inheritbackgroundpath[from=circle]
  \foregroundpath{
    \centerpoint%
    \pgf@xc=\pgf@x%
    \pgf@yc=\pgf@y%
    \pgfutil@tempdima=\radius%
    \pgfmathsetlength{\pgf@xb}{\pgfkeysvalueof{/pgf/outer xsep}}%
    \pgfmathsetlength{\pgf@yb}{\pgfkeysvalueof{/pgf/outer ysep}}%
    \ifdim\pgf@xb<\pgf@yb%
      \advance\pgfutil@tempdima by-\pgf@yb%
    \else%
      \advance\pgfutil@tempdima by-\pgf@xb%
    \fi%
    \pgfpathmoveto{\pgfpointadd{\pgfqpoint{\pgf@xc}{\pgf@yc}}{\pgfqpoint{-0.707107\pgfutil@tempdima}{-0.707107\pgfutil@tempdima}}}
    \pgfpathlineto{\pgfpointadd{\pgfqpoint{\pgf@xc}{\pgf@yc}}{\pgfqpoint{0.707107\pgfutil@tempdima}{0.707107\pgfutil@tempdima}}}
    \pgfpathmoveto{\pgfpointadd{\pgfqpoint{\pgf@xc}{\pgf@yc}}{\pgfqpoint{-0.707107\pgfutil@tempdima}{0.707107\pgfutil@tempdima}}}
    \pgfpathlineto{\pgfpointadd{\pgfqpoint{\pgf@xc}{\pgf@yc}}{\pgfqpoint{0.707107\pgfutil@tempdima}{-0.707107\pgfutil@tempdima}}}
  }
}
\makeatother

\colorlet{symbols}{blue!90!black}
\colorlet{symbolsb}{black!90!black}
\colorlet{testcolor}{green!60!black}

\usetikzlibrary{shapes.misc}
\usetikzlibrary{shapes.symbols}
\usetikzlibrary{snakes}
\usetikzlibrary{decorations}
\usetikzlibrary{decorations.markings}

\def\drawx{\draw[-,solid] (-3pt,-3pt) -- (3pt,3pt);\draw[-,solid] (-3pt,3pt) -- (3pt,-3pt);}
\tikzset{
	root/.style={circle,fill=testcolor,inner sep=0pt, minimum size=2mm},
	dot/.style={circle,fill=black,inner sep=0pt, minimum size=1mm},
	var/.style={circle,fill=black!10,draw=black,inner sep=0pt, minimum size=2mm},
	dotred/.style={circle,fill=black!50,inner sep=0pt, minimum size=2mm},
	generic/.style={semithick,shorten >=1pt,shorten <=1pt},
	dist/.style={ultra thick,draw=testcolor,shorten >=1pt,shorten <=1pt},
	testfcn/.style={ultra thick,testcolor,shorten >=1pt,shorten <=1pt,<-},
	testfcnx/.style={ultra thick,testcolor,shorten >=1pt,shorten <=1pt,<-,
		postaction={decorate,decoration={markings,mark=at position 0.6 with {\drawx}}}},
	kprime/.style={semithick,shorten >=1pt,shorten <=1pt,densely dashed,->},
	kprimex/.style={semithick,shorten >=1pt,shorten <=1pt,densely dashed,->,
		postaction={decorate,decoration={markings,mark=at position 0.4 with {\drawx}}}},
	kernel/.style={semithick,shorten >=1pt,shorten <=1pt,->},
	multx/.style={shorten >=1pt,shorten <=1pt,
		postaction={decorate,decoration={markings,mark=at position 0.5 with {\drawx}}}},
	kernelx/.style={semithick,shorten >=1pt,shorten <=1pt,->,
		postaction={decorate,decoration={markings,mark=at position 0.4 with {\drawx}}}},
	kernel1/.style={->,semithick,shorten >=1pt,shorten <=1pt,postaction={decorate,decoration={markings,mark=at position 0.45 with {\draw[-] (0,-0.1) -- (0,0.1);}}}},
	kernel2/.style={->,semithick,shorten >=1pt,shorten <=1pt,postaction={decorate,decoration={markings,mark=at position 0.45 with {\draw[-] (0.05,-0.1) -- (0.05,0.1);\draw[-] (-0.05,-0.1) -- (-0.05,0.1);}}}},
	kernelBig/.style={semithick,shorten >=1pt,shorten <=1pt,decorate, decoration={zigzag,amplitude=1.5pt,segment length = 3pt,pre length=2pt,post length=2pt}},
	rho/.style={dotted,semithick,shorten >=1pt,shorten <=1pt},
	renorm/.style={shape=circle,fill=white,inner sep=1pt},
	labl/.style={shape=rectangle,fill=white,inner sep=1pt},
	xi/.style={circle,fill=symbols!10,draw=symbols,inner sep=0pt,minimum size=1.2mm},
	xiblack/.style={circle,fill=symbolsb,draw=symbolsb,inner sep=0pt,minimum size=1.2mm},
	xix/.style={crosscircle,fill=symbols!10,draw=symbols,inner sep=0pt,minimum size=1.2mm},
	xib/.style={circle,fill=symbols!10,draw=symbols,inner sep=0pt,minimum size=1.6mm},
	xibx/.style={crosscircle,fill=symbols!10,draw=symbols,inner sep=0pt,minimum size=1.6mm},
	not/.style={circle,fill=symbols,draw=symbols,inner sep=0pt,minimum size=0.5mm},
	notblack/.style={circle,fill=symbolsb,draw=symbolsb,inner sep=0pt,minimum size=0.5mm},
	>=stealth,
	}
\makeatletter
\def\DeclareSymbol#1#2#3{\expandafter\gdef\csname MH@symb@#1\endcsname{\tikz[baseline=#2,scale=0.15,draw=symbols]{#3}}\expandafter\gdef\csname MH@symb@#1s\endcsname{\scalebox{0.7}{\tikz[baseline=#2,scale=0.15,draw=symbols]{#3}}}}
\def\<#1>{\csname MH@symb@#1\endcsname}
\makeatother

\DeclareSymbol{circle}{0.5}{\draw (0,0.7) node[xi] {};}
\DeclareSymbol{line}{0.5}{\draw (0,0.2) node[not] {} -- (0,1.4) node[not] {};}

\DeclareSymbol{Psi}{0.5}{\draw (0,0) node[not] {} -- (0,1.5) node[xi] {};}
\DeclareSymbol{Psiblack}{0.5}{\draw (0,0) node[notblack] {} -- (0,1.5) node[xiblack] {};}
\DeclareSymbol{Psi2}{0.5}{\draw (-0.8,1) node[xi] {} -- (0,0) node[not] {} -- (0.8,1) node[xi] {};}
\DeclareSymbol{IPsi2}{0}{\draw (0,1) -- (0.8,2.2) node[xi] {};\draw (0,-0.25) node[not] {} -- (0,1) node[not] {} -- (-0.8,2.2) node[xi] {};}
\DeclareSymbol{PsiIPsi2}{0}{\draw (0,1) -- (0.8,2.2) node[xi] {};\draw (0,-0.25) node[not] {} -- (0,1) node[not] {} -- (-0.8,2.2) node[xi] {};\draw (0,-0.25) node[not] {} -- (0.8,1) node[xi] {};}

\date{\today}

\begin{document}

\

\begin{center}
{\large\textbf{
A full discretization of the rough fractional linear heat equation
}}\\~\\
Aur\'elien Deya\footnote{Université de Lorraine, CNRS, IECL, F-54000 Nancy, France. Email: {\tt aurelien.deya@univ-lorraine.fr}}
and Renaud Marty\footnote{Université de Lorraine, CNRS, IECL, F-54000 Nancy, France. Email: {\tt renaud.marty@univ-lorraine.fr}.}
\end{center}

\bigskip

{\small \noindent {\bf Abstract:} We study a full discretization scheme for the stochastic linear heat equation
\begin{equation*}
\begin{cases}
\partial_t \<Psi>=\Delta \<Psi>+\dot{B}\, , \quad t\in [0,1], \ x\in \R,\\
\<Psi>_0=0\, ,
\end{cases}
\end{equation*}
when $\dot{B}$ is a very \emph{rough space-time fractional noise}.

\smallskip

The discretization procedure is divised into three steps: $(i)$ regularization of the noise through a mollifying-type approach; $(ii)$ discretization of the (smoothened) noise as a finite sum of Gaussian variables over rectangles in $[0,1]\times \R$; $(iii)$ discretization of the heat operator on the (non-compact) domain $[0,1]\times \R$, along the principles of Galerkin finite elements method. 

\smallskip

We establish the convergence of the resulting approximation to $\<Psi>$, which, in such a specific rough framework, can only hold in a space of distributions. We also provide some partial simulations of the algorithm. 

\bigskip


\bigskip

}

\section{Introduction and main result}

\subsection{Introduction}

\

\smallskip

The main objective of this study is to provide a full discretization scheme for the solution $\<Psi>$ of the stochastic heat equation
\begin{equation}\label{equa-luxo}
\begin{cases}
\partial_t \<Psi>=\Delta \<Psi>+\dot{B}\, , \quad t\in [0,1], \ x\in \R,\\
\<Psi>_0=0\, ,
\end{cases}
\end{equation}
where $\dot{B}$ is a stochastic space-time noise, defined on a complete probability space $(\Omega,\cf,\mathbb{P})$. In fact, the specificity of our analysis will lie in the consideration of a \emph{rough fractional noise}. Namely, given a fractional sheet $\{B_{t,x}, \, (t,x)\in \R^2\}$ of Hurst indexes $H_0,H_1\in (0,1)$ (see Definition \ref{defi:fractional-sheet} for details), we set
\begin{equation}\label{noise-as-deriv}
\dot{B}:=\frac{\partial^2 B}{\partial t \partial x} \, ,
\end{equation}
where the derivatives are understood in the sense of distributions. 

\smallskip

Due to its great flexibility, the fractional noise model has now been widely recognized as one of the most relevant alternatives to the standard white noise situation, whether for finite-dimensional systems or in SPDE settings. The fractional setting is also known to provide a convenient framework to study the influence of the noise roughness on the dynamics. In brief, when letting the parameters $H_0,H_1$ progressively decrease from $1$ to $0$, the regularity of $\dot{B}$ decreases as well, and the analysis becomes more and more intricate. 

\smallskip

In this context, let us recall that the space-time white noise setting precisely corresponds to the case where $H_0=H_1=\frac12$. In this specific situation, the approximation issue for the stochastic heat model \eqref{equa-luxo} or its extensions (whether with a multiplicative noise, or a non-linear drift) has been the source of a huge amount of papers since the late nineties and the pioneering works by Gy{\"o}ngy, Nualart and others (see for instance \cite{gyongy,gyongy2,gyongy-nualart-95,gyongy-nualart-97}, or \cite{anton-cohen-quer} and its bibliography for a survey on more recent developments). One of our main objectives in the present study will be to go beyond the standard white-noise situation, by considering a much less regular noise. Indeed, we will here focus on the case of a fractional noise $\dot{B}$ with indexes $H_0,H_1$ satisfying the condition
\begin{equation}\label{cond:h-0-h-1-intro}
0<2H_0+H_1<1\, .
\end{equation}

Our essential motivation for considering such a rough situation is actually easy to formulate: one can indeed show that as soon as $2H_0+H_1<1$, the solution $\<Psi>$ of \eqref{equa-luxo} is \emph{no longer a function in space}, but only a \emph{general distribution } (see Proposition \ref{sto} for details). Accordingly, the associated approximation issue cannot be examined through function norms either, and negative-order Sobolev topologies must come into the picture. This strongly contrasts with most of the existing statements in the white-noise literature (typically, convergence is therein established using the $L^2$-norm in space), and we thus consider our handling of negative-order Sobolev norms in the analysis as a new contribution in the understanding of the stochastic linear heat problem.

\smallskip

In this regard, the assumption \eqref{cond:h-0-h-1-intro}, leading to a distributional-valued $\<Psi>$, can be compared with the behaviour of the corresponding two-dimensional heat equation driven by a space-time white noise (that is, the equation on $[0,1]\times \R^2$ or $[0,1]\times \mathbb{T}^2$). Indeed, it turns out that the solution of such two-dimensional white-noise equation cannot be treated as a process with values in $L^2(\R^2)$, but only as a process with values in the Sobolev space $H^{-\varepsilon}(\R^2)$ for any $\varepsilon>0$ (see for instance \cite{daprato-deb-1,daprato-deb-2}). The consideration of a rough fractional noise thus allows us to face with a similar challenge, but in the one-dimensional setting, for which discretization methods are naturally more convenient to set up.

\

Another important motivation behind our interest for the linear solution $\<Psi>$ lies in the central role played by this process in many recent developments about the pathwise approach to general stochastic PDEs. For instance, in the study of the celebrated white-noise-driven $\Phi^4$-model
\begin{equation}\label{phi-4-model}
\partial_t \Phi=\Delta \Phi-\Phi^3+\xi , \quad t\in [0,T], \ x\in \mathbb{T}^d, \ d\in \{2,3\},
\end{equation}
the corresponding linear solution $\<Psi>$ (i.e., $\partial_t \<Psi>=\Delta \<Psi>+\xi$) can regarded as some first-order approximation of $\Phi$, and the analysis then consists in the control of the more regular path $\Psi:=\Phi-\<Psi>$ (see \cite{daprato-deb-2} for details when $d=2$, \cite[Section 9.2]{hai-14}, \cite[Section 6]{hairer-intro} when $d=3$, and let us also stress the fundamental role of $\<Psi>$ in the renormalization procedures associated with such singular models). Similar phenomena have recently been exhibited in the fractional situation for the quadratic counterpart of \eqref{phi-4-model} (see \cite{oh-okamoto} for $d=2$, and \cite{schaeffer} for $d\geq 1$), and the strategy happens to be equally fruitful in wave and Schr{\"o}dinger settings (see e.g. \cite{deya-wave,deya-wave-2} and \cite{deya-schaeffer-thomann}, respectively). 

\smallskip

With these various works in mind, we consider the present investigations about the discretization (and possibly the simulation) of $\<Psi>$ as an important first step toward the discretization of more general singular stochastic PDEs.

\

Before we present our approximation strategy, let us emphasize the following three major difficulties raised by the model.

\smallskip

\noindent
$(i)$ First, it is well-known that the flexibility of fractional noises (i.e., the fact that one can control the overall roughness of the noise through the parameters $H_0,H_1$) comes at a cost: indeed, as soon as $H_i\neq \frac12$, sophisticated fractional kernels must be involved in the analysis of any construction related to the field, which rules out the drastic simplifications offered by It{\^o}-type isometry properties. These additional technicalities can be observed right from the proof of Proposition \ref{sto}, that is right from the interpretation phase of the model, and they will also have a major impact on the subsequent steps.  

\smallskip

\noindent
$(ii)$ In the rough regime \eqref{cond:h-0-h-1-intro}, and as we mentionned it earlier, the solution $\<Psi>$ is no longer a well-defined Gaussian process on $[0,1]\times \R$, and it can only be handled as a distribution in space (see Section \ref{sec:def-solution} for more details). We are thus forced to deal with negative-order Sobolev norms (represented by fractional weights in the Fourier mode) throughout the study, which naturally adds another level of technicality to our computations.  

\smallskip

\noindent
$(iii)$ As it can be seen from \eqref{equa-luxo}, we intend to handle the equation on the whole Euclidean space $\R$, which, as far as we know, is not exactly the most common setting in the approximation literature (bounded domains with suitable boundary conditions appear to be much more frequent, for technical reasons one can easily understand). In fact, our objective in this regard is to make a first possible step toward some of the most recent developments on parabolic models driven by fractional noises, and which are all concerned with equations on the non-compact domain $[0,T]\times \R^d$ (see for instance \cite{chen,chen-deya-ouyang-tindel,deya-heat,hu-huang-nualart-tindel}). Let us also point out that the definition of a space-time fractional noise on the Euclidean space is quite obvious (along \eqref{noise-as-deriv}), whereas there is no consensus about the definition of such an object on a torus.

\smallskip

Of course, the consideration of a non-compact space domain is not costless either. As a particular consequence, our discretization scheme for \eqref{equa-luxo} shall appeal to a finite grid \emph{growing to $\R$} (see details in Section \eqref{subsec:descript-scheme}), which requires a careful control of the approximation process on the growing boundary (see for instance the bound derived from the Galerkin approximation of the heat operator in Proposition \ref{prop:gene-galer-s-t}). Besides, due to the asymptotic behaviour of the fractional sheet, convergence estimates for $\<Psi>$ and its approximation can only be analyzed by means of weighted topologies in space, which eventually echoes in the statement of our main Theorem \ref{main-loose-estimate} (see the involvement of the arbitrary cut-off function $\rho$ in our final bound \eqref{main-loose-estimate}).

\

\emph{Thus, even though the overall linear dynamics of equation \eqref{equa-luxo} may look quite basic at first sight, we think that the above-described features $(i)$-$(ii)$-$(iii)$ turn the analysis and discretization of the problem into a highly non-trivial question, and to the best of our knowledge, there exists no previous approximation study taking those specificities into account.}

\

Let us now briefly describe the successive steps that will punctuate our discretization procedure.

\

\noindent
\textbf{(1) Interpretation of the solution through a smoothening procedure.} The high roughness of the noise $\dot{B}$ under consideration (as induced by assumption \eqref{cond:h-0-h-1-intro}) immediately gives rise to a first basic question before one can think about discretization, namely: how to interpret the solution $\<Psi>$ of \eqref{equa-luxo} in this setting? If $\dot{B}$ were to be more regular, then this solution would be explicitly given by the space-time convolution $\<Psi>=G \ast \dot{B}$,
where $G$ stands for the heat kernel
$$G_t(x):=\frac{1}{\sqrt{4\pi t}} \exp\Big( -\frac{x^2}{4t}\Big) \, \1_{\{t>0\}}  \, , \ t\in [0,T], \ x\in \R \, .$$
Unfortunately, when switching to the rough regime, the meaning of the convolution of $\dot{B}$ with the singular kernel $G$ is no longer clear.

\smallskip

In order to reach such an interpretation, we will rely on a standard regularization procedure, and thus follow the strategy used in most recent pathwise approaches to SPDEs (regularity structures, paracontrolled calculus). In other words, starting from a smooth approximation $B^n$ of $B$, we intend to define $\<Psi>$ as the (potential) limit, in a suitable Sobolev space, of the sequence $G\ast \partial_t\partial_x B^n$ of approximated solutions.

\smallskip

For further reference, let us specify right now our choice for the approximation of $B$ (we will comment further on this choice in Section \ref{sec:def-solution}, see Remark \ref{rk:mollifying} and Remark \ref{rk:biblio}): namely, for every parameter $\ka >0$, we consider the sequence $(B^{\ka,n})_{n\geq 1}$ defined for every $n\geq 1$ as
\begin{equation}\label{b-bar-n}
B^{\ka,n}_t(x)=c_{H_0}c_{H_1} \int_{|\xi|\leq 2^{2\ka n}}\int_{|\eta| \leq 2^{\ka n}} \widehat{W}(d\xi,d\eta) \, \frac{e^{\imath \xi t}-1}{|\xi|^{H_0+\frac12}}\frac{e^{\imath \eta x}-1}{|\eta|^{H_1+\frac12}} \, ,
\end{equation}
where $\widehat{W}$ stands for the Fourier transform of a space-time white noise, and 
$$c_{H_i}=\frac{1}{2}\left( \int_0^{\infty}d\xi \, \frac{1-\cos\xi}{|\xi |^{2H_i+1}} \right)^{-1/2}, \  i=0,1 .$$
It is easy to check that for every fixed $n\geq 1$, the process $B^{\ka,n}$ is (a.s.) smooth on $\R^2$, due to the \enquote{frequency} cut-off $\{|\xi|\leq 2^{2\ka n},\, |\eta| \leq 2^{\ka n}\}$ in the representation \eqref{b-bar-n}. Besides, using standard results about the harmonizable representation of fractional sheets (see e.g. \cite{samo-taqqu}), it can be shown that $B^{\ka,n}$ converges (a.s.) to a fractional sheet $B$ of Hurst indexes $H_0,H_1$.

\smallskip

With this approximation in hand, we will prove (Proposition \ref{sto}) the existence of a threshold value $\al_{d,H}\geq 0$ such that for every $\al>\al_{d,H}$, the sequence $(\<Psi>^{\ka,n})_{n\geq 1}$ of classical solutions to
\begin{equation}\label{equa-luxo-ka-n}
\begin{cases}
\partial_t \<Psi>^{\ka,n}=\Delta \<Psi>^{\ka,n}+\partial_t\partial_x B^{\ka,n}\, , \quad t\in [0,1], \ x\in \R,\\
\<Psi>^{\ka,n}_0=0\, ,
\end{cases}
\end{equation}
converges in the scale $\cac([0,T],\cw^{-\al,p})$ for every $p\geq 1$, where the notation $\cw^{-\al,p}$ refers to the fractional Bessel-potential space in $\R$ (see \eqref{bessel-potential-space}). Along the above considerations, we henceforth define $\<Psi>$ as the limit of this sequence.

\

\noindent
\textbf{(2) Discretization of the noise}. We can now turn to the discretization procedure itself (note indeed that the previous smoothened solution $\<Psi>^{\ka,n}$ is clearly not sufficient in this regard). In fact, our general objective can be loosely summed up as follows: find a way to approximate the solution $\<Psi>$ through a discrete iterative algorithm involving (a finite number of) Gaussian increments.


\smallskip

At a basic level, this challenge somehow corresponds to the search for an extension, in the heat setting, of the discretization methods used for the elementary linear standard differential equation
\begin{equation}\label{eds}
d\Psi_t=\dot{B}_t \, , \ \Psi_0=0 \, ,
\end{equation}
where $B$ is a standard one-parameter fractional Brownian motion. The solution of \eqref{eds} is of course given by the process $B$ itself, but when it comes to discretization, the standard linear interpolation of $\Psi$ can be regarded as the result of a two-step scheme: $(i)$ discretize the noise $\dot{B}$ through increments of $B$
$$\dot{B}^n:=n\sum_{i=0}^{n-1} (B_{t_{i+1}}-B_{t_i})\1_{[t_i,t_{i+1})} \, , \quad t_i:=\frac{i}{n}\, ;$$
$(ii)$ define $\Psi^n$ as the convolution of $\dot{B}^n$ with the Heaviside kernel $\1_{\R_+}$, i.e. set, for $t\in [t_i,t_{i+1})$, 
\begin{equation}\label{defi-psi-n}
\Psi^n_t:=\int_0^t ds \, \dot{B}^n_s=B_{t_i}+n(t-t_i)  (B_{t_{i+1}}-B_{t_i}) \, ,
\end{equation}
which indeed leads us to an aproximation of $\Psi$ based on a Gaussian vector
$$(X_0,\ldots,X_{n-1}):=(B_{t_1}-B_0,B_{t_2}-B_{t_1},\dots,B_{t_n}-B_{t_{n-1}}) \, .$$

\

Let us transpose the above steps in the present heat situation, and more precisely to the equation \eqref{equa-luxo-ka-n} (as a reminiscence of our interpretation of $\<Psi>$ as the limit of $\<Psi>^{\ka,n}$). Accordingly, we first discretize the noise in \eqref{equa-luxo-ka-n} by means of rectangular increments of $B^{\ka,n}$ along the (growing) dyadic grid
\begin{equation}\label{grid-intro}
t_i:=\frac{i}{2^n} \ \  (i=0,\ldots,2^n) , \quad x_j:=\frac{j}{2^n} \ \  (j=-2^{2n},\ldots,2^{2n}) \, ,
\end{equation}
that is we consider the approximation of $\partial_t\partial_x B^{\ka,n}$ on $[0,1]\times \R$ given by
\begin{equation}\label{fir-app}
 \partial_t \partial_x \widetilde{B}^{\ka,n}:=\sum_{i=0}^{2^n-1} \sum_{j=-2^{2n}}^{2^{2n}-1}(2^{2n} \!\sq^n_{i,j}\!\!  B^{\ka,n}) \, \1_{\sq^n_{ij}} \, ,
\end{equation}
where $\1_{\sq^n_{ij}}(s,x):=\1_{[t_i,t_{i+1})}(s) \1_{[x_j,x_{j+1})}(x)$, and for every two-parameter path $b:[0,1]\times \R\to \R$, 
\begin{equation}\label{rectangular-increment-intro}
\sq^n_{i,j} b:=b_{t_{i+1}}(x_{j+1})-b_{t_{i+1}}(x_{j})-b_{t_{i}}(x_{j+1})+b_{t_{i}}(x_{j}) \, .
\end{equation}
Then, following the one-parameter pattern in \eqref{defi-psi-n}, we define the approximation $\widetilde{\<Psi>}^{\ka,n}$ as the solution related to $ \partial_t \partial_x \widetilde{B}^{\ka,n}$, that is as the (well-defined) convolution
\begin{equation}\label{luxotilde-ka-n}
\widetilde{\<Psi>}^{\ka,n}_{t}(x):=(G\ast \partial_t \partial_x \widetilde{B}^{\ka,n})_t(x) \, .
\end{equation}
Just as in \eqref{defi-psi-n}, and for every $(t,x)\in [0,T]\times \R$, the value of  $\widetilde{\<Psi>}^{\ka,n}_{t}(x)$ can thus be expressed as a combination of the values of the Gaussian vector
$$\{\sq^n_{i,j} B^{\ka,n}, \ i=0,\ldots,2^n-1, \, j=-2^{2n},\ldots,2^{2n}-1\} \, .$$
This natural noise-discretization step will be fully justified in Section \ref{sec:noise-discretization}: we will therein prove that, at least if the \enquote{frequency} parameter $\ka$ in \eqref{b-bar-n} is small enough, i.e. if the ratio between the smoothening speed ($2^{\ka n}$) and the discretization speed ($2^n$) is small enough, then the sequence $(\widetilde{\<Psi>}^{\ka,n})_{n\geq 1}$ does converge to the actual solution $\<Psi>$, as $n\to \infty$. 

\

\noindent
\textbf{(3) Space-time discretization of the heat operator}. Let us go back to the interpretation of the process $\widetilde{\<Psi>}^{\ka,n}$ in \eqref{luxotilde-ka-n} as the solution of the heat equation
\begin{equation}\label{equa-luxotilde-ka-n}
\begin{cases}
\partial_t \widetilde{\<Psi>}^{\ka,n}=\Delta \widetilde{\<Psi>}^{\ka,n}+\partial_t \partial_x \widetilde{B}^{\ka,n}\, , \quad t\in [0,1], \ x\in \R,\\
\widetilde{\<Psi>}^{\ka,n}_0=0\, ,
\end{cases}
\end{equation}
and note that, for every fixed $n\geq 1$, $\partial_t \partial_x \widetilde{B}^{\ka,n}$ now stands for a (random) bounded function on $[0,1]\times \R$, as it can immediately be seen from \eqref{fir-app}.

\smallskip

In order to achieve our full-discretization objective, we still need to propose an approximation scheme for the heat dynamics. In fact, when dealing with a well-defined perturbation function (such as $\partial_t \partial_x \widetilde{B}^{\ka,n}$, for fixed $n\geq 1$), space-time discretizations of the heat operator can be derived from standard (deterministic) finite element methods, which ultimately generate basic linear iterative systems (see e.g. \cite{johnson,thomee}).

\smallskip

Our purpose in the third (and final) step of the study will thus be to carefully examine how these deterministic methods can be applied in our setting, and above all how the resulting approximation of \eqref{equa-luxotilde-ka-n} can be controlled in terms of the perturbation $\partial_t \partial_x \widetilde{B}^{\ka,n}$ (seen as an element of $L^\infty([0,1]\times \R)$). To implement this strategy, we will focus on the combination of a - space - Galerkin-type projection and a - time - implicit Euler scheme, a standard choice in the heat-approximation literature (see Section \ref{sec:space-discret} for a complete description).

\smallskip

Of course, the involvement of the $L^\infty$-norm of $\partial_t \partial_x \widetilde{B}^{\ka,n}$ in the corresponding estimates can only come at a price as far as $n$ is concerned (recall that $\partial_t \partial_x \widetilde{B}^{\ka,n}$ is only expected to be uniformly bounded in $n$ as a negative-order distribution). In light of our controls (see Proposition \ref{prop:discreti-s-t}), a possible way to counterbalance this $n$-loss will consist in the application of the Galerkin procedure on a finer grid than the one used to discretize $\widetilde{B}^{\ka,n}$. Let us slightly anticipate the next section and point out that this balancing phenomenon can be easily observed on the description \eqref{defi-psi-black} of our discretization scheme, by comparing the $(2^{-4n},2^{-2n})$ space-time mesh in \eqref{grid}-\eqref{defi-psi-black} with the $(2^{-n},2^{-n})$ discretization mesh used for $B$ in \eqref{involvement-b-ka-n-scheme}-\eqref{rectangular-increm-intro} (see also Remark \ref{rk:grid-b-intro}).

\

\subsection{Main discretization scheme}\label{subsec:descript-scheme}

With the previous description in mind, let us introduce our main discretization scheme for the solution $\<Psi>$. To this end, recall first that for all $\ka >0$ and $n\geq 0$, the notation $B^{\ka,n}$ refers to the smoothened version of $B$ defined by \eqref{b-bar-n}, and at the core of our interpretation of $\<Psi>$ (along Proposition \ref{sto}).

\smallskip

We consider the grid (note the notation with respect to \eqref{grid-intro})
\begin{equation}\label{grid}
t_i=t_i^n:=\frac{i}{2^{4n}}, \, i=0,\ldots,2^{4n}, \quad x_j=x^n_j:=\frac{j}{2^{2n}}, \, j\in \mathbb{Z} \, .
\end{equation}
Given $i=0,\ldots,2^{4n}$, we will denote by $\tilde{i}$ the (only) integer such that $\frac{\tilde{i}}{2^n}\leq t_i< \frac{\tilde{i}+1}{2^n}$. In the same way, given $j\in \mathbb{Z}$, we denote by $\dbtilde{j}$ the (only) integer such that $\frac{\dbtilde{j}}{2^n}\leq x_j< \frac{\dbtilde{j}+1}{2^n}$.

\smallskip

Then we introduce the set of functions $\Phi_j^n: \R \to \R$ ($j\in \mathbb{Z}$) along the formula
\begin{equation}
\Phi_j^n(x):=
\begin{cases}
2^{2n}(x-x_{j-1}) & \text{if}\ x \in [x_{j-1},x_j]\\
2^{2n}(x_{j+1}-x) & \text{if} \ x\in [x_j,x_{j+1}]\\
0& \text{otherwise} \, .
\end{cases}
\end{equation}
Using the above notation, we can describe our approximation process as follows. Namely, for all $i=0,\ldots,2^{4n}$ and $x\in \R$, we set
\begin{equation}\label{defi-psi-ka-n}
\bar{\<Psi>}^{\ka,n}_{t_i}(x):=\sum_{j=-N+1}^{N-1} \bar{\<Psiblack>}^j_{t_i} \Phi^n_j(x) \, ,
\end{equation}
where $N:=2^{3n+1}$ and the points $\bar{\<Psiblack>}^j_{t_i}$ ($i=0,\ldots,2^{4n}$, $j=-N+1,\ldots,N-1$) are given by the iteration procedure:
\begin{equation}\label{defi-psi-black}
\begin{cases}
4 \, \bar{\<Psiblack>}^{j}_{t_{i+1}}-\frac54\,  \bar{\<Psiblack>}^{j+1}_{t_{i+1}}=\bar{\<Psiblack>}^j_{t_i}+\frac14\, \bar{\<Psiblack>}^{j+1}_{t_i}+\frac{3}{2^{2n+1}}\cdot \delta B^{\ka,n}_{ij} &j=-N+1\\
\\
4\, \bar{\<Psiblack>}^{j}_{t_{i+1}}-\frac54\, \big[\bar{\<Psiblack>}^{j+1}_{t_{i+1}}+\bar{\<Psiblack>}^{j-1}_{t_{i+1}}\big]=\bar{\<Psiblack>}^j_{t_i}+\frac14\, \big[\bar{\<Psiblack>}^{j-1}_{t_i}+\bar{\<Psiblack>}^{j+1}_{t_i}\big]+\frac{3}{2^{2n+1}}\cdot \delta B^{\ka,n}_{ij} & j=-N+2,\ldots, N-2 \, ,\\
\\
4\, \bar{\<Psiblack>}^{j}_{t_{i+1}}-\frac54\, \bar{\<Psiblack>}^{j-1}_{t_{i+1}}=\frac14\, \bar{\<Psiblack>}^{j-1}_{t_i}+\bar{\<Psiblack>}^j_{t_i}+\frac{3}{2^{2n+1}}\cdot \delta B^{\ka,n}_{ij} &j=N-1
\end{cases}
\end{equation}
where we have set
\begin{equation}\label{involvement-b-ka-n-scheme}
\delta B^{\ka,n}_{ij}:=\1_{\{x_j>\frac{\dbtilde{j}}{2^n}\}}\sq_{\tilde{i},\dbtilde{j}}^n\! B^{\ka,n} +\1_{\{x_j=\frac{\dbtilde{j}}{2^n}\}}\bigg[\frac12 \sq_{\tilde{i},\dbtilde{j}-1}^n\!\! B^{\ka,n}+\frac12 \sq_{\tilde{i},\dbtilde{j}}^n\! B^{\ka,n} \bigg] \, .
\end{equation}
We recall that the notation $\square^n_{k,\ell}$ has been introduced in \eqref{rectangular-increment-intro} and refers to the rectangular increment
\begin{equation}\label{rectangular-increm-intro}
\sq_{k,\ell}^n\! B^{\ka,n}:=B^{\ka,n}_{\frac{k+1}{2^n},\frac{\ell+1}{2^n}}-B^{\ka,n}_{\frac{k+1}{2^n},\frac{\ell}{2^n}}-B^{\ka,n}_{\frac{k}{2^n},\frac{\ell+1}{2^n}}+B^{\ka,n}_{\frac{k}{2^n},\frac{\ell}{2^n}} \, .
\end{equation}

\

\begin{remark}\label{rk:grid-b-intro}
Observe that, through the $(\tilde{i},\dbtilde{j})$-notation, the above scheme only involves the rectangular increments of $B^{\ka,n}$ over the sub-grid $(\frac{i}{2^n},\frac{j}{2^n})_{i,j}$ of $(t_i,x_j)_{i,j}$. 
\end{remark}

\smallskip

\begin{remark}
The above specific calibration of the scheme (i.e. the choice of the specific grid $(t_i,x_j)$ in \eqref{grid}) is naturally derived from the subsequent theoretical convergence results. Note however that we do not expect this calibration to be optimal. In other words, the convergence property in the forthcoming Theorem \ref{main-theo} certainly remains true for coarser grids $t_i'=\frac{i}{2^{\la n}}$, $x_j'=\frac{j}{2^{\beta n}}$, with $1\leq \la <4$ and $1\leq \beta<2$ (possibly depending on $(H_0,H_1)$). See Proposition \ref{prop:discreti-s-t} and the related Remarks \ref{rk:non-optimal-choice} and \ref{rk:l-2-h} for further details about this choice of calibration.
\end{remark}

\smallskip

\subsection{Main convergence statement}

We are now in a position to present the main theoretical result of the paper, which fully justifies our consideration of the discretization scheme \eqref{defi-psi-ka-n}-\eqref{defi-psi-black} toward $\bar{\<Psi>}^{\ka,n}$.


\begin{theorem}\label{main-theo}
Fix $(H_0,H_1)\in (0,1)^2$ such that $0<2H_0+H_1<1$, and set
\begin{equation}\label{defi-al-0}
\al_0:=1-(2H_0+H_1)>0 \, .
\end{equation}
Then, for every $\al>\al_0$, and for every smooth compactly-supported function $\rho:\R\to [0,1]$, there exist a deterministic constant $\nu=\nu((H_0,H_1),\al)>0$, as well as a random constant $C=C(\rho,(H_0,H_1),\al)>0$, such that for all $0<\ka\leq \frac{\al_0}{5}$ and $n\geq 1$, one has almost surely
\begin{equation}\label{main-loose-estimate}
\sup_{i=0,\ldots,2^{4n}}\big\| \rho \cdot \big\{ \bar{\<Psi>}^{\ka,n}_{t_i}- \<Psi>_{t_i} \big\} \big\|_{\ch^{-\al}(\R)} \leq C\,  2^{-n\nu} \, .
\end{equation}
\end{theorem}

\

The above result stems from the combination of the estimates obtained in Proposition \ref{sto} (noise smoothening), Proposition \ref{prop:discreti-noise} (noise discretization) and Proposition \ref{prop:discreti-s-t} (space-time discretization of the heat operator). The retriction on $\ka$, namely $0<\ka\leq \frac{\al_0}{5}$, can be seen as a consequence of some balance strategy between the noise-smoothening step ($B\mapsto B^{\ka,n}$) and the noise-discretization step ($\partial_t\partial_x B^{\ka,n} \mapsto  \partial_t \partial_x \widetilde{B}^{\ka,n} $), as it will be detailed in Section \ref{sec:noise-discretization}.

\

Let us complete the statement of Theorem \ref{main-theo} with a few additional remarks. 

\begin{remark}
The involvement of a (arbitrary) space cut-off function $\rho$ in \eqref{main-loose-estimate} must be understood as a way to express \enquote{local convergence} in the space $\ch^{-\al}(\R)$ of negative-order distributions. Observe indeed that if $\al=0$, that is if one considers $\ch^{-\al}(\R)=L^2(\R)$, then by taking $\rho:\R\to [0,1]$ equal to $1$ on any given compact set $K\subset \R$, one has of course $\|f\|_{L^2(K)}\leq \|\rho \cdot f\|_{L^2(\R)}$ for any $f\in L^2_{loc}(\R)$, and accordingly estimates such as \eqref{main-loose-estimate} would entail $L^2$-convergence on compact sets. Note also that the convergence results and its proof would certainly remain true for a more general class of weights $\rho$ on $\R$, with support possibly non compact (e.g., for a Gaussian weight $\rho(x):=e^{-x^2}$).
\end{remark}

\smallskip

\begin{remark}\label{rk:ka-nu}
Some possible explicit value for $\nu$ can be derived from the statements of Proposition \ref{sto}, Proposition \ref{prop:discreti-noise} and Proposition \ref{prop:discreti-s-t}. For instance, with the notation $\al_0$ introduced in \eqref{defi-al-0}, we can take
$$\nu:=\min\bigg( \frac{2\al_0 H_0}{5},\frac{\al_0 H_1 }{5},\frac{\al_0 (\al-\al_0)}{5},1-\al_0,\frac{\al_0}{5}\bigg) \, .$$
In any case, we do not expect the subsequent analysis to provide us with an \enquote{optimal} speed of convergence for the proposed scheme \eqref{defi-psi-black} (say for fixed $\al >\al_0$ in the left-hand side of \eqref{main-loose-estimate}), due to our consideration of a deterministic strategy and a $L^2(\R)$-norm in the space-time heat discretization step (see Proposition \ref{prop:discreti-s-t} and the related Remark \ref{rk:l-2-h} for further details).
\end{remark}


\

The rest of the paper is organized as follows. In Section \ref{sec:def-solution}, we examine the noise-smoothening procedure toward a proper definition of $\<Psi>$. Then, in Section \ref{sec:noise-discretization}, we initiate the discretization scheme through the transition from $\partial_t\partial_x B^{\ka,n}$ to $\partial_t \partial_x \widetilde{B}^{\ka,n} $. The theoretical analysis is completed in Section \ref{sec:space-discret}, with the space-time discretization of the heat operator. Finally, we have provided, in Section \ref{sec:numerical}, a few results and comments about the (partial) simulation of the algorithm \eqref{defi-psi-ka-n}-\eqref{defi-psi-black}.

\

From a technical point of view, the subsequent analysis relies on the combination of fractional calculus with (relatively) standard discretization techniques. For pedagogical purposes, we have endeavored to provide many details at every step of these investigations, which hopefully can make the study accessible to a large audience.

\

\section{Definition of the solution}\label{sec:def-solution}

For the sake of completeness, let us first recall the definition of the fractional sheet, that is the field at the core of this study.

\begin{definition}\label{defi:fractional-sheet}
On a complete probability space $(\Omega,\mathfrak{F},\mathbb{P})$, we call a fractional sheet of Hurst indexes $H_0,H_1\in (0,1)^2$ on $[0,1]\times \R$ any centered Gaussian field $B:\Omega \times ([0,1]\times \R)\to \R$ with covariance function given by the formula: for all $s,t\in [0,1]$ and $x,y\in \R$,
$$\mathbb{E}\big[ B_s(x) B_t(y)\big]=R_{H_0}(s,t) R_{H_1}(x,y) \ , \ \text{where} \ R_H(a,b):=\frac12 \big\{|a|^{2H}+|b|^{2H}- |a-b|^{2H}\big\} \ .$$
\end{definition}

When $H_0=H_1=\frac12$, the above definition of the fractional sheet is known to coincide with the one of a standard Brownian field. In any case, that is for every $(H_0,H_1)\in (0,1)^2$, it can be shown that $B$ is not a differentiable field, and accordingly the definition of the noise $\dot{B}$ in \eqref{noise-as-deriv} can only be understood as a general distribution. Owing to this lack of regularity, the interpretation of $\<Psi>$ as the convolution of $\dot{B}$ with the (singular) heat kernel $G$ is clearly not a trivial issue, and we propose to address this question through a regularization procedure. Thus, for all fixed $\ka >0$ and $n\geq 0$, we consider the smooth approximation $B^{\ka,n}$ of $B$ provided by \eqref{b-bar-n}. Using the so-called harmonizable representation of $B$, i.e.
\begin{equation}\label{harmoni-rep}
B_t(x)=c_{H_0} c_{H_1} \int_{\R}\int_{\R} \widehat{W}(d\xi,d\eta) \, \frac{e^{\imath \xi t}-1}{|\xi|^{H_0+\frac12}}\frac{e^{\imath \eta x}-1}{|\eta|^{H_1+\frac12}} \, ,
\end{equation}
it can indeed be shown that for every $\ka>0$, one has almost surely
$$B^{\ka,n} \stackrel{n\to\infty}{\longrightarrow} B\quad \text{in}\ \cac([0,1]\times \R;\R)\, .$$

\smallskip

Our objective now is to study the convergence of the sequence of (classical) solutions associated with $(B^{\ka,n})_{n\geq 1}$, that is the sequence of well-defined processes
\begin{equation}\label{def:psi-ka-n}
\<Psi>^{\ka,n}_{t}(x):=(G\ast \partial_t\partial_x B^{\ka,n})_t(x) \, .
\end{equation}
To do so, we will appeal to the following scale of fractional Sobolev spaces.

\smallskip

\noindent
\textbf{Notation.} For all $s\in \R$ and $1\leq p<\infty$, let the notation $\mathcal{W}^{s,p}$ refer the Bessel-potential
\begin{equation}\label{bessel-potential-space}
\mathcal{W}^{s,p}=\mathcal{W}^{s,p}(\mathbb{R}):=\left\{f\in \mathcal{S}'(\mathbb{R}):\ \|f |\mathcal{W}^{s,p}(\mathbb{R}^d)\|=\|\mathcal{F}^{-1}(\{1+|.|^2\}^{\frac{s}{2}}\mathcal{F}f) |L^p(\mathbb{R})\|  <\infty \right\} \, .
\end{equation}
Also we will consider the spaces 
\begin{equation}
\ch^{s}:=\mathcal{W}^{s,2},\quad \text{for every}\ s\in \R\, .
\end{equation}

\

Using the above notation, the result at the basis of our interpretation of \eqref{equa-luxo} can be stated as follows.

\begin{proposition}\label{sto}
Fix $(H_0,H_1)\in (0,1)^2$ such that $0<2H_0+H_1<1$, and set
\begin{equation}\label{assump-gene-al}
\alpha_0:= 1-\big(2H_0+H_1\big)>0\, .
\end{equation} 
Then, for every $\ka > 0$ and for every cut-off function $\rho\in \cac_c^\infty(\R)$ (i.e., smooth and compactly-supported), the sequence $(\rho\cdot \<Psi>^{\ka,n})_{n  \geq 0}$ converges in the space $L^{2p}(\Omega;\mathcal{C}([0,1]; \mathcal{W}^{-\al,2p}(\mathbb{R})))$, for all $\al>\al_0$ and $p\geq 1$. Moreover, the limit, that we denote by $\rho\cdot\<Psi>$, does not depend on $\ka$. 

\smallskip

Finally, for all $\al>\al_0$, $\ka >0$, $p\geq 1$, $n\geq 1$ and $\varsigma>0$ such that
\begin{equation}\label{cond-varsigma}
0<\varsigma <\min\big( 2H_0,H_1,\al-\al_0\big),
\end{equation}
one has almost surely
\begin{equation}\label{speed-converg-def}
\sup_{t\in [0,1]}\big\| \rho \cdot\<Psi>^{\ka,n}_t-\rho \cdot\<Psi>_t\big\|_{\ch^{-\al}(\mathbb{R})} \lesssim 2^{-\varsigma\ka n } ,
\end{equation}
where the proportional constant does not depend on $n$.
\end{proposition}

The proof of Proposition \ref{sto} will be developed in Section \ref{subsec:preliminary} and Section\ref{subsec:proof-sto-1} below.

\

Observe that this result actually gives birth to a family of processes 
$$\big\{\rho \cdot \<Psi> \in L^{2p}(\Omega;\mathcal{C}([0,1]; \mathcal{W}^{-\al,2p}(\mathbb{R}))), \ \rho \in \cac_c^\infty(\R)\big\}\, . $$
For the sake of completeness, we can then patch together those local solutions into a single distribution $\<Psi>$. The details of this canonical procedure can be found in \cite[Section 2.5]{deya-schaeffer-thomann}, and we will only sum up the result in the following Proposition-Definition (to simplify the presentation, we fix $p\geq 2$ and $\al > \al_0$, and we denote by $\cf([0,1];\cd'(\R))$ the set of distributional-valued functions on $[0,1]$).

\begin{propdef}\label{prop:defi-psi-glo}
Let $\cp$ stand for the set of sequences $\si=(\si_k)_{k\geq 1}$, where, for each $k\geq 1$, $\si_k:\R \to \R$ is a smooth function such that
\begin{equation*}
    \si_k(x)=  \left\{
    \begin{array}{ll}
		1 & \mbox{if}\  |x|\leq k \, ,\\
        0 & \mbox{if} \ |x|\geq k+1\, .  
    \end{array}
\right.
\end{equation*}
Then, for every $\si\in \cp$, there exists a subspace $\Omega^{(\si)}\subset \Omega$ of full measure $1$ and an element
$$\<Psi>^{(\si)} :\Omega^{(\si)} \to \cf([0,1];\cd'(\R)) $$
such that the following assertions hold true:

\smallskip

\noindent
$(i)$ For any (space) cut-off function $\rho\in \cac_c^\infty(\R)$ and for any $\ka >0$, one has, on $\Omega^{(\si)}$,
$$\rho \cdot\<Psi>^{\ka,n} \underset{n \rightarrow \infty}\rightarrow \rho\cdot \<Psi>^{(\si)} \quad \text{in}\ \ \mathcal{C}([0,1]; \mathcal{W}^{-\al,p}(\mathbb{R}^d)) \, .$$

\smallskip

\noindent
$(ii)$ If $\si,\ga \in \cp$, then one has
$$
\<Psi>^{(\si)}=\<Psi>^{(\ga)} \quad \text{on} \ \ \Omega^{(\si)}\cap \Omega^{(\ga)} \, .
$$

\smallskip

Owing to these two properties, we define $\<Psi>$ as $\<Psi>:=\<Psi>^{(\si)}$ for some fixed (arbitrary) sequence $\si\in \cp$, and we call this random element in $\cf([0,1];\cd'(\R))$ the \emph{mild solution of equation \eqref{equa-luxo}}.
\end{propdef}

Thanks to the above result, we are now endowed with a \emph{globally} defined solution $\<Psi>$, which locally coincides with the limits exhibited in Proposition \ref{sto}. This being said, in our subsequent investigations, we will only focus on \emph{local} convergence to $\<Psi>$, as it can be seen from our main estimate \eqref{main-loose-estimate}.

\

Before we turn to the proof of Proposition \ref{sto}, let us complete the statement with two remarks.

\begin{remark}\label{rk:mollifying}
The consideration of the \enquote{Fourier-type} approximation $B^{\ka,n}$ of $B$ is quite natural in our fractional Sobolev setting, and it will indeed readily provide us with a convenient covariance formula for the process $\<Psi>^{\ka,n}$ (see Proposition \ref{prop:cova-luxo-n}).  

\smallskip

Another usual choice for the approximation of $B$ is the one derived from a mollifying procedure, that is one takes $B^{\vp,\ka,n}:=\vp_{\ka,n} \ast B$, where $\vp_{\ka,n}(s,x):=2^{3\ka n}\vp(2^{2\ka n}s,2^{\ka n}x)$, for some mollifier $\vp:\R^2\to\R$. In fact, our approximation $B^{\ka,n}$ can somehow be regarded as a particular case of this general mollifying procedure. Indeed, starting from the representation \eqref{harmoni-rep} of $B$, one can write (at least formally)
\begin{align*}
(\partial_t\partial_x B^{\vp,\ka,n})(t,x)&=(\vp_{\ka,n} \ast \partial_t\partial_x B)(t,x)\\
&=c_{H_0}c_{H_1} \iint ds dy \, \vp_{\ka,n}(s,y)\iint \widehat{W}(d\xi,d\eta) \, (-\xi \eta)\frac{e^{\imath \xi (t-s)}}{|\xi|^{H_0+\frac12}}\frac{e^{\imath \eta (x-y)}}{|\eta|^{H_1+\frac12}} \\
&=c_{H_0}c_{H_1}  \iint \widehat{W}(d\xi,d\eta) \, \widehat{\vp_{\ka,n}}(\xi,\eta) (-\xi \eta)\frac{e^{\imath \xi t}}{|\xi|^{H_0+\frac12}}\frac{e^{\imath \eta x}}{|\eta|^{H_1+\frac12}}\\
&=c_{H_0}c_{H_1}  \iint \widehat{W}(d\xi,d\eta) \, \widehat{\vp}(2^{-2\ka n}\xi,2^{-\ka n}\eta) (-\xi \eta)\frac{e^{\imath \xi t}}{|\xi|^{H_0+\frac12}}\frac{e^{\imath \eta x}}{|\eta|^{H_1+\frac12}}\, ,
\end{align*}
and thus, picking $\vp$ such that $\widehat{\vp}(\xi,\eta)=\1_{\{|\xi|\leq 1\}} \1_{\{|\eta|\leq 1\}}$, one recovers the approximation $\partial_t\partial_x B^{\ka,n}$ of the noise. We think that, at the price of a few technical modifications, it is certainly possible to extend the whole subsequent analysis to a more general class of mollifying approximations $B^{\vp,\ka,n}$.
\end{remark}

\begin{remark}\label{rk:biblio}
The above construction procedure of the solution $\<Psi>$, based on the specific approximation $B^{\ka,n}$ of $B$, has already been implemented in \cite{deya-wave,deya-wave-2} for the fractional wave equation, and in \cite{deya-schaeffer-thomann} for the Schr{\"o}dinger fractional equation. To be more specific, in the three references \cite{deya-wave,deya-wave-2,deya-schaeffer-thomann}, the authors' analysis only relies on the consideration of $B^{1,n}$, i.e. $B^{\ka,n}$ with $\ka=1$. Letting $\ka$ vary (in Section \ref{sec:noise-discretization}) will here give us the possibility to maintain a certain balance within the two-step transformation process of the noise (see Remark \ref{rk:ka}). 
\end{remark}

\smallskip

\subsection{Preliminary considerations}\label{subsec:preliminary}

Starting from the expression \eqref{b-bar-n} of $B^{\ka,n}$, and setting $c_H:=c_{H_0}c_{H_1}$, we can first expand the solution $\<Psi>^{\ka,n}=G\ast \dot{B}^{\ka,n}$ as
\small
\begin{align*}
&\<Psi>^{\ka,n}_t(x)=\int_{0}^{t}ds\int_{\mathbb{R}}dy\, G_{t-s}(x-y)(\partial_t\partial_x B^{\ka,n})_s(y)\\ 
&=-c_H\int_{0}^{t}ds\int_{\mathbb{R}}dy\, \int_{|\xi|\leq 2^{2\ka n}}\int_{|\eta|\leq 2^{\ka n}}G_{t-s}(x-y)\frac{\xi}{|\xi|^{H_0+\frac{1}{2}}}\frac{\eta}{|\eta|^{H_1+\frac{1}{2}}}e^{\imath\xi s}e^{\imath \eta y }\, \widehat{W}(d\xi,d\eta)\\
&= -c_H\int_{|\xi|\leq 2^{2\ka n}}\int_{|\eta|\leq 2^{\ka n}}\frac{\xi}{|\xi|^{H_0+\frac{1}{2}}}\frac{\eta}{|\eta|^{H_1+\frac{1}{2}}}e^{\imath\eta x}\left[\int_{0}^{t}ds \, e^{\imath\xi s}\left(\int_{\mathbb{R}}dy\, G_{t-s}(x-y)e^{-\imath \eta (x-y )}\right)\right]\widehat{W}(d\xi,d\eta)\\
&= -c_H\int_{|\xi|\leq 2^{2\ka n}}\int_{|\eta|\leq 2^{\ka n}}\frac{\xi}{|\xi|^{H_0+\frac{1}{2}}}\frac{\eta}{|\eta|^{H_1+\frac{1}{2}}}e^{\imath\eta x }\left[\int_{0}^{t}ds \, e^{\imath\xi(t-s)}\left(\int_{\mathbb{R}}dy\, G_{s}(y)e^{-\imath \eta y }\right)\right]\widehat{W}(d\xi,d\eta)\, ,
\end{align*}
\normalsize
and so we obtain the representation
\begin{equation}\label{formula-heuris}
\<Psi>^{\ka,n}_t(x)=-c_H\int_{|\xi|\leq 2^{2\ka n}}\int_{|\eta|\leq 2^{\ka n}}\frac{\xi}{|\xi|^{H_0+\frac{1}{2}}}\frac{\eta}{|\eta|^{H_1+\frac{1}{2}}}e^{\imath  \eta x}\gamma_t(\xi,|\eta|)\, \widehat{W}(d\xi,d\eta)\, ,
\end{equation}
where for all $t\geq 0$, $\xi \in \mathbb{R}$ and $r>0$, the quantity $\gamma_t(\xi,r)$ is defined by
\begin{equation}\label{defi-ga-t}
\gamma_t(\xi,r):=\int_{0}^{t}ds \, e^{\imath\xi(t-s)}\widehat{G_{s}}(r)=e^{\imath\xi t}\int_{0}^{t}e^{- s \{r^2+\imath \xi\} }ds \, .
\end{equation}

The following proposition is an immediate consequence of \eqref{formula-heuris}.
\begin{proposition}\label{prop:cova-luxo-n}
The covariance of the centered complex Gaussian process 
$$\big\{\<Psi>^{\ka,n}_s(x), \, \ka >0, \, n\geq 1, \, s\geq 0, \, x\in \R \big\}$$
is given by the formulas: for all $\ka,\ka'>0$, $n,m\geq 1$, $s, t\geq 0$ and $x, y \in \mathbb{R}^d$,
\begin{equation}\label{cova-luxo}
\mathbb{E}\Big[\<Psi>^{\ka,n}_s(x)\overline{\<Psi>^{\ka',m}_t(y)}\Big]=c_H^2\int_{(\xi,\eta)\in D^{\ka,n} \cap D^{\ka',m}}\frac{d\xi}{|\xi|^{2H_0-1}}\frac{d\eta}{|\eta|^{2H_1-1}}\gamma_{s}(\xi,|\eta|)\overline{\gamma_{t}(\xi,|\eta|)}e^{\imath \eta (x-y)} \, ,
\end{equation}
\begin{equation}\label{cova-luxo-2}
\mathbb{E}\Big[\<Psi>^{\ka,n}_s(x)\<Psi>^{\ka',m}_t(y)\Big]=-c_H^2\int_{(\xi,\eta)\in D^{\ka,n} \cap D^{\ka',m}}\frac{d\xi}{|\xi|^{2H_0-1}}\frac{d\eta}{|\eta|^{2H_1-1}}\gamma_{s}(\xi,|\eta|)\gamma_{t}(-\xi,|\eta|)e^{\imath \eta (x-y)} \, ,
\end{equation}
where for all $\ka >0$ and $n\geq 1$, we set
$$D^{\ka,n}:=\{ (\xi,\eta)\in \R^2: \, |\xi|\leq 2^{2\ka n} , \, |\eta|\leq 2^{\ka n}\}\, .$$
\end{proposition}

\smallskip

For all $0\leq s \leq t \leq 1$, $\xi \in \mathbb{R}$ and $r>0$, we set $\gamma_{s,t}(\xi,r):=\gamma_t(\xi,r)-\gamma_s(\xi,r)$. The following elementary bound on $\gamma_{s,t}$ will turn out to be the key estimate in the proof of Proposition \ref{sto}.

\smallskip

\begin{lemma}\label{lem:tec}
Fix $0<H<1$. Then for all $\eta\in \R$, $\varepsilon\in (0,H)$ and $0\leq s \leq t \leq 1$, one has
\begin{equation}\label{estim-int-ga-xi}
\int_{\R}d\xi\, \frac{|\ga_{s,t}(\xi,|\eta|)|^2}{|\xi|^{2H-1}}\lesssim \frac{|t-s|^{\varepsilon}}{1+|\eta|^{4(H-\varepsilon)}} \, .
\end{equation}
\end{lemma}

\begin{proof}
By the definition \eqref{defi-ga-t} of $\ga_{t}(\xi,|\eta|)$, one has
$$\ga_{s,t}(\xi,|\eta|)=\big\{e^{\imath \xi t}- e^{\imath \xi s}\big\}\int_0^t dr \, e^{-r( |\eta|^2+\imath \xi)}+e^{\imath \xi s} \int_s^t dr \, e^{-r( |\eta|^2+\imath \xi)} \, ,$$
from which we immediately deduce, for all $\varepsilon_1,\varepsilon_2,\la\in [0,1]$,
\begin{equation}\label{estim-ga-basi}
\big|\ga_{s,t}(\xi,|\eta|)\big| \lesssim \frac{|t-s|^{\varepsilon_1}|\xi|^{\varepsilon_1}}{||\eta|^2+\imath \xi|^{\la}}+\frac{|t-s|^{\varepsilon_2}}{||\eta|^2+\imath \xi|^{1-\varepsilon_2}} \, .
\end{equation}
Based on this estimate, we have on the one hand, for any $\varepsilon\in (0,H)$,
\begin{equation}\label{estim-1-lem-ga}
\int_{\R}d\xi\, \frac{|\ga_{s,t}(\xi,|\eta|)|^2}{|\xi|^{2H-1}} \lesssim |t-s|^{2\varepsilon} \bigg[ \int_{|\xi|\leq 1} \frac{d\xi}{|\xi|^{2H-1}}+\int_{|\xi|\geq 1} \frac{d\xi}{|\xi|^{1+2(H-\varepsilon)}}\bigg] \lesssim |t-s|^{2\varepsilon} \, .
\end{equation}
On the other hand, thanks to \eqref{estim-ga-basi}, one has for all $|\eta|\geq 1$ and $\varepsilon\in (0,H/2)$,
\begin{align}
&\int_{\R}d\xi\, \frac{|\ga_{s,t}(\xi,|\eta|)|^2}{|\xi|^{2H-1}}\nonumber\\
&\lesssim |t-s|^{2\varepsilon} \bigg[ \frac{1}{|\eta|^{4-4\varepsilon}} \int_{|\xi|\leq 1} \frac{d\xi}{|\xi|^{2H-1}}+\frac{1}{|\eta|^{4(H-2\varepsilon)}} \int_{|\xi|\geq 1} \frac{d\xi}{|\xi|^{2H-1}|\xi|^{2(1-H+\varepsilon)}} \bigg] \lesssim \frac{|t-s|^{2\varepsilon}}{|\eta|^{4(H-2\varepsilon)}} \, .\label{estim-2-lem-ga}
\end{align}

\

Combining \eqref{estim-1-lem-ga} and \eqref{estim-2-lem-ga} clearly yields \eqref{estim-int-ga-xi}.
\end{proof}

The following - independent - technical lemma, borrowed from \cite[Lemma 2.6]{deya-schaeffer-thomann}, will also prove useful in the estimates of the next section.

\begin{lemma}\label{lem:handle-chi-correc}
Let $\rho:\R \to \R$ be a test function and fix $\si\in \R$. Then, for every $p\geq 1$ and for all $\eta_1,\ldots,\eta_p\in \R^d$, it holds that
$$\bigg| \int_{\R} dx \, \prod_{i=1}^p \int_{\R^2} \frac{d\la_i d\lati_i}{\{1+|\la_i|^2\}^{\frac{\si}{2}}\{1+|\lati_i|^2\}^{\frac{\si}{2}}} e^{\imath \langle x,\la_i-\lati_i\rangle} \widehat{\chi}(\la_i-\eta_i)\overline{\widehat{\chi}(\lati_i-\eta_i)}  \bigg| \lesssim \prod_{i=1}^p \frac{1}{\{1+|\eta_i|^2\}^\si} \, ,$$
where the proportional constant only depends on $\chi$ and $\si$. 
\end{lemma}

\smallskip

\begin{remark}
The above preliminary material, as well as the subsequent proof, can be seen as the \enquote{heat} counterpart of the analysis carried out in \cite[Section 2]{deya-wave} for the wave model, and in \cite[Section 2.1]{deya-schaeffer-thomann} for the Schr{\"o}dinger case.
\end{remark}

\smallskip

\subsection{Proof of Proposition~\ref{sto}}\label{subsec:proof-sto-1}
Following the statement of the proposition, we fix $\al>\al_0$, where $\al_0$ is the quantity defined by \eqref{assump-gene-al}.

\

\noindent
\textit{Step 1: A moment estimate.} We show that for all $p\geq 1$, $1\leq n \leq m$, $0<\ka \leq \ka'$, $0\leq s \leq t \leq 1$ and $\varsigma>0$ satisfying \eqref{cond-varsigma}, we can find $\varepsilon>0$ small enough such that
\begin{equation}\label{bou-psi-step-1}
 \int_{\mathbb{R}}dx\, \mathbb{E}\bigg[\Big|\mathcal{F}^{-1}\Big(\{1+|.|^2\}^{-\frac{\alpha}{2}}\mathcal{F}\big(\rho\cdot \big[\<Psi>^{\ka',m}_{s,t}-\<Psi>^{\ka,n}_{s,t}\big]\big)\Big)(x)\Big|^{2p}\bigg] \lesssim 2^{-2n\ka\varsigma p}|t-s|^{2\varepsilon p}\, ,
\end{equation}
where the proportional constant only depends on $p, \alpha$ and $\rho$.

\smallskip

One can first notice that the random variable under consideration is clearly Gaussian, and so, for every $p\geq 1$, one has
\begin{align}
&\mathbb{E}\bigg[\Big|\mathcal{F}^{-1}\Big(\{1+|.|^2\}^{-\frac{\alpha}{2}}\mathcal{F}\big(\rho\cdot \big[\<Psi>^{\ka',m}_{s,t}-\<Psi>^{\ka,n}_{s,t}\big]\big)\Big)(x)\Big|^{2p}\bigg]\leq c_p\,  \mathbb{E}\bigg[\Big|\mathcal{F}^{-1}\Big(\{1+|.|^2\}^{-\frac{\alpha}{2}}\mathcal{F}\big(\rho\cdot \big[\<Psi>^{\ka',m}_{s,t}-\<Psi>^{\ka,n}_{s,t}\big]\big)\Big)(x)\Big|^{2}\bigg]^p\, ,\label{appl-hyper}
\end{align}
where the constant $c_p$ only depends on $p$. Let us then write 
\begin{align*}
&\mathcal{F}^{-1}\Big(\{1+|.|^2\}^{-\frac{\alpha}{2}}\mathcal{F}\big(\rho\cdot \big[\<Psi>^{\ka',m}_{s,t}-\<Psi>^{\ka,n}_{s,t}\big]\big)\Big)(x)\\
&=\int_{\mathbb{R}}d\lambda \, e^{\imath  x \lambda } \{1+|\lambda|^2\}^{-\frac{\alpha}{2}}\mathcal{F}\big(\rho\cdot \big[\<Psi>^{\ka',m}_{s,t}-\<Psi>^{\ka,n}_{s,t}\big]\big)(\lambda)\\
&=\int_{\mathbb{R}}d\la \, \{1+|\lambda|^2\}^{-\frac{\alpha}{2}}e^{\imath x \lambda }\left(\int_{\mathbb{R}}d\be\, \widehat{\rho}(\la-\be) \cf\big(\big[\<Psi>^{\ka',m}_{s,t}-\<Psi>^{\ka,n}_{s,t}\big]\big)(\be)\right)\, ,
\end{align*}
and hence
\begin{align}
&\mathbb{E}\bigg[\Big|\mathcal{F}^{-1}\Big(\{1+|.|^2\}^{-\frac{\alpha}{2}}\mathcal{F}\big(\rho\cdot \big[\<Psi>^{\ka',m}_{s,t}-\<Psi>^{\ka,n}_{s,t}\big]\big)\Big)(x)\Big|^{2}\bigg]\nonumber\\
&=\frac{1}{(2\pi)^{2}}\int_{\R^2} \frac{d\la d\lati}{\{1+|\la|^2\}^{\frac{\al}{2}} \{1+|\lati|^2\}^{\frac{\al}{2}}}e^{\imath  x(\la-\lati)} \int_{\R^2} d\be d\betati \, \widehat{\rho}(\la-\be) \overline{\widehat{\rho}(\lati-\betati)} \mathcal{Q}^{\ka,\ka'}_{n,m;s,t}(\be,\betati) \, ,\label{bound-correc-q}
\end{align}
where we have set
$$\mathcal{Q}^{\ka,\ka'}_{n,m;s,t}(\be,\betati):=\mathbb{E}\Big[\cf\big(\big[\<Psi>^{\ka',m}_{s,t}-\<Psi>^{\ka,n}_{s,t}\big]\big)(\be) \overline{\cf\big(\big[\<Psi>^{\ka',m}_{s,t}-\<Psi>^{\ka,n}_{s,t}\big]\big)(\betati)}\Big] \, .$$
Based on the covariance formula \eqref{cova-luxo}, one has now 
\begin{align*}
&\mathbb{E}\Big[\big[\<Psi>^{\ka',m}_{s,t}(y)-\<Psi>^{\ka,n}_{s,t}(y)\big]\overline{\big[\<Psi>^{\ka',m}_{s,t}(\tilde{y})-\<Psi>^{\ka,n}_{s,t}(\tilde{y})\big]}\Big]\nonumber\\
&= c_H^2\int_{(\xi,\eta)\in D^{\ka',m} \backslash D^{\ka,n}}\frac{d\xi}{|\xi|^{2H_0-1}}\frac{d\eta}{|\eta|^{2H_1-1}}|\gamma_{s,t}(\xi,|\eta|)|^2e^{\imath \eta y }e^{-\imath \eta \tilde{y}} \, ,
\end{align*}
which allows us to recast the above quantity into
\begin{equation}\label{expr-q}
\mathcal{Q}^{\ka,\ka'}_{n,m;s,t}(\beta,\betati)=c_H^2\,  \int_{(\xi,\eta)\in D^{\ka',m} \backslash D^{\ka,n}}\frac{d\xi}{|\xi|^{2H_0-1}}\frac{d\eta}{|\eta|^{2H_1-1}}|\gamma_{s,t}(\xi,|\eta|)|^2 \delta_{\beta=\eta} \delta_{\betati=\eta} \, .
\end{equation}
Thanks to \eqref{appl-hyper}, \eqref{bound-correc-q} and \eqref{expr-q}, we get that
\begin{align*}
& \int_{\mathbb{R}}\mathbb{E}\bigg[\Big|\mathcal{F}^{-1}\Big(\{1+|.|^2\}^{-\frac{\alpha}{2}}\mathcal{F}\big(\rho\cdot \big[\<Psi>^{\ka',m}_{s,t}-\<Psi>^{\ka,n}_{s,t}\big]\big)\Big)(x)\Big|^{2p}\bigg]\, dx\\
&\lesssim \int_{\R} dx \, \bigg(\int_{(\xi,\eta)\in D^{\ka',m} \backslash D^{\ka,n}}\frac{d\xi}{|\xi|^{2H_0-1}}\frac{d\eta}{|\eta|^{2H_1-1}}|\gamma_{s,t}(\xi,|\eta|)|^2\\
&\hspace{3cm} \int_{\R^2} \frac{d\la d\lati}{\{1+|\la|^2\}^{\frac{\al}{2}} \{1+|\lati|^2\}^{\frac{\al}{2}}}e^{\imath  x(\la-\lati)}  \, \widehat{\rho}(\la-\eta) \overline{\widehat{\rho}(\lati-\eta)} \bigg)^p\\
&\lesssim  \left(\int_{(\xi,\eta)\in D^{\ka',m} \backslash D^{\ka,n}}\frac{d\xi}{|\xi|^{2H_0-1}}\frac{d\eta}{|\eta|^{2H_1-1}}\{1+|\eta|^2\}^{-\alpha}|\gamma_{s,t}(\xi,|\eta|)|^2\, d\xi d\eta\right)^p \, ,
\end{align*}
where the last inequality is a consequence of Lemma \ref{lem:handle-chi-correc}.

\smallskip

Observe now that $\1_{D^{\ka',m} \backslash D^{\ka,n}}\leq \1_{\{(\xi,\eta)\in \R^2: \, |\xi|\geq 2^{2\ka n}\}}+\1_{\{(\xi,\eta)\in \R^2: \, |\eta|\geq 2^{\ka n}\}}$, and thus
\begin{align}
& \left(\int_{(\xi,\eta)\in D^{\ka',m} \backslash D^{\ka,n}}\frac{d\xi}{|\xi|^{2H_0-1}}\frac{d\eta}{|\eta|^{2H_1-1}}\{1+|\eta|^2\}^{-\alpha}|\gamma_{s,t}(\xi,|\eta|)|^2\, d\xi d\eta\right)^p  \nonumber\\
&\lesssim  \left(\int_{|\xi|\geq 2^{2\ka n}}\frac{d\xi}{|\xi|^{2H_0-1}}\int_{\R}\frac{d\eta}{|\eta|^{2H_1-1}}\{1+|\eta|^2\}^{-\alpha}|\gamma_{s,t}(\xi,|\eta|)|^2\, d\xi d\eta\right)^p  \nonumber\\
&\hspace{1cm}+ \left(\int_{\R}\frac{d\xi}{|\xi|^{2H_0-1}}\int_{|\eta|\geq 2^{\ka n}}\frac{d\eta}{|\eta|^{2H_1-1}}\{1+|\eta|^2\}^{-\alpha}|\gamma_{s,t}(\xi,|\eta|)|^2\, d\xi d\eta\right)^p  \nonumber\\
&=:(\mathbb{I}^{\ka,n}(s,t))^p+(\mathbb{II}^{\ka,n}(s,t))^p \, . \label{decompos-i-ii}
\end{align}
Let us estimate the quantity $\mathbb{I}^{\ka,n}(s,t)$ first. To do so, pick $\varsigma >0$ satisfying \eqref{cond-varsigma}, which yields
\begin{align*}
\mathbb{I}^{\ka,n}(s,t)&\leq 2^{-2n\ka\varsigma}  \int_{\mathbb{R}^2}\, \frac{d\xi}{|\xi|^{2(H_0-\frac{\varsigma}{2})-1}}\frac{d\eta}{|\eta|^{2H_1-1}}\{1+|\eta|^2\}^{-\alpha}|\gamma_{s,t}(\xi,|\eta|)|^2\\
&\lesssim 2^{-2n\ka\varsigma} \int_{0}^{\infty} \frac{dr}{r^{2H_1-1}\{1+r^2\}^{\alpha}}\left(\int_{\mathbb{R}}d\xi\, \frac{|\gamma_{s,t}(\xi,r)|^2}{|\xi|^{2(H_0-\frac{\varsigma}{2})-1}}\right) \, .
\end{align*}
We are here in a position to apply Lemma \ref{lem:tec} with $H:=H_0-\frac{\varsigma}{2}$, which entails that for all $0<\varepsilon< H_0-\frac{\varsigma}{2}$,
\begin{align}
\mathbb{I}^{\ka,n}(s,t)&\lesssim 2^{-2n\ka\varsigma} |t-s|^{\varepsilon}\int_{0}^{\infty}\frac{dr}{r^{2H_1-1}\{1+r^2\}^{\alpha}}\frac{1}{1+r^{4(H_0-\frac{\varsigma}{2}-\varepsilon)}}\nonumber\\
&\lesssim 2^{-2n\ka\varsigma}|t-s|^{\varepsilon}\left(\int_0^1 \frac{dr}{r^{2H_1-1}} + \int_{1}^{\infty}\frac{1}{r^{2(\alpha-\al_0-\varsigma)+1 -4\varepsilon}}dr \right)\, .\label{boun-i-pr}
\end{align}
Owing to assumption \eqref{cond-varsigma}, we can pick $\varepsilon>0$ small enough such that $\alpha-\al_0-\varsigma>2\varepsilon$, and for this choice, it is readily checked that the integrals into brackets in \eqref{boun-i-pr} are both finite. We have thus established that
$$\mathbb{I}^{\ka,n}(s,t)\lesssim 2^{-2n\ka\varsigma}|t-s|^{2\varepsilon}\, .$$

\smallskip

The estimation of $\mathbb{II}^{\ka,n}(s,t)$ can then be done along similar arguments. Namely, one has, for all $0<\varepsilon< H_0$,
\begin{align*}
\mathbb{II}^{\ka,n}(s,t)&\leq 2^{-2n\ka\varsigma}  \int_{\mathbb{R}^2}\, \frac{d\xi}{|\xi|^{2H_0-1}}\frac{d\eta}{|\eta|^{2(H_1-\varsigma)-1}}\{1+|\eta|^2\}^{-\alpha}|\gamma_{s,t}(\xi,|\eta|)|^2\\
&\lesssim 2^{-2n\ka\varsigma} \int_{0}^{\infty} \frac{dr}{r^{2(H_1-\varsigma)-1}\{1+r^2\}^{\alpha}}\left(\int_{\mathbb{R}}d\xi\, \frac{|\gamma_{s,t}(\xi,r)|^2}{|\xi|^{2H_0-1}}\right)\\
&\lesssim 2^{-2n\ka\varsigma} |t-s|^{\varepsilon}\int_{0}^{\infty}\frac{dr}{r^{2(H_1-\varsigma)-1}\{1+r^2\}^{\alpha}}\frac{1}{1+r^{4(H_0-\varepsilon)}}\\
&\lesssim 2^{-2n\ka\varsigma}|t-s|^{\varepsilon}\left(\int_0^1 \frac{dr}{r^{2(H_1-\varsigma)-1}} + \int_{1}^{\infty}\frac{1}{r^{2(\alpha-\al_0-\varsigma)+1 -4\varepsilon}}dr \right)\, .
\end{align*}
Again, we can pick $\varepsilon>0$ so that $\alpha-\al_0-\varsigma>2\varepsilon$, which leads us to
$$\mathbb{II}^{\ka,n}(s,t)\lesssim 2^{-2n\ka\varsigma}|t-s|^{\varepsilon}\, .$$

\smallskip

Going back to~\eqref{decompos-i-ii}, we obtain the expected estimate \eqref{bou-psi-step-1}.

\

\noindent
\textit{Step 2: Conclusion.}
Estimate \eqref{bou-psi-step-1} can be equivalently formulated as
\begin{equation}\label{recap-step-1}
\mathbb{E}\Big[\big\|\rho\cdot \big[\<Psi>^{\ka',m}_{s,t}-\<Psi>^{\ka,n}_{s,t}\big]\big\|_{\mathcal{W}^{-\alpha,2p}}^{2p}\Big]\lesssim 2^{-2n\ka\varsigma p}|t-s|^{2\varepsilon p}\, ,
\end{equation}
for all $p\geq 1$, $1\leq n \leq m$, $0<\ka \leq \ka'$, $0\leq s \leq t \leq 1$, $\varsigma>0$ satisfying \eqref{cond-varsigma}, and $\varepsilon>0$ small enough. In particular, it holds that
\begin{equation}\label{cauchy-seq}
\mathbb{E}\Big[\big\|\rho\cdot \big[\<Psi>^{\ka,m}_{s,t}-\<Psi>^{\ka,n}_{s,t}\big]\big\|_{\mathcal{W}^{-\alpha,2p}}^{2p}\Big]\lesssim 2^{-2n\ka\varsigma p}|t-s|^{2\varepsilon p}\, ,
\end{equation}
and
\begin{equation}\label{identif-lim}
\mathbb{E}\Big[\big\|\rho\cdot \big[\<Psi>^{\ka',n}_{s,t}-\<Psi>^{\ka,n}_{s,t}\big]\big\|_{\mathcal{W}^{-\alpha,2p}}^{2p}\Big]\lesssim 2^{-2n\ka\varsigma p}|t-s|^{2\varepsilon p}\, .
\end{equation}

By picking $p\geq 1$ large enough in \eqref{cauchy-seq}, we get, by Kolmogorov criterion, that $\rho\cdot \big[\<Psi>^{\ka,m}-\<Psi>^{\ka,n}\big]\in \mathcal{C}([0,1]; \mathcal{W}^{-\alpha,2p}(\mathbb{R}^d)) $ almost surely. We can then apply the classical Garsia-Rodemich-Rumsey estimate and assert that a.s, for all $p\geq 1$, $\varepsilon'>0$, $0\leq t \leq 1$, one has
$$ \big\|\rho\cdot \big[\<Psi>^{\ka,m}_{t}-\<Psi>^{\ka,n}_{t}\big]\big\|_{\mathcal{W}^{-\alpha,2p}}^{2p} \lesssim \int_{[0,1]^2} \frac{\big\|\rho\cdot \big[\<Psi>^{\ka,m}_{u,v}-\<Psi>^{\ka,n}_{u,v}\big]\big\|_{\mathcal{W}^{-\alpha,2p}}^{2p}}{|u-v|^{2\varepsilon' p+2}}\, dudv\, ,$$
for some (deterministic) proportional constant that only depends on $\varepsilon'$ and $p$.

\smallskip

As a consequence, using~\eqref{recap-step-1} again, we obtain that for any $0<\varepsilon'<\varepsilon$,
\begin{eqnarray*}
\mathbb{E}\Big[\big\|\rho\cdot \big[\<Psi>^{\ka,m}-\<Psi>^{\ka,n}\big]\big\|_{\cac([0,1];\mathcal{W}^{-\alpha,2p})}^{2p}\Big]
&\lesssim& 2^{-2n\ka\varsigma p}\int_{[0,1]^2}\frac{dudv}{|u-v|^{-2(\varepsilon-\varepsilon') p+2}}\ \lesssim \ 2^{-2n\ka\varsigma p}\, ,
\end{eqnarray*}
for any $p\geq p_0$, where $p_0\geq 1$ is such that $-2(\varepsilon-\varepsilon') p_0+2<1$.

\smallskip

Note that for $1\leq p\leq p_0$, one has, since $\rho$ is compactly-supported,
$$\mathbb{E}\Big[\big\|\rho\cdot \big[\<Psi>^{\ka,m}-\<Psi>^{\ka,n}\big]\big\|_{\cac([0,1];\mathcal{W}^{-\alpha,2p})}^{2p}\Big]^{\frac{1}{2p}}\lesssim \mathbb{E}\Big[\big\|\rho\cdot \big[\<Psi>^{\ka,m}-\<Psi>^{\ka,n}\big]\big\|_{\cac([0,1];\mathcal{W}^{-\alpha,2p_0})}^{2p_0}\Big]^{\frac{1}{2p_0}} \, ,$$
and so we can conclude that for any $p\geq 1$,
\begin{equation}\label{bou-in-l-p}
\big\|\rho\cdot \big[\<Psi>^{\ka,m}-\<Psi>^{\ka,n}\big]\big\|_{ L^{2p}(\Omega; \mathcal{C}([0,1];\mathcal{W}^{-\alpha,2p}))} \lesssim 2^{-n\ka\varsigma }\, .
\end{equation}
In particular, $(\rho\cdot \<Psi>^{\ka,n})_{n\geq 1}$ is a Cauchy sequence in $L^{2p}(\Omega; \mathcal{C}([0,1]; \mathcal{W}^{-\alpha,2p}(\mathbb{R}^d)))$, and thus it converges in this space. Let us (temporarily) denote by $\rho\cdot\<Psi>^\ka$ the limit of this sequence, for fixed $\ka>0$.

\smallskip

The fact that $\rho\cdot\<Psi>^\ka$ actually does not depend on $\ka$ can be readily deduced from \eqref{identif-lim}. Indeed, if $0<\ka\leq \ka'$, one has, for all $t\in [0,1]$ and $n\geq 1$,
\begin{align*}
&\mathbb{E}\Big[\big\|\rho\cdot \<Psi>^{\ka'}_{t}-\rho\cdot\<Psi>^{\ka}_{t}\big\|_{\mathcal{W}^{-\alpha,2p}}^{2p}\Big]\\
&\lesssim \mathbb{E}\Big[\big\|\rho\cdot\<Psi>^{\ka'}_{t}-\rho\cdot\<Psi>^{\ka',n}_{t}\big\|_{\mathcal{W}^{-\alpha,2p}}^{2p}\Big]+\mathbb{E}\Big[\big\|\rho\cdot\<Psi>^{\ka',n}_{t}-\rho\cdot\<Psi>^{\ka,n}_{t}\big\|_{\mathcal{W}^{-\alpha,2p}}^{2p}\Big]+\mathbb{E}\Big[\big\|\rho\cdot\<Psi>^{\ka,n}_{t}-\rho\cdot\<Psi>^{\ka}_{t}\big\|_{\mathcal{W}^{-\alpha,2p}}^{2p}\Big]\, ,
\end{align*}
and it is clear that those three quantities tend to $0$ as $n$ tends to $\infty$. Therefore, we can henceforth write $\rho\cdot\<Psi>$ instead of $\rho\cdot\<Psi>^\ka$.

\smallskip

The bound \eqref{speed-converg-def} can finally be derived from an application of the Borel-Cantelli lemma. Indeed, by letting $m$ tend to infinity in \eqref{bou-in-l-p} (for fixed $n\geq 1$ and $p=2$), we get that
$$\mathbb{E}\Big[\big\|\rho\cdot\<Psi>^{\ka,n}-\rho\cdot\<Psi>\big\|_{\mathcal{C}([0,1];\ch^{-\alpha})}^2\Big] \lesssim 2^{-2n\ka\varsigma }\, ,$$
and accordingly, for all $0<\tilde{\varsigma} < \varsigma$ and $n\geq 1$,
$$\mathbb{P}\Big(\big\|\rho\cdot\<Psi>^{\ka,n}-\rho\cdot\<Psi>\big\|_{\mathcal{C}([0,1];\ch^{-\alpha})} > 2^{-n\ka\tilde{\varsigma}}\Big)\lesssim 2^{-2n\ka(\varsigma-\tilde{\varsigma}) }\, .$$

\

\section{Noise discretization}\label{sec:noise-discretization}

Let us now initiate our discretization procedure, starting with the treatment of the noise. To be more specific, and as we announced it in the introduction, we are here interested in the discretization of the smoothened version $\partial_t\partial_x B^{\ka,n}$ of $\dot{B}$, at the basis of our interpretation of $\<Psi>$ (along definition \eqref{def:psi-ka-n} and Proposition \ref{sto}).

\smallskip

For this section (and this section only), we will rely on the time-space grid introduced in \eqref{grid-intro}, that is we set
\begin{equation*}
t_i=t_i^n:=\frac{i}{2^n} \ \  (i=0,\ldots,2^n-1) , \quad x_j=x_j^n:=\frac{j}{2^n} \ \  (j=-2^{2n},\ldots,2^{2n}-1) \ .
\end{equation*}
Note in particular that the set of points $(x^n_j)_{n\geq 1, \text{-}2^{2n}\leq j\leq 2^{2n}\text{-}1}$ is dense in $\R$. With this grid in hand, we now  define the discretized noise $\partial_t \partial_x \widetilde{B}^{\ka,n}$ as 
\begin{equation}\label{interpol-bipara-bis}
(\partial_t \partial_x \widetilde{B}^{\ka,n}) (t,x):=
\begin{cases}
2^{2n} \sq^n_{i,j}\!\! B^{\ka,n}\quad \text{if}\ t\in [t_i,t_{i+1}) \ \text{and} \ x\in [x_j,x_{j+1}),\\
0 \ \text{otherwise},
\end{cases}
\end{equation}
where we recall the notation $\sq^n_{i,j} b:=b_{t_{i+1},x_{j+1}}-b_{t_{i+1},x_{j}}-b_{t_{i},x_{j+1}}+b_{t_{i},x_{j}}$ for any two-parameter path $b$.

\smallskip

Observe that the so-defined approximation $\partial_t \partial_x \widetilde{B}^{\ka,n}$ indeed corresponds to the space-time derivative of some (piecewise) smooth approximation $\widetilde{B}^{\ka,n}$ of $B$: consider for instance the continuous sheet given for all $t\in [t_i,t_{i+1})$ and $x\in [x_j,x_{j+1})$ by
\begin{align}
&\widetilde{B}^{\ka,n}_{t,x}:=B^{\ka,n}_{t_i,x_j}\nonumber\\
&\hspace{0.5cm}+2^n(t-t_i)(B^{\ka,n}_{t_{i+1},x_j}-B^{\ka,n}_{t_i,x_j})+2^n(x-x_j) (B^{\ka,n}_{t_{i},x_{j+1}}-B^{\ka,n}_{t_i,x_j})+2^{2n}(t-t_i)(x-x_j) \sq^n_{i,j}\!\! B^{\ka,n}\, .\label{discreti-noi}
\end{align}

\

The corresponding approximation of $\<Psi>$ can then be written, for every $(t,x)\in [0,1]\times \R$, as
\begin{equation}\label{expression-luxo-n-bis}
\widetilde{\<Psi>}^{\ka,n}_{t}(x):=(G\ast \partial_t \partial_x \widetilde{B}^{\ka,n})_t(x)=2^{2n}\sum_{i=0}^{2^n-1} \sum_{j=-2^{2n}}^{2^{2n}-1}G^n_{i,j}(t,x)\,  \sq^n_{i,j}\!\! B^{\ka,n} \, ,
\end{equation}
where we have set
$$G^n_{i,j}(t,x):=\bigg(\int_{t_{i}}^{t_{i+1}}ds \int_{x_j}^{x_{j+1}} dy\,  G_{t-s}(x-y) \bigg) \, .$$

\

Our main control regarding the above noise-discretization procedure can now be stated as follows (we recall that the notation $\<Psi>^{\ka,n}$ refers to the solution associated with $\partial_t\partial_x B^{\ka,n}$, i.e. to the process considered in Proposition \ref{sto}).

\begin{proposition}\label{prop:discreti-noise}
Fix $(H_0,H_1)\in (0,1)^2$ such that $2H_0+H_1<1$, and let $\al_0>0$ be defined as in \eqref{assump-gene-al}. 

\smallskip

Then, for all $\al>\al_0$, $0<\ka\leq \frac{\al_0}{5}$, $0<\varsigma\leq \min(1-5\ka,\al_0-4\ka)$, $p\geq 1$ and for every test function $\rho:\mathbb{R}\rightarrow \mathbb{R}$ (i.e., smooth and compactly-supported), one has almost surely
\begin{equation}\label{speed-converg-1}
\sup_{t\in [0,1]}\big\| \rho \cdot\big\{\<Psi>^{\ka,n}_t-\widetilde{\<Psi>}^{\ka,n}_t\big\}\big\|_{\ch^{-\al}(\mathbb{R})} \lesssim 2^{-\varsigma n } ,
\end{equation}
\color{black}
where the proportional constant does not depend on $n$.
\end{proposition}

\

\begin{remark}\label{rk:ka}
The restriction on the \enquote{frequency} parameter $\ka$, that is the condition $0<\ka\leq \frac{\al_0}{5}$, can be interpreted as the result of some interesting competition phenomenon between the $2^{\ka n}$-approximation scaling in step 1 (see \eqref{b-bar-n}) and the size of the discretization grid in the present step 2. From a technical point of view, this competition can be observed through the chain of estimates \eqref{modele}, \eqref{chain-2} and \eqref{chain-3} in the proof of Proposition \ref{prop:discreti-noise}. 
At this point, it is not completely clear to us whether some finer estimates or the use of alternative Besov topologies could lead to a less stringent condition on $\ka$, or even to the extension of the estimate \eqref{speed-converg-1} to any $\ka >0$. Also, we do not know whether a similar bound could be established when replacing the approximation $B^{\ka,n}$ with the original sheet $B$ in the right-hand side of \eqref{interpol-bipara-bis} (this would morally corresponds to taking $\ka=\infty$).
\end{remark}

\smallskip

\begin{remark}
Let us briefly go back to the basic comparison sketched out in the introduction between the one-parameter process $\Psi^n$ (defined by \eqref{defi-psi-n}) and the above approximation process $\widetilde{\<Psi>}^{\ka,n}$. Along this analogy, the result of Proposition \ref{speed-converg-1} (morally) corresponds to the parabolic counterpart of the convergence of $\Psi^n$ to $B$. Now recall that a natural strategy to show that $\Psi^n \to B$ in the Sobolev scale $\ch^\ga$ (or more generally $\cw^{\ga,p}$) consists in the use of the continuous embedding (see e.g. \cite{runst-sickel}): for all $\ga\in (0,1)$ and $\varepsilon >0$ small enough,
$$\cs^{\ga+\varepsilon,p} \subset \cw^{\ga,p}$$
where $\cs^{\al,p}$ ($\al\in (0,1)$) refers to the so-called Slobodeckij space
$$\cs^{\al,p}:=\big\{f: \ \iint ds dt\, \frac{|f_t-f_s|^p}{|t-s|^{1+\al p}} <\infty \big\} \, .$$
In this way, convergence of $\Psi^n$ to $B$ in $\cw^{\ga,p}$ (for $\ga<H$) can be easily derived from the well-known (pathwise) H{\"o}lder regularity of $B$.\\
\indent
Unfortunately, due to the negative-order Sobolev regularity of the solution process $\<Psi>$ (as seen in Proposition \ref{sto}), such a simplification through an embedding strategy does not seem available in the present rough heat situation, and thus computations cannot be reduced to a pathwise control of spatial increments.
\end{remark}

\smallskip

\subsection{Proof of Proposition \ref{prop:discreti-noise}} We rely on a similar two-step strategy as in the proof of Proposition \ref{sto}.

\smallskip

\noindent
\textit{Step 1: A moment estimate.} The first (and main) objective is to establish the following bound:
\begin{equation}\label{gene-mom-estimate-psi-psiti}
\int_{\R} dx \, \bigg(\mathbb{E}\Big[ \Big| \cf^{-1}\Big(\{1+|.|^2\}^{-\frac{\al}{2}} \cf\big( \rho \cdot \big[ \widetilde{\<Psi>}^{\ka,n}_{s,t}-\<Psi>^{\ka,n}_{s,t}\big] \big) \Big)(x)\Big|^{2}\Big]\bigg)^p\lesssim 2^{-2n\varsigma p} |t-s|^{2\varepsilon p} \, ,
\end{equation}
for all $p\geq 1$ and $\ka,\varsigma,\varepsilon >0$ small enough, and where the proportional constant does not depend on $n$.

\smallskip

To this end, let us write, with the notation introduced in \eqref{expression-luxo-n-bis} and with representation \eqref{b-bar-n} in mind,
\small
\begin{align*}
&\widetilde{\<Psi>}^{\ka,n}_{s,t}=2^{2n}\sum_{i=0}^{2^n-1} \sum_{j=-2^{2n}}^{2^{2n}-1}G^n_{i,j}(t,x)\,  \sq^n_{i,j}\!\! B^{\ka,n} \\
&=-c_{H}\int_{\{|\xi|\leq 2^{2\ka n}, |\eta|\leq 2^{\ka n}\}}\widehat{W}(d\xi,d\eta) \, \frac{\xi}{|\xi|^{H_0+\frac12}}\frac{\eta}{|\eta|^{H_1+\frac12}}\sum_{i=0}^{2^n-1}\sum_{j=-2^{2n}}^{2^{2n}-1} G_{i,j}^n(s,t;y) \bigg( 2^{2n} \int_{t_i}^{t_{i+1}}du \int_{x_j}^{x_{j+1}}dz\, e^{\imath \xi u}e^{\imath \eta z}\bigg)
\end{align*}
\normalsize
where we have set $c_H:=c_{H_0}c_{H_1}$ and $G^n_{i,j}(s,t;y):=G^n_{i,j}(t,y)-G^n_{i,j}(s,y)$.

\smallskip

Using the representation \eqref{formula-heuris} of $\<Psi>^{\ka,n}$, the difference $\<Psi>^{\ka,n}_{s,t}-\widetilde{\<Psi>}^{\ka,n}_{s,t}$ can now be recast into
\begin{equation*}
\widetilde{\<Psi>}^{\ka,n}_{s,t}(y)-\<Psi>^{\ka,n}_{s,t}(y)=-c_{H}\int_{\{|\xi|\leq 2^{2\ka n}, |\eta|\leq 2^{\ka n}\}}\widehat{W}(d\xi,d\eta) \, \frac{\xi}{|\xi|^{H_0+\frac12}}\frac{\eta}{|\eta|^{H_1+\frac12}} \frx^n_{s,t}(y;\xi,\eta) \, 
\end{equation*}
with 
\begin{align*}
&\frx^n_{s,t}(y;\xi,\eta):=\bigg[\sum_{i=0}^{2^n-1}\sum_{j=-2^{2n}}^{2^{2n}-1} G_{i,j}^n(s,t;y) \bigg( 2^{2n} \int_{t_i}^{t_{i+1}}du \int_{x_j}^{x_{j+1}}dz\, e^{\imath \xi u}e^{\imath \eta z}\bigg)\bigg]-e^{\imath \eta y}\ga_{s,t}(\xi,|\eta|)\, .
\end{align*}
As a result, 
\begin{align*}
&\mathbb{E}\Big[ \big[ \widetilde{\<Psi>}^{\ka,n}_{s,t}-\<Psi>^{\ka,n}_{s,t}\big](y) \overline{ \big[ \widetilde{\<Psi>}^{\ka,n}_{s,t}-\<Psi>^{\ka,n}_{s,t}\big](\yti)}\Big]=c_{H}^2 \int_{\{|\xi|\leq 2^{2\ka n}, |\eta|\leq 2^{\ka n}\}}\frac{d\xi d\eta}{|\xi|^{2H_0-1}|\eta|^{2H_1-1}}\frx^n_{s,t}(y;\xi,\eta) \overline{\frx^n_{s,t}(\yti;\xi,\eta) }
\end{align*}
and thus, setting
$$\cj^{\frx,n}_{s,t}(\xi,\eta,\la):=\int_{\R}dy  \, \rho(y) e^{-\imath \la y}\frx^n_{s,t}(y;\xi,\eta) \, ,$$ 
one has
\begin{align*}
&\mathbb{E}\Big[ \Big| \cf^{-1}\Big(\{1+|.|^2\}^{-\frac{\al}{2}} \cf\Big( \rho \cdot \Big[ \widetilde{\<Psi>}^{\ka,n}_{s,t}-\<Psi>^{\ka,n}_{s,t}\Big]\Big) \Big)(x)\Big|^{2}\Big]\nonumber\\
&=\int_{\R}d\la\int_{\R}d\lati \, e^{\imath x (\la-\lati)} \frac{1}{\{1+|\la|^2\}^{\frac{\al}{2}}}\frac{1}{\{1+|\lati|^2\}^{\frac{\al}{2}}}\\
&\hspace{3cm} \int_{\R}dy  \, \rho(y) e^{-\imath \la y} \int_{\R}d\yti\, \rho(\yti) e^{\imath \lati \yti}\, \mathbb{E}\Big[ \big[ \widetilde{\<Psi>}^{\ka,n}_{s,t}-\<Psi>^{\ka,n}_{s,t}\big](y) \overline{ \big[ \widetilde{\<Psi>}^{\ka,n}_{s,t}-\<Psi>^{\ka,n}_{s,t}\big](\yti)}\Big] \\
&=\int_{\{|\xi|\leq 2^{2\ka n}, |\eta|\leq 2^{\ka n}\}}\frac{d\xi d\eta}{|\xi|^{2H_0-1}|\eta|^{2H_1-1}}\int_{\R}d\la\int_{\R}d\lati \, e^{\imath x (\la-\lati)} \frac{1}{\{1+|\la|^2\}^{\frac{\al}{2}}}\frac{1}{\{1+|\lati|^2\}^{\frac{\al}{2}}}\\
&\hspace{10cm}\cj^{\frx,n}_{s,t}(\xi,\eta,\la) \overline{\cj^{\frx,n}_{s,t}(\xi,\eta,\lati)} \\ 
&=\int_{\{|\xi|\leq 2^{2\ka n}, |\eta|\leq 2^{\ka n}\}}\frac{d\xi d\eta}{|\xi|^{2H_0-1}|\eta|^{2H_1-1}}\int_{\R}d\la\int_{\R}d\lati \, e^{\imath x (\la-\lati)} \frac{1}{\{1+|\eta+\la|^2\}^{\frac{\al}{2}}}\frac{1}{\{1+|\eta+\lati|^2\}^{\frac{\al}{2}}}\\
&\hspace{9cm}\cj^{\frx,n}_{s,t}(\xi,\eta,\eta+\la) \overline{\cj^{\frx,n}_{s,t}(\xi,\eta,\eta+\lati)} \, .
\end{align*}
Based on the latter expression, we get
\begin{align}
&\int_{\R} dx \, \bigg(\mathbb{E}\Big[ \Big| \cf^{-1}\Big(\{1+|.|^2\}^{-\frac{\al}{2}} \cf\big( \rho \cdot \big[ \widetilde{\<Psi>}^{\ka,n}_{s,t}-\<Psi>^{\ka,n}_{s,t}\big] \big) \Big)(x)\Big|^{2}\Big]\bigg)^p\nonumber\\
&=\prod_{i=1}^{p-1}\int_{\{|\xi_i|\leq 2^{2\ka n}, |\eta_i|\leq 2^{\ka n}\}}\frac{d\xi_i d\eta_i}{|\xi_i|^{2H_0-1}|\eta_i|^{2H_1-1}}\nonumber\\
&\hspace{2cm} \int_{\R} \frac{d\la_i}{\{1+|\eta_i+\la_i|^2\}^{\frac{\al}{2}}}\cj^{\frx,n}_{s,t}(\xi_i,\eta_i,\eta_i+\la_i)\int_{\R} \frac{d\lati_i}{\{1+|\eta_i+\lati_i|^2\}^{\frac{\al}{2}}} \overline{\cj^{\frx,n}_{s,t}(\xi_i,\eta_i,\eta_i+\lati_i)}\nonumber\\
&\hspace{1cm}\cdot\int_{\{|\xi_p|\leq 2^{2\ka n}, |\eta_p|\leq 2^{\ka n}\}}\frac{d\xi_p d\eta_p}{|\xi_p|^{2H_0-1}|\eta_p|^{2H_1-1}}\int_{\R} \frac{d\la_p}{\{1+|\eta_p+\la_p|^2\}^{\frac{\al}{2}}}\cj^{\frx,n}_{s,t}(\xi_p,\eta_p,\eta_p+\la_p)\nonumber\\
&\frac{1}{\{1+|\eta_p+(\la_p-\lati_{p-1}+\la_{p-1}-\ldots+\la_1)|^2\}^{\frac{\al}{2}}} \overline{\cj^{\frx,n}_{s,t}(\xi_p,\eta_p,\eta_p+(\la_p-\lati_{p-1}+\la_{p-1}-\ldots+\la_1))}\nonumber\\
&\leq \bigg(\prod_{i=1}^{p-1}\int_{\{|\xi_i|\leq 2^{2\ka n}, |\eta_i|\leq 2^{\ka n}\}}\frac{d\xi_i d\eta_i}{|\xi_i|^{2H_0-1}|\eta_i|^{2H_1-1}}\nonumber\\
&\hspace{2cm} \int_{\R} \frac{d\la_i}{\{1+|\eta_i+\la_i|^2\}^{\frac{\al}{2}}}\big|\cj^{\frx,n}_{s,t}(\xi_i,\eta_i,\eta_i+\la_i)\big|\int_{\R} \frac{d\lati_i}{\{1+|\eta_i+\lati_i|^2\}^{\frac{\al}{2}}} \big|\cj^{\frx,n}_{s,t}(\xi_i,\eta_i,\eta_i+\lati_i)\big|\bigg)\nonumber\\
&\hspace{1cm}\cdot\bigg(\int_{\{|\xi_p|\leq 2^{2\ka n}, |\eta_p|\leq 2^{\ka n}\}}\frac{d\xi_p d\eta_p}{|\xi_p|^{2H_0-1}|\eta_p|^{2H_1-1}}\int_{\R} \frac{d\la_p}{\{1+|\eta_p+\la_p|^2\}^{\al}}\big| \cj^{\frx,n}_{s,t}(\xi_p,\eta_p,\eta_p+\la_p)\big|^2\bigg) \, .\label{mom-p-frx}
\end{align}

\

\

We are now in a position to introduce our main technical estimate toward \eqref{gene-mom-estimate-psi-psiti} (for the sake of clarity, we have postponed the proof of this result to Section \ref{sec:proof-estim-i-frx}):
\begin{proposition}\label{prop:estim-i-frx}
With the above notation, one has, for all $(\xi,\eta)\in \R^2$ and $0<\varepsilon<\min(\frac14,\frac{\al}{2})$,
\begin{align}\label{boun-frx-int-la-1}
&\int_{\R} \frac{d\la}{\{1+|\eta+\la|^2\}^{\frac{\al}{2}}}\big|\cj^{\frx,n}_{s,t}(\xi,\eta,\eta+\la)\big|\lesssim  |t-s|^{\varepsilon}\Big[2^{-n} |\xi|+2^{-n}+2^{-n\al_0} |\eta| \Big] 
\end{align}
and 
\begin{align}\label{boun-frx-int-la-2}
&\int_{\R} \frac{d\la}{\{1+|\eta+\la|^2\}^{\al}}\big|\cj^{\frx,n}_{s,t}(\xi,\eta,\eta+\la)\big|^2\lesssim |t-s|^{2\varepsilon}\Big[ 2^{-n} |\xi|+2^{-n}+2^{-n\al_0} |\eta|  \Big]^2 \, .
\end{align}
\end{proposition}

\

\

Injecting \eqref{boun-frx-int-la-1} and \eqref{boun-frx-int-la-2} into \eqref{mom-p-frx}, we obtain that
\begin{align}
&\int_{\R} dx \, \bigg(\mathbb{E}\Big[ \Big| \cf^{-1}\Big(\{1+|.|^2\}^{-\frac{\al}{2}} \cf\big( \rho \cdot \big[ \widetilde{\<Psi>}^{\ka,n}_{s,t}-\<Psi>^{\ka,n}_{s,t}\big] \big) \Big)(x)\Big|^{2}\Big]\bigg)^p\nonumber\\
&\lesssim |t-s|^{2\varepsilon p} \bigg(\int_{\{|\xi|\leq 2^{2\ka n}, |\eta|\leq 2^{\ka n}\}}\frac{d\xi d\eta}{|\xi|^{2H_0-1}|\eta|^{2H_1-1}}\Big|2^{-n} |\xi|+2^{-n}+2^{-n\al_0} |\eta| \Big|^2 \bigg)^p\nonumber\\
&\lesssim 2^{-2n\varsigma p} |t-s|^{2\varepsilon p}\cdot \nonumber\\
&\hspace{0.5cm} \bigg(\int_{\{|\xi|\leq 2^{2\ka n}, |\eta|\leq 2^{\ka n}\}}\frac{d\xi d\eta}{|\xi|^{2H_0-1}|\eta|^{2H_1-1}}\Big|2^{-n(1-\varsigma)} |\xi|+2^{-n(1-\varsigma)}+2^{-n(\al_0-\varsigma)} |\eta| \Big|^2 \bigg)^p \, .\label{modele}
\end{align}
Now picking $\varsigma $ such that $0<\varsigma\leq \min(1-5\ka,\al_0-4\ka)$, one has, for any $(\xi,\eta)$ such that $|\xi|\leq 2^{2\ka n}$ and $|\eta|\leq 2^{\ka n}$, 
\begin{align*}
&\max\big(2^{-n(1-\varsigma)}|\xi|^2 |\eta|,2^{-n(1-\varsigma)}|\xi| |\eta|,2^{-n(\al_0-\varsigma)}|\xi| |\eta|^2\big)\\
&\lesssim \max\big(2^{-n(1-\varsigma-5\ka)},2^{-n(1-\varsigma-3\ka)},2^{-n(\al_0-\varsigma-4\ka)}\big) \lesssim 1 \, ,
\end{align*}
and so
\begin{align}
&\Big|2^{-n(1-\varsigma)} |\xi|+2^{-n(1-\varsigma)}+2^{-n(\al_0-\varsigma)} |\eta| \Big|\lesssim\frac{1}{1+|\xi|}\frac{1}{1+|\eta|} \, ,\label{chain-2}
\end{align}
which entails that
\begin{align}
&\sup_{n\geq 0}\int_{\{|\xi|\leq 2^{2\ka n}, |\eta|\leq 2^{\ka n}\}}\frac{d\xi d\eta}{|\xi|^{2H_0-1}|\eta|^{2H_1-1}}\Big|2^{-n(1-\varsigma)} |\xi|+2^{-n(1-\varsigma)}+2^{-n(\al_0-\varsigma)} |\eta| \Big|^2 \nonumber\\
&\hspace{2cm}\lesssim \bigg( \int_{\R}\frac{d\xi }{|\xi|^{2H_0-1}\{1+|\xi|^2\}}\bigg) \bigg(\int_{\R}\frac{d\eta}{|\eta|^{2H_1-1}\{1+|\eta|^2\}}\bigg) \ < \ \infty \, .\label{chain-3}
\end{align}
Going back to \eqref{modele}, this immediately provides us with the final estimate, namely: for all $p\geq 1$, $0<\varsigma\leq \min(1-5\ka,\al_0-4\ka)$ and $\varepsilon >0$ small enough,
\begin{equation}\label{estim-int-ca-finale}
\int_{\R} dx \, \bigg(\mathbb{E}\Big[ \Big| \cf^{-1}\Big(\{1+|.|^2\}^{-\frac{\al}{2}} \cf\big( \rho \cdot \big[ \widetilde{\<Psi>}^{\ka,n}_{s,t}-\<Psi>^{\ka,n}_{s,t}\big] \big) \Big)(x)\Big|^{2}\Big]\bigg)^p\lesssim 2^{-2n\varsigma p} |t-s|^{2\varepsilon p} \, ,
\end{equation}
for some proportional constant that does not depend on $n$.

\

\

\noindent
\textit{Step 2: Conclusion.} The arguments to conclude are essentially the same as those of the proof of Proposition \ref{sto} (\textit{Step 2}). First, one has, for every $p\geq 1$,
\begin{align*}
\mathbb{E}\Big[ \big\|\rho \cdot \big[ \widetilde{\<Psi>}^{\ka,n}_{s,t}-\<Psi>^{\ka,n}_{s,t}\big]\big\|_{\cw^{-\al,2p}}^{2p}\Big]&=\int_{\R} dx \, \mathbb{E}\Big[ \Big| \cf^{-1}\Big(\{1+|.|^2\}^{-\frac{\al}{2}} \cf\big( \rho \cdot \big[ \widetilde{\<Psi>}^{\ka,n}_{s,t}-\<Psi>^{\ka,n}_{s,t}\big] \big) \Big)(x)\Big|^{2p}\Big]\\
&\lesssim \int_{\R} dx \, \mathbb{E}\Big[ \Big| \cf^{-1}\Big(\{1+|.|^2\}^{-\frac{\al}{2}} \cf\big( \rho \cdot \big[ \widetilde{\<Psi>}^{\ka,n}_{s,t}-\<Psi>^{\ka,n}_{s,t}\big] \big) \Big)(x)\Big|^{2}\Big]^p \, ,
\end{align*}
where the latter inequality is derived from Gaussian hypercontractivity property.

\smallskip

Thus, thanks to \eqref{estim-int-ca-finale}, we get that for all $p\geq 1$, $0<\varsigma\leq \min(1-5\ka,\al_0-4\ka)$ and $\varepsilon >0$ small enough,
$$
\mathbb{E}\Big[ \big\|\rho \cdot \big[\widetilde{\<Psi>}^{\ka,n}_{s,t}-\<Psi>^{\ka,n}_{s,t}\big]\big\|_{\cw^{-\al,2p}}^{2p} \Big]\lesssim 2^{-2n\varsigma p} |t-s|^{2\varepsilon p}  \, .
$$
We can here apply the Garsia-Rodemich-Rumsey estimate and assert that for $\tilde{\varepsilon}>0$,
\begin{equation*}
\mathbb{E} \Big[ \big\|\rho \cdot \big[ \widetilde{\<Psi>}^{\ka,n}-\<Psi>^{\ka,n}\big]\big\|_{\cac([0,1];\cw^{-\al,2p})}^{2p} \Big]^{\frac{1}{2p}}\lesssim 2^{-n\varsigma }\bigg(\int_{[0,1]^2} \frac{du dv}{|u-v|^{2-2p(\varepsilon-\tilde{\varepsilon})}}\bigg)^{\frac{1}{2p}} \, .
\end{equation*}
The almost sure bound \eqref{speed-converg-1} can then be deduced from the same Borel-Cantelli argument as in the proof of Proposition \ref{sto}, and this completes the proof of Proposition \ref{prop:discreti-noise}. 

\

\subsection{Proof of Proposition \ref{prop:estim-i-frx}}\label{sec:proof-estim-i-frx}

Let us first recast $\cj^{\frx,n}_{s,t}(\xi,\eta,\eta+\la)$ as
\begin{align*}
&\cj^{\frx,n}_{s,t}(\xi,\eta,\eta+\la)\\
&=\bigg[\sum_{i=0}^{2^n-1}\sum_{j=-2^{2n}}^{2^{2n}-1} \bigg(\int_{\R}dy  \, \rho(y) e^{-\imath (\eta+\la) y}G_{i,j}^n(s,t;y)\bigg) \bigg( 2^{2n} \int_{t_i}^{t_{i+1}}du \int_{x_j}^{x_{j+1}} dz\, e^{\imath \xi u}e^{\imath \eta z}\bigg)\bigg]-\widehat{\rho}(\la)\ga_{s,t}(\xi,|\eta|)\, .
\end{align*}
Then, since
\begin{align*}
G^n_{i,j}(s,t;y)&=\int_{t_i}^{t_{i+1}} dr \int_{x_j}^{x_{j+1}} dw \, \big\{G_{t-r}(y-w)-G_{s-r}(y-w)\big\}\\
&=\int_{t_i}^{t_{i+1}} dr \int_{x_j}^{x_{j+1}} dw \int_{\R}d\beta\, \big\{\widehat{G_{t-r}}(\be)-\widehat{G_{s-r}}(\be)\big\}e^{\imath \beta(y-w)} \, ,
\end{align*}
one can write, for all $i,j$,
\begin{align*}
&\bigg(\int_{\R}dy  \, \rho(y) e^{-\imath (\eta+\la) y}G_{i,j}^n(s,t;y)\bigg) \bigg( 2^{2n} \int_{t_i}^{t_{i+1}}du \int_{x_j}^{x_{j+1}} dz\, e^{\imath \xi u}e^{\imath \eta z}\bigg)\\
&=2^{2n}\int_{\R}dy  \, \rho(y) e^{-\imath y(\eta+\la) }\int_{t_i}^{t_{i+1}} dr \int_{x_j}^{x_{j+1}} dw \int_{\R}d\be \,\big\{\widehat{G_{t-r}}(\be)-\widehat{G_{s-r}}(\be)\big\} e^{\imath \be(y-w)}\int_{t_i}^{t_{i+1}} e^{\imath \xi u} \, du \int_{x_j}^{x_{j+1}} e^{\imath \eta z} \, dz\\
&=\int_{\R}d\be \bigg[\int_{\R}dy  \, \rho(y) e^{-\imath y(\eta+\la-\be) }\bigg]\bigg[2^{n}\int_{t_i}^{t_{i+1}} dr \, \big\{\widehat{G_{t-r}}(\be)-\widehat{G_{s-r}}(\be)\big\}\int_{t_i}^{t_{i+1}} e^{\imath \xi u} \, du \bigg]\\
&\hspace{10cm}\bigg[2^{n}\int_{x_j}^{x_{j+1}} dw \int_{x_j}^{x_{j+1}} dz\, e^{-\imath \be w} e^{\imath \eta z} \bigg]\\
&=\int_{\R}d\be \bigg[\int_{\R}dy  \, \rho(y) e^{-\imath y(\la-\be) }\bigg]\bigg[2^{n}\int_{t_i}^{t_{i+1}} dr \, \big\{\widehat{G_{t-r}}(\eta+\be)-\widehat{G_{s-r}}(\eta+\be)\big\}\int_{t_i}^{t_{i+1}} e^{\imath \xi u} \, du \bigg]\\
&\hspace{9.5cm}\bigg[2^{n}\int_{x_j}^{x_{j+1}} dw \int_{x_j}^{x_{j+1}} dz\, e^{-\imath (\eta+\be) w} e^{\imath \eta z} \bigg] 
\end{align*}
and so
\begin{align*}
&\sum_{i=0}^{2^n-1}\sum_{j=-2^{2n}}^{2^{2n}-1}\bigg(\int_{\R}dy  \, \rho(y) e^{-\imath (\eta+\la) y}G_{i,j}^n(s,t;y)\bigg) \bigg( 2^{2n} \int_{t_i}^{t_{i+1}}du \int_{x_j}^{x_{j+1}}dz\, e^{\imath \xi u}e^{\imath \eta z}\bigg)\\
&\hspace{3cm}=\int_{\R} d\be \, \widehat{\rho}(\la-\be) \ga^{n}_{s,t}(\xi,\eta+\be) \delta^{n}(\eta,\eta+\be)
\end{align*}
with
\begin{equation}\label{defi-ga-n}
\ga^{n}_{t}(\xi,\be):=\int_{0}^{t} dr \, \widehat{G_{t-r}}(\be)\bigg(\sum_{i=0}^{2^n-1}\1_{\{t_i<r<t_{i+1}\}}2^n\int_{t_i}^{t_{i+1}} e^{\imath \xi u} \, du\bigg)
\end{equation}
and
\begin{equation}\label{defi-delta-n}
\delta^{n}(\eta,\be):=\int_{\R} dw   \, e^{-\imath w\be} \bigg(\sum_{j=-2^{2n}}^{2^{2n}-1}\1_{\{x_j<w<x_{j+1}\}}2^n \int_{x_j}^{x_{j+1}}dz \, e^{\imath \eta z}\bigg) \, .
\end{equation}
We have thus derived the representation
\begin{align}
&\cj^{\frx,n}_{s,t}(\xi,\eta,\eta+\la)=\bigg[\int_{\R} d\be \, \widehat{\rho}(\la-\be) \ga^{n}_{s,t}(\xi,\eta+\be) \delta^{n}(\eta,\eta+\be)\bigg]-\widehat{\rho}(\la)\ga_{s,t}(\xi,|\eta|)\, ,\label{cj-n-expand}
\end{align}
which will help us to prove the following intermediate estimate:

\begin{lemma}\label{lem:estim-i-frx}
For all $(\xi,\eta,\la)\in \R^3$, $\ell \geq 1$, $\sigma_1,\sigma_2,\sigma_3\in [0,1]$ and $0<\varepsilon<\frac14$, it holds that 
\begin{align}
&|\cj_{s,t}^{\frx,n}(\xi,\eta,\eta+\la)|\nonumber\\
&\lesssim \big|\widehat{\rho}(\la)\big| |t-s|^{\sigma_1}\{1+|\eta|^{2\sigma_1}\}2^{-n\sigma_2} |\xi|^{\sigma_2}  +\frac{|t-s|^\varepsilon}{1+|\la|}\Big[2^{-n\ell}+2^{-n}|\eta|+2^{-n\sigma_3} |\eta| |\la|^{\sigma_3} \Big] \, ,\label{las-boun}
\end{align}
where the proportional constant does not depend on $n,\xi,\eta,\la$.
\end{lemma}

Before we prove this technical lemma, let us see how it allows us to derive the estimates in Proposition \ref{prop:estim-i-frx}. In fact, applying \eqref{las-boun} with $\ell=1$, $\sigma_1=\varepsilon$, $\sigma_2=1$ and $\sigma_3=\al_0$, we deduce that for every $0<\varepsilon<\frac14$,
\begin{align}
&\int_{\R} \frac{d\la}{\{1+|\eta+\la|^2\}^{\frac{\al}{2}}}\big|\cj^{\frx,n}_{s,t}(\xi,\eta,\eta+\la)\big|\lesssim |t-s|^{\varepsilon}\{1+|\eta|^{2\varepsilon}\}2^{-n} |\xi|\int_{\R}d\la \, \frac{\big|\widehat{\rho}(\la)\big| }{\{1+|\eta+\la|^2\}^{\frac{\al}{2}}}\nonumber\\
&\hspace{4cm}+|t-s|^\varepsilon\Big[2^{-n}+2^{-n}|\eta|+2^{-n\al_0} |\eta| \Big]\int_{\R} \frac{d\la}{\{1+|\eta+\la|^\al\}\{1+|\la|^{1-\al_0}\}}\, .\label{intermed}
\end{align}
At this point, one easily checks that for every $\beta >0$,
\begin{equation}\label{rho-ha-est}
\int_{\R}d\la \, \frac{\big|\widehat{\rho}(\la)\big| }{\{1+|\eta+\la|^2\}^{\frac{\beta}{2}}} \lesssim \frac{1}{1+|\eta|^\beta} \, .
\end{equation}
Also, since $0<\al_0<\min(1,\al)$, it is not hard to see that
$$\sup_{\eta\in \R}\int_{\R} \frac{d\la}{\{1+|\eta+\la|^\al\}\{1+|\la|^{1-\al_0}\}}\ < \ \infty \, ,$$
and therefore \eqref{intermed} leads us to
\begin{align*}
&\int_{\R} \frac{d\la}{\{1+|\eta+\la|^2\}^{\frac{\al}{2}}}\big|\cj^{\frx,n}_{s,t}(\xi,\eta,\eta+\la)\big|\lesssim |t-s|^{\varepsilon}\bigg[2^{-n} |\xi|\frac{1+|\eta|^{2\varepsilon}}{1+|\eta|^{\al}}+2^{-n}+2^{-n\al_0} |\eta| \bigg]
\end{align*}
which, if $0<\varepsilon<\min(\frac14,\frac{\al}{2})$, immediately yields \eqref{boun-frx-int-la-1}.

\smallskip

We can then derive \eqref{boun-frx-int-la-2} with similar arguments. Namely, applying again \eqref{las-boun} with $\ell=1$, $\sigma_1=\varepsilon$, $\sigma_2=1$ and $\sigma_3=\al_0$, we get first that for every $0<\varepsilon<\frac14$,
\begin{align}
&\int_{\R} \frac{d\la}{\{1+|\eta+\la|^2\}^{\al}}\big|\cj^{\frx,n}_{s,t}(\xi,\eta,\eta+\la)\big|^2\lesssim |t-s|^{2\varepsilon}\Big[\{1+|\eta|^{2\varepsilon}\}2^{-n} |\xi|\Big]^2\int_{\R}d\la \, \frac{\big|\widehat{\rho}(\la)\big|^2 }{\{1+|\eta+\la|^2\}^{\al}}\nonumber\\
&\hspace{4cm}+|t-s|^{2\varepsilon}\Big[2^{-n}+2^{-n}|\eta|+2^{-n\al_0} |\eta| \Big]^2\int_{\R} \frac{d\la}{\{1+|\eta+\la|^{2\al}\}\{1+|\la|^{2(1-\al_0)}\}}\, .\label{intermed-2}
\end{align}
For the same elementary reasons as in \eqref{rho-ha-est} (write $\big|\widehat{\rho}(\la)\big|^2=\big|(\widehat{\rho \ast \rho})(\la)\big|$), one has
$$\int_{\R}d\la \, \frac{\big|\widehat{\rho}(\la)\big|^2 }{\{1+|\eta+\la|^2\}^{\al}} \lesssim \frac{1}{1+|\eta|^{2\al}}\, ,$$
and again, since $0<\al_0<\min(1,\al)$, it is easy to check that
$$\sup_{\eta\in \R}\int_{\R} \frac{d\la}{\{1+|\eta+\la|^{2\al}\}\{1+|\la|^{2(1-\al_0)}\}}\ < \ \infty \, .$$
Therefore, we end up with
\begin{align*}
&\int_{\R} \frac{d\la}{\{1+|\eta+\la|^2\}^{\al}}\big|\cj^{\frx,n}_{s,t}(\xi,\eta,\eta+\la)\big|^2\lesssim |t-s|^{2\varepsilon}\bigg[2^{-n} |\xi|\frac{1+|\eta|^{2\varepsilon}}{1+|\eta|^{\al}}+2^{-n}+2^{-n\al_0} |\eta| \bigg]^2
\end{align*}
and by picking $0<\varepsilon<\min(\frac14,\frac{\al}{2})$, we obtain \eqref{boun-frx-int-la-2}, which achieves the proof of Proposition \ref{prop:estim-i-frx}.

\

Thus, it only remains us to provide the details behind Lemma \ref{lem:estim-i-frx}.
\begin{proof}[Proof of Lemma \ref{lem:estim-i-frx}]
Going back to the definition \eqref{defi-delta-n} of $\delta^n$, one has
\begin{align*}
&\delta^{n}(\eta,\eta+\be)=\int_{\R} dw   \, e^{-\imath w\be} \bigg(\sum_{j=-2^{2n}}^{2^{2n}-1}\1_{\{x_j<w<x_{j+1}\}}2^n \int_{x_j}^{x_{j+1}}dz \, e^{\imath \eta (z-w)}\bigg)\\
&=\int_{\R} dw   \, e^{-\imath w\be} \bigg(\sum_{j=-2^{2n}}^{2^{2n}-1}\1_{\{x_j<w<x_{j+1}\}}2^n \int_{x_j}^{x_{j+1}}dz \, \{e^{\imath \eta (z-w)}-1\}\bigg)+\int_{\R} dw   \, e^{-\imath w\be}\1_{\{|w|\leq 2^n\}}
\end{align*}
so that, according to \eqref{cj-n-expand}, we can write
\begin{align}
&\cj_{s,t}^{\frx,n}(\xi,\eta,\eta+\la)\nonumber\\
&=\bigg[\int_{\R} dw   \bigg(\int_{\R} d\be \, e^{-\imath w\be} \widehat{\rho}(\la-\be) \ga^{n}_{s,t}(\xi,\eta+\be)  \bigg) \bigg(\sum_{j=-2^{2n}}^{2^{2n}-1}\1_{\{x_j<w<x_{j+1}\}}2^n \int_{x_j}^{x_{j+1}}dz \, \{e^{\imath \eta (z-w)}-1\}\bigg)\bigg]\nonumber\\
&\hspace{3cm}+\int_{\R} dw   \, \1_{\{|w|\leq 2^n\}} \int_{\R} d\be \, e^{-\imath w\be}\widehat{\rho}(\la-\be) \ga^{n}_{s,t}(\xi,\eta+\be) -\widehat{\rho}(\la)\ga_{s,t}(\xi,|\eta|)\nonumber\\
&=\bigg[\int_{\R} dw   \bigg(\int_{\R} d\be \, e^{-\imath w\be} \widehat{\rho}(\la-\be) \ga^{n}_{s,t}(\xi,\eta+\be)  \bigg) \bigg(\sum_{j=-2^{2n}}^{2^{2n}-1}\1_{\{x_j<w<x_{j+1}\}}2^n \int_{x_j}^{x_{j+1}}dz \, \{e^{\imath \eta (z-w)}-1\}\bigg)\bigg]\nonumber\\
&\hspace{1cm}+\bigg[-\int_{|w|\geq 2^n} dw   \int_{\R} d\be \, e^{-\imath w\be}\widehat{\rho}(\la-\be) \ga^{n}_{s,t}(\xi,\eta+\be)\bigg]+\bigg[\widehat{\rho}(\la)\ga^n_{s,t}(\xi,\eta) -\widehat{\rho}(\la)\ga_{s,t}(\xi,|\eta|)\bigg]\nonumber\\
&=: \mathbb{I}^n_{s,t}(\xi,\eta,\la)+\mathbb{II}^n_{s,t}(\xi,\eta,\la)+\mathbb{III}^n_{s,t}(\xi,\eta,\la) \, .\label{decompo-i-ii-iii}
\end{align}

\smallskip

Let us estimate these three quantities separately.

\

\noindent
\textit{Treatment of $\mathbb{I}^n_{s,t}(\xi,\eta,\la)$.} We have
\begin{equation}\label{f-n-s-t}
\int_{\R} d\be \, e^{-\imath w\be} \widehat{\rho}(\la-\be) \ga^{n}_{s,t}(\xi,\eta+\be)=e^{-\imath \la w}\int_{\R} d\be \, e^{\imath w\be} \widehat{\rho}(\be) \ga^{n}_{s,t}(\xi,\eta+\la-\beta)=e^{-\imath\la w} f^n_{s,t}(w;\xi,\eta,\la)  \, ,
\end{equation}
with
\begin{equation}\label{defi:f-n-s-t}
f^n_{s,t}(w;\xi,\eta,\la):=\bigg[\rho\ast \cf^{-1}\Big(\ga^{n}_{s,t}(\xi,\eta+\la-.)\Big)\bigg](w) \, .
\end{equation}
Using the additional notation 
$$h^{n,j}(w;\eta):=2^n\int_{x_j}^{x_{j+1}}dz \, \big\{e^{\imath \eta (z-w)}-1\big\} \, ,$$
we can rewrite $\mathbb{I}^n_{s,t}(\xi,\eta,\la)$ as
$$\mathbb{I}^n_{s,t}(\xi,\eta,\la)=\sum_{j=-2^{2n}}^{2^{2n}-1} \int_{x_j}^{x_{j+1}}dw \, e^{-\imath \la w} f^n_{s,t}(w;\xi,\eta,\la)h^{n,j}(w;\eta) \, .$$
On the one hand, it is clear that $| h^{n,j}(w;\eta) |\leq 2^{-n}|\eta|$ for any $w\in [x_j,x_{j+1}]$, which immediately yields
\begin{equation}\label{unif-bou-la}
\big|\mathbb{I}^n_{s,t}(\xi,\eta,\la)\big|\leq 2^{-n}|\eta| \, \big\| f^n_{s,t}(.;\xi,\eta,\la) \big\|_{L^1}  \ .
\end{equation}
On the other hand, one has for every $j$,
\begin{align}
&\int_{x_j}^{x_{j+1}}dw \, e^{-\imath \la w} f^n_{s,t}(w;\xi,\eta,\la)h^{n,j}(w;\eta)\nonumber\\
&=\frac{-1}{\imath \la} \bigg[ e^{-\imath\la x_{j+1}} f^n_{s,t}(x_{j+1};\xi,\eta,\la)h^{n,j}(x_{j+1};\eta)-e^{-\imath \la x_j} f^n_{s,t}(x_j;\xi,\eta,\la)h^{n,j}(x_j;\eta)\nonumber\\
&\hspace{1cm}-\int_{x_j}^{x_{j+1}} dw \, e^{-\imath \la w} (\partial_w f^n_{s,t})(w;\xi,\eta,\la)h^{n,j}(w;\eta)-\int_{x_j}^{x_{j+1}} dw \, e^{-\imath \la w} f^n_{s,t}(w;\xi,\eta,\la)(\partial_w h^{n,j})(w;\eta)\bigg]\nonumber\\
&=\frac{-1}{\imath \la} \Big[ I^{n,j}_{s,t}+I\! I^{n,j}_{s,t}+I\! I\! I^{n,j}_{s,t} \Big] \, ,\label{sum-a-b-c-j}
\end{align}
with
$$I^{n,j}_{s,t}:=\big\{e^{-\imath\la x_{j+1}} f^n_{s,t}(x_{j+1};\xi,\eta,\la)-e^{-\imath\la x_{j}} f^n_{s,t}(x_{j};\xi,\eta,\la)\big\}h^{n,j}(x_{j+1};\eta) \, ,$$
$$ I\! I^{n,j}_{s,t}:=-\int_{x_j}^{x_{j+1}} dw \, e^{-\imath \la w} (\partial_w f^n_{s,t})(w;\xi,\eta,\la)h^{n,j}(w;\eta) \, ,$$
$$ I\! I\! I^{n,j}_{s,t}:=\int_{x_j}^{x_{j+1}} dw \, \big\{e^{-\imath\la x_{j}} f^n_{s,t}(x_{j};\xi,\eta,\la)-e^{-\imath \la w} f^n_{s,t}(w;\xi,\eta,\la)\big\}(\partial_w h^{n,j})(w;\eta) \, .$$

\

First, since $h^{n,j}(x_{j+1};\eta)=2^n\int_{0}^{2^{-n}}dz \, \big\{e^{-\imath \eta z}-1\big\}$ does not depend on $j$, we get that
\begin{align}\label{sum-a-j}
\bigg|\sum_{j=-2^{2n}}^{2^{2n}-1} I^{n,j}_{s,t} \bigg|&=\bigg|\bigg(2^n\int_{0}^{2^{-n}}dz \, \big\{e^{-\imath \eta z}-1\big\}\bigg) \big\{ e^{-\imath\la 2^n} f^n_{s,t}(2^n;\xi,\eta,\la)-e^{\imath\la 2^n} f^n_{s,t}(-2^n;\xi,\eta,\la)\big\} \bigg|\\
&\lesssim 2^{-n}|\eta|\, \big\|f^n_{s,t}(.;\xi,\eta,\la)\big\|_{L^\infty} \, .
\end{align}

\

Then
\begin{equation}\label{sum-b-j}
\bigg| \sum_{j=-2^{2n}}^{2^{2n}-1}  I\! I^{n,j}_{s,t}\bigg|\lesssim 2^{-n}|\eta| \sum_{j=-2^{2n}}^{2^{2n}}\int_{x_j}^{x_{j+1}}dw\,  \big|(\partial_w f^n_{s,t})(w;\xi,\eta,\la) \big| \lesssim 2^{-n}|\eta|\, \big\|(\partial_w f^n_{s,t})(.;\xi,\eta,\la)\big\|_{L^1} \, .
\end{equation}

\

Finally, since $|(\partial_w h^{n,j})(w;\eta)|=\big|(-\imath \eta)2^n\int_{x_j}^{x_{j+1}}dz \, e^{\imath \eta (z-w)}\big|\leq |\eta|$, we have for any $\sigma\in [0,1]$
\begin{align*}
|I\! I\! I^{n,j}_{s,t}|&\leq |\eta| \int_{x_j}^{x_{j+1}} dw \, \Big[\big|f^n_{s,t}(x_{j};\xi,\eta,\la)-f^n_{s,t}(w;\xi,\eta,\la)\big|+\big|e^{-\imath\la x_{j}}-e^{-\imath \la w}\big| \big|f^n_{s,t}(w;\xi,\eta,\la)\big|\Big]\\
&\leq |\eta| \bigg[\int_{x_j}^{x_{j+1}} dw \int_{x_j}^w dv \, \big| \big(\partial_w f^n_{s,t}\big)(v;\xi,\eta,\la)\big|+ |\la|^{\sigma} 2^{-n\sigma}\int_{x_j}^{x_{j+1}} dw \, \big|f^n_{s,t}(w;\xi,\eta,\la)\big| \bigg]\\
&\leq |\eta| \bigg[2^{-n}\int_{x_j}^{x_{j+1}} dv \, \big| \big(\partial_w f^n_{s,t}\big)(v;\xi,\eta,\la)\big|+ |\la|^{\sigma} 2^{-n\sigma}\int_{x_j}^{x_{j+1}} dw \, \big|f^n_{s,t}(w;\xi,\eta,\la)\big| \bigg] \, ,
\end{align*}
which immediately entails
\begin{equation}\label{sum-c-j}
\bigg|\sum_{j=-2^{2n}}^{2^{2n}-1} I\! I\! I^{n,j}_{s,t} \bigg| \leq \bigg[ 2^{-n}|\eta|\,  \big\| \big(\partial_w f^n_{s,t}\big)(.;\xi,\eta,\la) \big\|_{L^1}+2^{-n\sigma} |\eta| |\la|^{\sigma}\,  \big\| f^n_{s,t}(.;\xi,\eta,\la) \big\|_{L^1} \bigg]\, .
\end{equation}

\

\

Injecting \eqref{sum-a-j}-\eqref{sum-b-j}-\eqref{sum-c-j} into \eqref{sum-a-b-c-j}, and then combining the result with \eqref{unif-bou-la}, we obtain that
\begin{align}
&\big|\mathbb{I}^n_{s,t}(\xi,\eta,\la) \big|\nonumber\\
&\lesssim \frac{1}{1+|\la|}\Big[2^{-n}|\eta|+2^{-n\sigma} |\eta| |\la|^\sigma \Big] \Big[\big\| f^n_{s,t}(.;\xi,\eta,\la) \big\|_{L^\infty}+\big\| f^n_{s,t}(.;\xi,\eta,\la) \big\|_{L^1} +\big\| \big(\partial_w f^n_{s,t}\big)(.;\xi,\eta,\la) \big\|_{L^1} \Big] \, .\label{first-bou-r-n}
\end{align}

\

In order to go further, note that by the definition of $f^n_{s,t}$, and since $\rho\in \cac_c^\infty(\R)$, one has
\begin{equation}\label{interm-j-1}
\big\| f^n_{s,t}(.;\xi,\eta,\la) \big\|_{L^\infty}+\big\| f^n_{s,t}(.;\xi,\eta,\la) \big\|_{L^1} +\big\| \big(\partial_w f^n_{s,t}\big)(.;\xi,\eta,\la) \big\|_{L^1}  \lesssim \Big\|\cf^{-1}\Big(\ga^{n}_{s,t}(\xi,\eta+\la-.)\Big)\Big\|_{L^1} \, .
\end{equation}

\

Then observe that
\begin{equation}\label{f-1-ga}
\cf^{-1}\Big(\ga^{n}_{s,t}(\xi,\eta+\la-.)\Big)(w)=e^{\imath w (\la+\eta)}\int_0^t dr \, \big\{G_{t-r}(w)-G_{s-r}(w)\big\}\bigg[ \sum_{i=0}^{2^{n}-1}  \1_{\{t^{n}_i < r<t^{n}_{i+1}\}} 2^{n}\int_{t_i}^{t_{i+1}} du\, e^{\imath \xi u}  \bigg]  \, ,
\end{equation}
and so, for any $0<\varepsilon<\frac14$,
\begin{equation}\label{interm-j-2}
\Big\|\cf^{-1}\Big(\ga^{n}_{s,t}(\xi,\eta+\la-.)\Big)\Big\|_{L^1}\leq \int_{\R}dw\int_{0}^{t} dr \, \big| G_{t-r}(w)-G_{s-r}(w)\big| \lesssim |t-s|^{\varepsilon} \, .
\end{equation}

\

Injecting \eqref{interm-j-1}-\eqref{interm-j-2} into \eqref{first-bou-r-n}, we get that for any $\si\in [0,1]$ and $0<\varepsilon<\frac14$,
\begin{equation}\label{estimate-i}
\big|\mathbb{I}^n_{s,t}(\xi,\eta,\la) \big|\lesssim \frac{|t-s|^\varepsilon}{1+|\la|}\Big[2^{-n}|\eta|+2^{-n\sigma} |\eta| |\la|^\sigma \Big]  \, ,
\end{equation}
for some proportional constant that does not depend on $\xi,\eta,\la$.

\

\

\noindent
\textit{Treatment of $\mathbb{II}^n_{s,t}(\xi,\eta,\la)$.} Following \eqref{f-n-s-t}, we can write
$$
\mathbb{II}^n_{s,t}(\xi,\eta,\la)=-\int_{|w|\geq 2^{n}}dw \, e^{-\imath\la w} f^n_{s,t}(w;\xi,\eta,\la)
$$
with $f^n_{s,t}$ defined by \eqref{defi:f-n-s-t}, and so
\begin{align}
&\big|\mathbb{II}^n_{s,t}(\xi,\eta,\la)\big|=\bigg|\int_{|w|\geq 2^{n}}dw\, e^{-\imath \la w} f^n_{s,t}(w;\xi,\eta,\la)\bigg|\nonumber\\
&\lesssim \max\bigg( \int_{|w|\geq 2^{n}}dw\, \big|f^n_{s,t}(w;\xi,\eta,\la)\big|, \frac{1}{|\la|} \bigg\{\sup_{|w|\geq 2^{n}} \big|f^n_{s,t}(w;\xi,\eta,\la)\big|+\int_{|w|\geq 2^{n}}dw\, \big|(\partial_w f^n_{s,t})(w;\xi,\eta,\la)\big|\bigg\}\bigg)\ .\label{born-ii}
\end{align}
Remember that $f^n_{s,t}=\rho\ast \cf^{-1}\big(\ga^{n}_{s,t}(\xi,\eta+\la-.)\big)$, and so, since $\text{supp} \, \rho\subset [-A,A]$ for some $A>0$, it is readily checked that
\begin{align}
&\int_{|w|\geq 2^{n}}dw\, \big|f^n_{s,t}(w;\xi,\eta,\la)\big|+\sup_{|w|\geq 2^{n}} \big|f^n_{s,t}(w;\xi,\eta,\la)\big|+\int_{|w|\geq 2^{n}}dw\, \big|(\partial_w f^n_{s,t})(w;\xi,\eta,\la)\big|\nonumber\\
&\lesssim \int_{|w|\geq 2^n-A} dw \Big|\cf^{-1}\Big(\ga^{n}_{s,t}(\xi,\eta+\la-.)\Big)(w)\Big| \lesssim \int_{|w|\geq 2^{n-1}} dw \Big|\cf^{-1}\Big(\ga^{n}_{s,t}(\xi,\eta+\la-.)\Big)(w)\Big| \, ,\label{interm-j-1-b}
\end{align}
for any $n$ large enough. Then, using the representation \eqref{f-1-ga}, we deduce that for all $\ell\geq 1$ and $0<\varepsilon<\frac14$,
\begin{align}
&\int_{|w|\geq 2^{n-1}} dw \Big|\cf^{-1}\Big(\ga^{n}_{s,t}(\xi,\eta+\la-.)\Big)(w)\Big|\leq \int_{|w|\geq 2^{n-1}}dw\int_{0}^{t} dr \, \big| G_{t-r}(w)-G_{s-r}(w)\big| \lesssim 2^{-n\ell} |t-s|^\varepsilon \, .\label{interm-j-2-b}
\end{align}
Combining \eqref{interm-j-1-b}-\eqref{interm-j-2-b} with \eqref{born-ii}, we obtain that
\begin{equation}\label{fin-boun-r-1-c}
\big| \mathbb{II}^n_{s,t}(\xi,\eta,\la) \big| \lesssim \frac{|t-s|^\varepsilon}{1+|\la|}2^{-n\ell}   \, ,
\end{equation}
for all $\ell\geq 1$ and $0<\varepsilon<\frac14$. 

\

\noindent
\textit{Treatment of $\mathbb{III}^n_{s,t}(\xi,\eta,\la)$.} We need to control the difference $\ga^n_{s,t}(\xi,\eta)-\ga_{s,t}(\xi,|\eta|)$. In fact, with expression \eqref{defi-ga-t} in mind, one has
\begin{align*}
&\ga^{n}_{s,t}(\xi,\eta)\\
&=\int_0^t dr \, \big\{\widehat{G_{t-r}}(\eta)-\widehat{G_{s-r}}(\eta)\big\} e^{\imath \xi r} \bigg[ \sum_{i=0}^{2^{n}-1} \1_{\{t_i < r<t_{i+1}\}}2^{n} \int_{t_i}^{t_{i+1}} du\, e^{\imath \xi(u-r)} \bigg]\\
&=\ga_{s,t}(\xi,|\eta|)+\int_0^t dr \, \big\{\widehat{G_{t-r}}(\eta)-\widehat{G_{s-r}}(\eta)\big\} e^{\imath \xi r} \bigg[ \sum_{i=0}^{2^{n}-1} \1_{\{t_i < r<t_{i+1}\}}2^{n} \int_{t_i}^{t_{i+1}} du\, \big\{e^{\imath \xi(u-r)}-1\big\} \bigg] \, .
\end{align*}
Now it is clear that for any $\sigma_1,\sigma_2\in [0,1]$,
\begin{align*}
&\bigg|\int_0^t dr \, \big\{\widehat{G_{t-r}}(\eta)-\widehat{G_{s-r}}(\eta)\big\} e^{\imath \xi r} \bigg[ \sum_{i=0}^{2^{n}-1} \1_{\{t_i < r<t_{i+1}\}}2^{n} \int_{t_i}^{t_{i+1}} du\, \big\{e^{\imath \xi(u-r)}-1\big\} \bigg] \bigg|\\
&\lesssim 2^{-n\sigma_1} |\xi|^{\sigma_1} \bigg[\int_0^s dr \, \big|e^{-(t-r)|\eta|^2}-e^{-(s-r)|\eta|^2} \big| +\int_s^t dr \, e^{-(t-r)|\eta|^2}\bigg]\lesssim 2^{-n\sigma_1} |\xi|^{\sigma_1} |t-s|^{\sigma_2} \{1+|\eta|^{2\sigma_2}\}  \, .
\end{align*}
Therefore, we can conclude that
\begin{equation}\label{boun-f-n-c}
\big|\mathbb{III}^n_{s,t}(\xi,\eta,\la)\big|\lesssim \big|\widehat{\rho}(\la)\big| 2^{-n\sigma_1} |\xi|^{\sigma_1} |t-s|^{\sigma_2} \{1+|\eta|^{2\sigma_2}\}\, .
\end{equation}

\

\

Injecting \eqref{estimate-i}, \eqref{fin-boun-r-1-c} and \eqref{boun-f-n-c} into \eqref{decompo-i-ii-iii}, we finally get the desired estimate \eqref{las-boun}.

\end{proof}

\section{Space-time discretization of the heat operator}\label{sec:space-discret}

At this point of the analysis, we are endowed with the approximation $\widetilde{\<Psi>}^{\ka,n}$ of the solution, derived from the discretization $\partial_t \partial_x \widetilde{B}^{\ka,n}$ of the noise (see \eqref{interpol-bipara-bis}). Otherwise stated, for all fixed $\ka>0$ and $n\geq 1$, $\widetilde{\<Psi>}^{\ka,n}$ corresponds to the classical solution of the standard heat equation
\begin{equation}\label{start-equation-space-discret}
\begin{cases}
\partial_t \<Psi>^{\ka,n}=\Delta \<Psi>^{\ka,n}+\partial_t\partial_x B^{\ka,n}\, , \quad t\in [0,1], \ x\in \R,\\
\<Psi>^{\ka,n}_0=0\, .
\end{cases}
\end{equation}

This section is devoted to the third and final step of our discretization procedure, namely the space-time discretization of the heat operator in \eqref{start-equation-space-discret}, \emph{for fixed $\ka>0,n\geq 1$} (which implies that $\partial_t \partial_x \widetilde{B}^{\ka,n}$ is here regarded as a well-defined bounded function).

\smallskip

To this end, we will rely on a Galerkin-type approximation strategy. Thus, let us start our analysis by recalling a few general principles about such a finite-element discretization method.

\subsection{A few basics about Galerkin approximation method}

Consider a general index $I$ and a set $\Phi:=(\Phi_j)_{j\in I}$, where each $\Phi_j$ is a positive compactly-supported function in $H^1(\R)$. For every $\Omega\subset \R$, denote by $I^{(\Phi,\Omega)}$ the set of indexes $j\in I$ such that $\text{supp}\, \Phi_j \subset \Omega$. Finally, for any set $\Omega\subset \R$ such that $I^{(\Phi,\Omega)}$ is finite, we define the finite-dimensional subspace $\cs^{(\Phi,\Omega)}$ as
$$\cs^{(\Phi,\Omega)}:=\text{Span}\big\{\Phi_j,\, j\in I^{(\Phi,\Omega)}\big\}.$$

\smallskip

The (space-time) Galerkin approximation operator $\cg^{(\Phi,\Omega)}_{D}$ can now be defined as follows. 

\begin{definition}
In the above setting, let $I$ and $\Omega\subset \R$ be such that $I^{(\Phi,\Omega)}$ is finite. Also, let $D=\{0=t_0<t_1<\ldots<t_M=1\}$ be a subdivisition of the time interval $[0,1]$. Then, for any function $u\in H^{(1,2)}_{loc}([0,1]\times \R)$ such that $u_0=0$, we define $\bar{u}:=\cg^{(\Phi,\Omega)}_{D}(u)$ as the set of functions
$$\bar{u}_{t_i} \in H^{(1,1)}(\R),\quad i=0,\ldots,M,$$
satisfying:

\smallskip

\noindent
$(i)$ $\bar{u}_0=0$;

\smallskip

\noindent
$(ii)$ for every $i=1,\ldots,M$, $\bar{u}_{t_i}\in \cs^{(\Phi,\Omega)}$;

\smallskip

\noindent
$(iii)$ for all $i=1,\ldots,M$ and $\vp\in \cs^{(\Phi,\Omega)}$,
\begin{equation}\label{identity-rigour}
\frac{1}{t_i-t_{i-1}}\, \langle \bar{u}_{t_i}-\bar{u}_{t_{i-1}},\vp\rangle+\langle \nabla \bar{u}_{t_i},\nabla\vp\rangle= \frac{1}{t_i-t_{i-1}}\int_{t_{i-1}}^{t_i} ds \, \langle (\partial_t u)_s-\Delta u_s,\vp\rangle \, .
\end{equation}
Here and for the rest of the section, the notation $\langle .,.\rangle$ refers to the product in $L^2(\R)$.
\end{definition}

\smallskip

\begin{remark}
Observe that identity \eqref{identity-rigour} morally coincides with the (discrete) heat equation
\begin{equation}\label{identity-galerkin}
\langle \bigg[\frac{1}{t_i-t_{i-1}}( \bar{u}_{t_i}-\bar{u}_{t_{i-1}})\bigg],\vp\rangle= \langle \bigg[\Delta \bar{u}_{t_i}+\frac{1}{t_i-t_{i-1}}\int_{t_{i-1}}^{t_i} ds\, \big[ (\partial_t u)_s-\Delta u_s\big]\bigg],\vp\rangle
 \, ,
\end{equation}
for every $\vp\in \cs^{(h,L)}$. However, elements in $\cs^{(\Phi,\Omega)}$ (such as $\bar{u}_{t_i}$) may not belong to $H^{2}(\R)$ (each $\Phi_j$ is only assumed to be in $H^1(\R)$), and therefore the quantity $\langle \Delta \bar{u}_{t_i},\vp\rangle$ can only be interpreted as $-\langle \nabla \bar{u}_{t_i},\nabla\vp\rangle$ in this setting.
\end{remark}

\begin{remark}
To be more specific in the terminology, the above-defined operation $\cg^{(\Phi,\Omega)}_{D}$ corresponds in fact to the combination of a (space) Galerkin projection and a (time) implicit Euler scheme.
\end{remark}

\

Note that, as a function with values in $\cs^{(\Phi,\Omega)}$, $\bar{u}:=\cg^{(\Phi,\Omega)}_m(u)$ can be expanded as a finite sum 
$$\bar{u}_{t_i}(x)=\sum_{j\in I^{(\Phi,\Omega)}} \bbu^j_{t_i}  \Phi_j(x)$$
for some vectors $\bbu_{t_i} \in \R^{I^{(\Phi,\Omega)}}$ ($i=0,\ldots,M$). Then, identity \eqref{identity-rigour} (applied to $\vp=\Phi_k$, $k\in I_{\Phi}^{\Omega}$) immediately gives rise to the following recursive scheme for $\bbu$: $\bbu_0=0$ and for $i=1,\ldots,m$
\begin{equation}\label{first-equation-bu}
\frac{1}{t_i-t_{i-1}}\ca_{(\Phi,\Omega)} (\bbu_{t_{i}}-\bbu_{t_{i-1}})+\cb_{(\Phi,\Omega)} \bbu_{t_i}=\frac{1}{t_i-t_{i-1}}\int_{t_{i-1}}^{t_i} ds \, \langle (\partial_t u)_s-\Delta u_s,\Phi\rangle \, ,
\end{equation}
where
\begin{equation}\label{defi-ca-cb}
\ca_{(\Phi,\Omega)}:=\big( \langle \Phi_j,\Phi_k\rangle\big)_{j,k\in I^{(\Phi,\Omega)}} \quad , \quad \cb_{(\Phi,\Omega)}:=\big( \langle \nabla \Phi_j,\nabla \Phi_k\rangle\big)_{j,k\in I^{(\Phi,\Omega)}}
\end{equation}
and where, for every function $f\in L^2_{loc}(\R)$, we define the vector $\langle f,\Phi\rangle$ by
$$(\langle f,\Phi \rangle)^j:=\langle f,\Phi_j \rangle\, , \quad j\in I^{(\Phi,\Omega)} \, .$$

\


We now specialize the analysis by considering, for every $h>0$, the basis of functions $(\Phi^h_j)_{j\in \mathbb{Z}}$ given by 
\begin{equation}
\Phi_j^h(x):=
\begin{cases}
\frac{1}{h}(x-x_{j-1}) & \text{if}\ x \in [x_{j-1},x_j]\\
\frac{1}{h}(x_{j+1}-x) & \text{if} \ x\in [x_j,x_{j+1}]\\
0& \text{otherwise} \, ,
\end{cases}
\end{equation}
where the subdivision points $(x_j)_{j\in \mathbb{Z}}$ are here merely defined as $x_j:=j h$. Besides, for every $L>0$, we will focus on the elementary domain $\Omega=\Omega_L:=[-L,L]$, and on the uniform subdivision $D_m:=\{t_i=t_i^m:=\frac{i}{2^m}, \, i=0,\ldots,2^m\}$ of $[0,1]$.

\smallskip

Thus, from now on and for the rest of the section, we restrict our attention to the Galerkin approximation operator $\cg^{(h,L)}_m$ associated with the basis $\Phi^h:=(\Phi^h_{j})_{j\in \mathbb{Z}}$, the domain $\Omega_L$ and the subdivision $D_m$. With the above notation, this corresponds to taking $\cg^{(h,L)}_m:=\cg^{(\Phi^h,\Omega_L)}_{D_m}$. In the same way, we denote by $\ca_{h,L}$ and $\cb_{h,L}$the matrices introduced in \eqref{defi-ca-cb}, when specialized to the basis $\Phi^h$ and the domain $\Omega_L$, and set $\cs^{(h,L)}:=\cs^{(\Phi^h,\Omega_L)}$.

\smallskip

{\it We will always assume that $L=N  h$ for some positive integer $N$, so that $I^{(h,L)}:=I^{(\Phi^h,\Omega_L)}=\{-N+1,...,N-1\}$.}

\

For the sake of clarity, let us rewrite the characterization of the Galerkin approximation in this setting. Namely, if $\bar{u}:=\cg^{(h,L)}_m(u)$, then $\bar{u}_0=0$, $\bar{u}_{t_i}\in \cs^{(h,L)}$ and for all $i=1,\ldots,2^m$, $\vp\in \cs^{(h,L)}$, 
\begin{equation}\label{charac-g-h-l-m}
2^m \langle \bar{u}_{t_i}-\bar{u}_{t_{i-1}},\vp\rangle+\langle \nabla \bar{u}_{t_i},\nabla \vp\rangle=2^m\int_{t_{i-1}}^{t_i} ds \, \big[\langle (\partial_t u)_s,\vp\rangle+\langle \nabla u_s,\nabla \vp\rangle \big]\, .
\end{equation}
Besides, if $\bar{u}_{t_i}(x)=\sum_{j=-N+1}^{N-1} \bbu^j_{t_i}  \Phi_j(x)$, then $\bbu$ satisfies the relation
\begin{equation}\label{equation-bbu}
2^m\ca_{h,L} (\bbu_{t_{i}}-\bbu_{t_{i-1}})+\cb_{h,L} \bbu_{t_i}=2^m\int_{t_{i-1}}^{t_i} ds \, \langle (\partial_t u)_s-\Delta u_s,\Phi^h\rangle \, .
\end{equation}

To go further with the analysis of \eqref{charac-g-h-l-m}-\eqref{equation-bbu}, observe that if $\bv=(\bv^j)_{-N+1\leq j\leq N-1}$, then, setting $v:=\sum_{j=-N+1}^{N-1} \bv^j \Phi_j^h$, one has
\begin{equation}\label{correspond-a-b}
\|v\|_{L^2(\R)}^2=\langle \ca_{h,L}\bv , \bv\rangle_2  \quad \text{and} \quad \|\nabla v\|_{L^2(\R)}^2=\langle \cb_{h,L}\bv , \bv\rangle_2 \, ,
\end{equation}
where we set $\langle \bu,\bv\rangle_2:=\sum_{j=-N+1}^{N-1} \bu^j\bv^j $ for all vectors $\bu=(\bu^j)_{-N+1\leq j\leq N-1},\bv=(\bv^j)_{-N+1\leq j\leq N-1}$. This shows in particular that $\ca_{h,L}$ is a symmetric positive definite matrix, and accordingly it can be decomposed as $\ca_{h,L}=\ce_{h,L}^\ast \ce_{h,L}$, for some lower triangular matrix $\ce_{h,L}=\big( \big(\ce_{h,L}\big)_{jk}\big)_{j,k\in I^{(h,L)}}$ with positive diagonal entries. With this notation in hand, the following elementary representation result (which applies in particular to \eqref{charac-g-h-l-m}-\eqref{equation-bbu}) will prove to be useful in the sequel.

\begin{lemma}\label{lem:solu-expli}
Fix $(h,L)$ such that $N:=\frac{L}{h}\in \mathbb{N}$. Let $\theta_{t_i}\in \cs^{(h,L)}$ ($i=0,\ldots,2^m$) be such that $\theta_0=0$ and for all $i=1,\ldots,2^m$, $\vp\in \cs^{(h,L)}$,
\begin{equation}\label{equat-theta}
2^m\langle \theta_{t_i}-\theta_{t_{i-1}},\vp\rangle+\langle \nabla \theta_{t_i},\nabla \vp\rangle=2^m\langle f_i,\vp\rangle \, ,
\end{equation}
for some functions $f_i\in L^2_{loc}(\R)$.

\smallskip

Then, if $\theta_{t_i}(x)=\sum_{j=-N+1}^{N-1} \pmb{\theta}^j_{t_i} \Phi_j^h(x)$, the vector $\pmb{\theta}_{t_i}=(\pmb{\theta}^j_{t_i})_{-N+1\leq j\leq N-1}$ is explicitly given by 
\begin{equation}\label{repres-form-galer}
\pmb{\theta}_{t_i}=\ce_{h,L}^{-1}\sum_{k=1}^i \big[(I+2^{-m}\cm_{h,L})^{-1}\big]^{i-k+1}(\ce_{h,L}^\ast)^{-1} \langle f_k,\Phi^h \rangle \, ,
\end{equation}
where we have set
$$\cm_{h,L}:=(\ce_{h,L}^{-1})^\ast \cb_{h,L} \ce_{h,L}^{-1} \, .$$
\end{lemma}

\begin{proof}
Equation \eqref{equat-theta} can be recast as
$$2^m \ca_{h,L} (\pmb{\theta}_{t_i}-\pmb{\theta}_{t_{i-1}})+\cb_{h,L}\pmb{\theta}_{t_i}=2^m \langle f_i,\Phi^h \rangle \, ,$$
and so one has the recursive formula
\begin{align*}
\pmb{\theta}_{t_i}& =\big[(\ca_{h,L}+2^{-m}\cb_{h,L})^{-1}\ca_{h,L}\big] \pmb{\theta}_{t_{i-1}}+(\ca_{h,L}+2^{-m}\cb_{h,L})^{-1} \langle f_i,\Phi^h \rangle\\
&=\big[\ce_{h,L}^{-1}(I+2^{-m}\cm_{h,L})^{-1}\ce_{h,L}\big] \pmb{\theta}_{t_{i-1}}+\big[\ce_{h,L}^{-1}(I+2^{-m}\cm_{h,L})^{-1}(\ce_{h,L}^\ast)^{-1}\big]  \langle f_i,\Phi^h \rangle \, .
\end{align*}
Since $\pmb{\theta}_0=0$, we get that
$$\pmb{\theta}_{t_i} =\sum_{k=1}^i \big[\ce_{h,L}^{-1}(I+2^{-m}\cm_{h,L})^{-1}\ce_{h,L}\big]^{i-k} \big[\ce_{h,L}^{-1}(I+2^{-m}\cm_{h,L})^{-1}(\ce_{h,L}^\ast)^{-1}\big]  \langle f_k,\Phi^h \rangle$$
and formula \eqref{repres-form-galer} immediately follows.
\end{proof}

\

With the above considerations in mind, and as a preliminary technical step towards Propositions \ref{prop:galerkin-std-s-t} and \ref{prop:gene-galer-s-t}, let us now collect some basic information related to the matrices at the core of this Galerkin approximation procedure. To fix notation, we set, for all vectors $\bu=(\bu^j)_{-N+1\leq j\leq N-1},\bv=(\bv^j)_{-N+1\leq j\leq N-1}$, and for every matrix $A=(A_{jk})_{-N+1\leq j,k\leq N-1}$,
$$\langle \bu,\bv\rangle_2:=\sum_{j=-N+1}^{N-1} \bu^j\bv^j \, , \quad \|\bv\|_2^2:=\langle \bv,\bv\rangle_2 \, , \quad \|A\|_{2;2}:=\sup_{\bv \neq 0} \frac{\|A\bv\|_2}{\|\bv\|_2} \, .$$
We recall the identity $\|A\|_{2;2}^2=\la_{\max}(A A^\ast)$, where $\la_{\max}(M)$ refers to the largest eigenvalue of $M$. 

\smallskip

\begin{lemma}\label{lem:encadr-h}
Fix $(h,L)$ be such that $N:=\frac{L}{h}\in \mathbb{N}$. Let $(\bv^j)_{i=-N+1,\ldots N-1}$ be a real-valued vector and set $v(x):=\sum_{j=-N+1}^{N-1} \bv^j \, \Phi^h_j(x)$ for every $x\in \R$. Then it holds that
\begin{equation}\label{encadr-h}
\frac{h}{3} \|\bv\|^2_2 \leq \|v\|^2_{L^2(\R)} \leq h \|\bv\|^2_2 
\end{equation} 
In particular, one has $\text{Sp}(\ca_{h,L})\subset \big[\frac{h}{3},h\big]$ and 
\begin{equation}\label{bou-e}
\big\| \ce_{h,L}^{-1}\big\|_{2;2}^2 \leq \frac{3}{h}\, .
\end{equation}
\end{lemma}

\begin{proof}
It is readily checked that
$$\langle \Phi^h_j,\Phi^h_k\rangle=\frac23 h \Big[ \1_{\{j=k\}} +\frac14\1_{\{|j-k|=1\}} \Big]$$
and accordingly
\begin{align*}
\|v\|^2&=\sum_{j,k=-N+1}^{N-1}\bv^j \bv^k \langle \Phi^h_j,\Phi^h_k\rangle\\
&=\frac23 h\bigg\{\|\bv\|_2^2+\frac14\bigg[ \sum_{j=-N+2}^{N-2}\Big[ \bv^j \bv^{j-1}+ \bv^j \bv^{j+1}\Big]+ \Big[  \bv^{-N+1} \bv^{-N+2}+ \bv^{N-2} \bv^{N-1}\Big]\bigg]\bigg\} \, .
\end{align*}
Combining the elementary bounds
$$ \big|\bv^j \bv^{j-1}+ \bv^j \bv^{j+1}\big| \leq \frac12 \Big[ 2(\bv^j)^2+(\bv^{j-1})^2 +(\bv^{j+1})^2\Big]$$
and
$$\big|\bv^{-N+1} \bv^{-N+2}+ \bv^{N-2} \bv^{N-1}\big| \leq \frac12 \Big[ (\bv^{-N+1})^2+(\bv^{-N+2})^2 +(\bv^{N-2})^2+(\bv^{N-1})^2\Big] \, ,$$
we get that
$$-\frac12\|v\|_2^2\leq \frac14\bigg[ \sum_{j=-N+2}^{N-2}\Big[ \bv^j \bv^{j-1}+ \bv^j \bv^{j+1}\Big]+ \Big[  \bv^{-N+1} \bv^{-N+2}+ \bv^{N-2} \bv^{N-1}\Big]\bigg] \leq \frac12\|v\|_2^2 \, ,$$
and  the controls in \eqref{encadr-h} immediately follow.

\smallskip



Once endowed with \eqref{encadr-h}, the inclusion $\text{Sp}(\ca_{h,L})\subset \big[\frac{h}{3},h\big]$ becomes a straightforward consequence of the first identity in \eqref{correspond-a-b}. In particular, $\la_{\min}(\ca_{h,L})\geq \frac{h}{3}$, and accordingly
$$\big\| \ce_{h,L}^{-1}\big\|_{2;2}^2=\la_{\max}\big( (\ce_{h,L}^\ast\ce_{h,L})^{-1}\big)=\la_{\max}(\ca_{h,L}^{-1})=\frac{1}{\la_{\min}(\ca_{h,L})}\leq \frac{3}{h}\, .$$
\end{proof}

\begin{lemma}\label{lem:bounds-galerkin-discr}
For all $h,L>0$ such that $N:=\frac{L}{h}\in \mathbb{N}$, the matrix $\cm_{h,L}$ is positive. As a consequence, for every $h>0$, it holds that 
\begin{equation}\label{bou-exp-discr}
\sup_{L>0:\, \frac{L}{h}\in \N }\sup_{\ell \geq 0} \, \big\|  \big[(I+2^{-m}\cm_{h,L})^{-1}\big]^{\ell}\big\|_{2;2} \leq 1 \, .
\end{equation}
Moreover, for all vectors $\bv_{t_1},\ldots, \bv_{t_{2^m}}$, for all $i=1,\ldots,2^m$ and $\varepsilon >0$, one has
\begin{equation}\label{bou-deriv-exp-discr}
\sup_{L>0:\, \frac{L}{h}\in \N } \,\bigg\| \sum_{k=1}^{i-1} \Big[\big[(I+2^{-m}\cm_{h,L})^{-1}\big]^{k}-\big[(I+2^{-m}\cm_{h,L})^{-1}\big]^{k+1}\Big]\bv_{t_{i-k}} \bigg\|_2\lesssim 2^{m\varepsilon}\sup_{k=1,\ldots,2^m} \|\bv_{t_{k}}\|_2\, ,
\end{equation}
where the proportional constant does not depend on $h$ and $i$.
\end{lemma}

\begin{proof}

Let $\la\in \R$ and $\bv \neq 0$ such that $\cm_{h,L} \bv =\la \bv$. Setting $\bw:=\ce_{h,L}^{-1}\bv$, this relation can also be read as $\cb_{h,L} \bw=\la \, \ca_{h,L} \bw$. Now define the function $w:=\sum_{i=-N+1}^{N-1} \bw^i \, \Phi^h_i$, and observe that, by \eqref{correspond-a-b},
$$0\leq \|\nabla w\|_{L^2(\R)}^2=\langle \bw,\cb_{h,L} \bw\rangle_2=\la\, \langle \bw , \ca_{h,L} \bw\rangle_2=\la \|w\|_{L^2(\R)}^2,$$
and hence $\la \geq 0$, which proves that $\cm_{h,L}$ is positive.

\smallskip

As a direct consequence of this positivity, there exists an orthonormal basis of eigenvectors $(\mathbf{e}_j)_{-N+1\leq j\leq N-1}$ associated with eigenvalues $\mu_i\geq 0$ of the matrix. One can thus write, for every vector $\bv=(\bv^j)_{-N+1\leq j\leq N-1}$,
$$ \big[(I+2^{-m}\cm_{h,L})^{-1}\big]^{\ell} \bv=\sum_{j=-N+1}^{N-1}  \bigg[\frac{1}{1+2^{-m}\mu_j}\bigg]^\ell \langle \bv,\mathbf{e}_j\rangle_2 \, \mathbf{e}_j  \, ,$$
and so 
$$\Big\|\big[(I+2^{-m}\cm_{h,L})^{-1}\big]^{\ell} \bv\Big\|_2^2=\sum_{j=-N+1}^{N-1}  \bigg|\frac{1}{1+2^{-m}\mu_j}\bigg|^{2\ell}\big|\langle\bv,\mathbf{e}_j\rangle_2\big|^2  \leq \sum_{j=-N+1}^{N-1} \big|\langle \bv,\mathbf{e}_j\rangle_2\big|^2=\|\bv\|_2^2 \, .$$
In the same way, 
\begin{align*}
&\sum_{k=1}^{i-1} \Big[\big[(I+2^{-m}\cm_{h,L})^{-1}\big]^{k}-\big[(I+2^{-m}\cm_{h,L})^{-1}\big]^{k+1}\Big]\bv_{t_{i-k}} \\
&=\sum_{k=1}^{i-1} \sum_{j=-N+1}^{N-1}  \bigg[\bigg(\frac{1}{1+2^{-m}\mu_j}\bigg)^k-\bigg(\frac{1}{1+2^{-m}\mu_j}\bigg)^{k+1}\bigg]\langle \bv_{t_{i-k}},\mathbf{e}_j\rangle_2\, \mathbf{e}_j\\
&=\sum_{j=-N+1}^{N-1}\bigg[\bigg(1-\frac{1}{1+2^{-m}\mu_j}\bigg) \sum_{k=1}^{i-1}  \bigg(\frac{1}{1+2^{-m}\mu_j}\bigg)^k \langle \bv_{t_{i-k}},\mathbf{e}_j\rangle_2\bigg] \mathbf{e}_j  \, ,
\end{align*}
and thus
\begin{align*}
&\bigg\| \sum_{k=1}^{i-1} \Big[\big[(I+2^{-m}\cm_{h,L})^{-1}\big]^{k}-\big[(I+2^{-m}\cm_{h,L})^{-1}\big]^{k+1}\Big]\bv_{t_{i-k}} \bigg\|_2^2\\
&=\sum_{j=-N+1}^{N-1} \bigg|1-\frac{1}{1+2^{-m}\mu_j}\bigg|^2 \bigg|\sum_{k=1}^{i-1}  \bigg(\frac{1}{1+2^{-m}\mu_j}\bigg)^k\langle\bv_{t_{i-k}},\mathbf{e}_j\rangle_2\bigg|^2\\
&\leq \sum_{j=-N+1}^{N-1} \bigg|1-\frac{1}{1+2^{-m}\mu_j}\bigg|^2 \bigg(\sum_{k=1}^{i-1} k^{1+2\varepsilon} \bigg(\frac{1}{1+2^{-m}\mu_j}\bigg)^{2k}\bigg)\bigg(\sum_{\ell=1}^{i-1}\frac{1}{\ell^{1+2\varepsilon}}  \big|\langle \bv_{t_{i-\ell}},\mathbf{e}_j\rangle_2\big|^2 \bigg) \, .
\end{align*}
At this point, observe that for every $j$,
\begin{align*}
&\sum_{k=1}^{i-1} k^{1+2\varepsilon} \bigg(\frac{1}{1+2^{-m}\mu_j}\bigg)^{2k}\leq 2^{2m\varepsilon}\sum_{k=1}^{\infty} k \bigg(\frac{1}{1+2^{-m}\mu_j}\bigg)^{k} \leq 2^{2m\varepsilon}\bigg|1-\frac{1}{1+2^{-m}\mu_j}\bigg|^{-2}\, ,
\end{align*}
and so we get that
\begin{align*}
&\bigg\| \sum_{k=1}^{i-1} \Big[\big[(I+2^{-m}\cm_{h,L})^{-1}\big]^{k}-\big[(I+2^{-m}\cm_{h,L})^{-1}\big]^{k+1}\Big]\bv_{t_{i-k}} \bigg\|_2^2\\
&\leq 2^{2m\varepsilon}\sum_{j=-N+1}^{N-1} \bigg(\sum_{\ell=1}^{i-1}\frac{1}{\ell^{1+2\varepsilon}}  \big|\langle \bv_{t_{i-\ell}},\mathbf{e}_j\rangle_2\big|^2 \bigg) \leq 2^{2m\varepsilon}\sum_{\ell=1}^{i-1}\frac{1}{\ell^{1+2\varepsilon}}  \|\bv_{t_{i-\ell}}\|_2^2 \lesssim 2^{2m\varepsilon} \Big(\sup_{k=1,\ldots,2^m} \|\bv_{t_{k}}\|_2\Big)^2 \, ,
\end{align*}
which corresponds to the desired estimate \eqref{bou-deriv-exp-discr}.

\

\end{proof}

\subsection{The case of vanishing functions on the boundary}

We first focus on the situation where the function $u$ under consideration vanishes on the boundary $\{-L,L\}$ of $\Omega_L$. In this case, a possible estimate for the Galerkin approximation goes as follows.

\begin{proposition}\label{prop:galerkin-std-s-t}
For all $h,L>0$ such that $N:=\frac{L}{h}\in \mathbb{N}$, all $u\in H^{(1,2)}([0,1]\times \Omega_L) \cap L^\infty ([0,1];H_0^1(\Omega_L)\cap H^2(\Omega_L))$ such that $u_0=0$, all $m\geq 1$ and $\varepsilon >0$, one has
\begin{equation}\label{estim:galerkin-std-s-t}
\sup_{i=1,\ldots,2^m} \big\|u_{t_i}-\cg^{(h,L)}_m(u)_{t_i} \big\|_{L^2(\Omega_L)} \lesssim 2^{-m(1-\varepsilon)} \sup_{t\in [0,1]} \|(\partial_t u)_t\|_{L^2(\Omega_L)}+2^{m\varepsilon}h^2 \sup_{t\in [0,1]}\|\Delta u_t\|_{L^2(\Omega_L)} \, ,
\end{equation}
where the proportional constant does not depend on $h$, $L$ and $m$.
\end{proposition}

Although the above proposition can more or less be considered as a standard result in PDE approximation theory (see for instance \cite[Theorem 8.2]{johnson}), we prefer to provide the details behind the estimate \eqref{estim:galerkin-std-s-t}, for the reader's convenience first, and also to insist on the independence of the proportional constant with respect to $L$ (such a property is not exactly clear from the statement of \cite[Theorem 8.2]{johnson}).
 
\

Let us set the stage for the proof by introducing some additional material and technical lemmas. For every $w\in H_0^1(\Omega_L)$, we define $\cp^{(h,L)}(w)$ as the orthogonal projection of $w$ on $\cs^{(h,L)}$ with respect to the product
$$\llangle w_1,w_2\rrangle:=\langle \nabla w_1,\nabla w_2\rangle\, .$$
We denote by $\vvvert.\vvvert$ the seminorm related to this product, that is $\vvvert w \vvvert^2:=\llangle w,w\rrangle$. 

\begin{lemma}\label{lem:cq-h-l}
Fix $h,L>0$ such that $N:=\frac{L}{h}\in \mathbb{N}$, and let $\cq^{(h,L)}: L^2(\Omega_L) \to \cs^{(h,L)}$ be the operator defined for every $w\in L^2(\Omega_L)$ by 
$$\cq^{(h,L)}(w):=\sum_{j=-N+1}^{N-1} w(x_j) \, \Phi_j^h \, .$$
Then, for every $w\in H_0^1(\Omega_L)\cap H^2(\Omega_L)$, it holds that
\begin{equation}\label{estim-cq-h-l}
\vvvert w-\cq^{(h,L)}(w) \vvvert \lesssim h \, \|\Delta w\|_{L^2(\Omega_L)} \, ,
\end{equation}
where the proportional constant does not depend on $h$ and $L$.
\end{lemma}

\begin{proof}
Going back to the definition of the functions $\Phi_j^h$, we get that for all $-N+1 \leq k\leq N-2$ and $x\in [x_k,x_{k+1}]$, 
\begin{align*}
\nabla w(x)-\sum_{j=-N+1}^{N-1}w(x_j)\nabla\Phi_j^h(x)&=\nabla w(x)-\big[w(x_k) \nabla\Phi_k^h(x)+w(x_{k+1}) \nabla \Phi_{k+1}^h(x)\big]\\
&=\nabla w(x)-\frac{1}{h} \big\{w(x_{k+1})-w(x_k)\big\} \, .
\end{align*}
Since $w$ is assumed to vanish on $\{-x_N,x_N\}$, the latter identity can actually be extended to $k\in \{-N,N-1\}$ and $x\in [x_k,x_{k+1}]$. For instance, if $x\in [x_{N-1},x_N]$,
$$\nabla w(x)-\sum_{j=-N+1}^{N-1}w(x_j)\nabla\Phi_j^h(x)=\nabla w(x)+\frac{1}{h} w(x_{N-1})=\nabla w(x)-\frac{1}{h}\big\{w(x_N)-w(x_{N-1})\big\} \, .$$
Thus, for all $-N \leq k\leq N-1$ and $x\in [x_k,x_{k+1}]$,
\begin{align*}
\bigg|\nabla w(x)-\sum_{j=-N+1}^{N-1}w(x_j)\nabla\Phi_j^h(x) \bigg|^2&=\bigg| \frac{1}{h}\int_{x_k}^{x_{k+1}} dy \, \big\{\nabla w(x)-\nabla w(y)\big\}\bigg|^2\\
&\leq 2 \bigg[\bigg| \frac{1}{h}\int_{x_k}^{x} dy\int_y^x dz \,\Delta w(z)\bigg|^2+\bigg| \frac{1}{h}\int_{x}^{x_{k+1}} dy\int_x^y dz \, \Delta w(z)\bigg|^2\bigg]\\
& \lesssim h \int_{x_k}^{x_{k+1}} dz \, |\Delta w(z)|^2 \, .
\end{align*} 
Since
$$\vvvert w-\cq^{(h,L)}(w) \vvvert^2=\sum_{k=-N}^{N-1} \int_{x_k}^{x_{k+1}} dx\, \bigg|\nabla w(x)-\sum_{j=-N+1}^{N-1}w(x_j)\nabla\Phi_j^h(x) \bigg|^2 \, .$$
the desired estimate immediately follows.
\end{proof}

Lemma \ref{lem:cq-h-l} will help us to deduce the following control on the projection $\cp^{(h,L)}$. 

\begin{lemma}\label{lem:theta}
Fix $h,L>0$ such that $N:=\frac{L}{h}\in \mathbb{N}$. For every $w\in H_0^1(\Omega_L)\cap H^2(\Omega_L)$, it holds that
\begin{equation}\label{estim-cp}
\big\|w-\cp^{(h,L)}(w)\big\|_{L^2(\Omega_L)} \lesssim h^2 \|\Delta w\|_{L^2(\Omega_L)}\, ,
\end{equation}
for some proportional constant that does not depend on $h$ and $L$.
\end{lemma}

\begin{proof}
Note first that by the definition of $\cp^{(h,L)}$ as an orthogonal projector on $\cs^{(h,L)}$, and using the operator $\cq^{(h,L)}$ introduced in Lemma \ref{lem:cq-h-l}, one has for every $z\in H_0^1(\Omega_L)\cap H^2(\Omega_L)$
\begin{equation}\label{estim-triple-norm}
\vvvert z-\cp^{(h,L)}(z) \vvvert\leq \vvvert z-\cq^{(h,L)}(z) \vvvert \lesssim h \|\Delta z\|_{L^2(\Omega_L)} \, ,
\end{equation}
where we have applied \eqref{estim-cq-h-l} to get the last estimate, with a proportional constant that does not depend on $L$ and $h$.

\smallskip

Then observe that since $\Omega_L$ is a compact domain, we can find $v\in H_0^1(\Omega_L)\cap H^2(\Omega_L)$ such that $-\Delta v=w-\cp^{(h,L)}(w)$, and accordingly
$$
\big\|w-\cp^{(h,L)}(w)\big\|_{L^2(\Omega_L)}^2=-\langle w-\cp^{(h,L)}(w),\Delta v\rangle=\llangle w-\cp^{(h,L)}(w), v\rrangle \, ,
$$
where the last identity follows from the fact that $w-\cp^{(h,L)}(w)$ vanishes on $\partial \Omega_L$. Using again the definition of $\cp^{(h,L)}$, we can here write
$$\llangle w-\cp^{(h,L)}(w), v\rrangle =\llangle w-\cp^{(h,L)}(w), v-\cp^{(h,L)}(v)\rrangle $$
and therefore
$$
\big\|w-\cp^{(h,L)}(w)\big\|_{L^2(\Omega_L)}^2=\llangle w-\cp^{(h,L)}(w), v-\cp^{(h,L)}(v)\rrangle \leq \vvvert w-\cp^{(h,L)}(w) \vvvert \vvvert v-\cp^{(h,L)}(v) \vvvert \, .
$$
Since both $w$ and $v$ belong to $H_0^1(\Omega_L)\cap H^2(\Omega_L)$, we can apply \eqref{estim-triple-norm} and get that 
$$
\big\|w-\cp^{(h,L)}(w)\big\|_{L^2(\Omega_L)}^2\lesssim h^2 \|\Delta w\|_{L^2(\Omega_L)} \|\Delta v\|_{L^2(\Omega_L)} \lesssim h^2 \|\Delta w\|_{L^2(\Omega_L)} \|w-\cp^{(h,L)}(w)\|_{L^2(\Omega_L)}\, ,
$$
with proportional constants that do not depend on $L,h$. The desired estimate \eqref{estim-cp} immediately follows.

\end{proof}

Let us now introduce the space-time projection operator $\mathscr{P}_m^{(h,L)}$ along the formula: for all $u\in L^\infty ([0,1];H_0^1(\Omega_L))$ such that $u_0=0$, 
\begin{equation}
\mathscr{P}_m^{(h,L)}(u)_0=0 \quad \text{and}\quad \mathscr{P}_m^{(h,L)}(u)_{t_i}:=2^m\int_{t_{i-1}}^{t_i} ds \, \cp^{(h,L)}(u_s) \quad \text{for}\ i=1,\ldots,2^m\, .
\end{equation}

\begin{lemma}\label{lem:p-rond}
Fix $h,L>0$ such that $N:=\frac{L}{h}\in \mathbb{N}$. For all $u\in H^{(1,2)}([0,1]\times \Omega_L) \cap L^\infty ([0,1];H_0^1(\Omega_L)\cap H^2(\Omega_L))$, $i=1,\ldots,2^m$ and $\vp\in \cs^{(h,L)}$, it holds that
\begin{equation}\label{identity-mscrp}
\langle \nabla(\mathscr{P}^{(h,L)}_m(u)_{t_i}),\nabla\vp\rangle=2^m\int_{t_{i-1}}^{t_i} ds \, \langle\nabla u_s,\nabla \vp\rangle
\end{equation}
and
\begin{equation}\label{estim:mscrp}
\sup_{i=1,\ldots,2^m}\big\|u_{t_i}-\mathscr{P}_m^{(h,L)}(u)_{t_i}\big\|\lesssim 2^{-m} \sup_{t\in [0,1]} \|(\partial_s u)_t\|+h^2 \sup_{t\in [0,1]}\|\Delta u_t\| \, ,
\end{equation}
where the proportional constant does not depend on $h$, $L$ and $m$.
\end{lemma}

\begin{proof}
Recall that for every $s\in [0,1]$ and $\vp\in \cs^{(h,L)}$, $\langle\nabla \cp^{(h,L)}(u_s), \nabla \vp\rangle=\langle \nabla u_s,\nabla \vp\rangle$, which immediately entails \eqref{identity-mscrp}. Then write
$$u_{t_i}-\mathscr{P}_m^{(h,L)}(u)=2^m \int_{t_{i-1}}^{t_i}ds \, \big[u_{t_i}-u_s\big]+2^m\int_{t_{i-1}}^{t_i}ds\, \big[u_s-\cp^{(h,L)} (u_s)\big]\, ,$$
so that
$$\big\|u_{t_i}-\mathscr{P}_m^{(h,L)}(u)_{t_i}\big\|\lesssim 2^m \int_{t_{i-1}}^{t_i}ds \int_{s}^{t_i}dr \|(\partial_s u)_r\|+2^m\int_{t_{i-1}}^{t_i}ds\,  \big\| u_s-\cp^{(h,L)}(u_s)\big\|$$
and the estimate \eqref{estim:mscrp} is now a straightforward consequence of \eqref{estim-cp}.

\end{proof}

We are now in a position to prove the main result of this subsection.

\begin{proof}[Proof of Proposition \ref{prop:galerkin-std-s-t}]
For every $i=0,\ldots,2^m$, let us decompose the difference $u_{t_i}-\cg^{(h,L)}_m(u)_{t_i}$ as 
\begin{equation}\label{decompo-ga-theta-s-t}
u_{t_i}-\cg^{(h,L)}_m(u)_{t_i}=\big[u_{t_i}-\mathscr{P}^{(h,L)}_m(u)_{t_i}\big]+\big[\mathscr{P}^{(h,L)}_m(u)_{t_i}-\cg^{(h,L)}_m(u)_{t_i}\big]=:(\ga^{(h,L)}_m)_{t_i}+(\theta^{(h,L)}_m)_{t_i} \, .
\end{equation}
Thanks to Lemma \ref{lem:p-rond}, we already know that
\begin{equation}\label{estim-gamma-s-t}
\sup_{i=1,\ldots,2^m} \big\|(\ga^{(h,L)}_m)_{t_i}\big\|_{L^2(\Omega_L)} \lesssim 2^{-m} \sup_{t\in [0,1]} \|(\partial_s u)_t\|+h^2 \sup_{t\in [0,1]}\|\Delta u_s\|\, ,
\end{equation}
for some proportional constant that does not depend on $h$, $L$ and $m$.

\

Let us now turn to the estimate for $(\theta^{(h,L)}_m)_{t_i}$ ($i=1,\ldots,2^m$), and to alleviate the notation, let us write from now on $\ga,\theta,\mathscr{P},\cg$ instead of $\ga^{(h,L)}_m,\theta^{(h,L)}_m,\mathscr{P}^{(h,L)}_m,\cg^{(h,L)}_m$, respectively. 

\smallskip

Given $\vp\in \cs^{(h,L)}$, observe that for every $i=1,\ldots,2^m$,
\begin{align*}
&2^m\langle \theta_{t_i}-\theta_{t_{i-1}},\vp\rangle+\langle \nabla \theta_{t_i},\nabla \vp\rangle\\
&=\big[2^m\langle \mathscr{P}(u)_{t_i}-\mathscr{P}(u)_{t_{i-1}},\vp\rangle+\langle \nabla(\mathscr{P}(v)_{t_i}),\nabla\vp\rangle\big]-\big[2^m\langle \cg(u)_{t_i}-\cg(u)_{t_{i-1}},\vp\rangle+\langle \nabla(\cg(v)_{t_i}),\nabla\vp\rangle\big]\\
&=\big[2^m\langle \mathscr{P}(u)_{t_i}-\mathscr{P}(u)_{t_{i-1}},\vp\rangle+\langle \nabla(\mathscr{P}(v)_{t_i}),\nabla\vp\rangle\big]-2^m\int_{t_{i-1}}^{t_i} ds \, \langle (\partial_t u)_s-\Delta u_s,\vp\rangle\\
&=2^m\big[\langle \mathscr{P}(u)_{t_i}-\mathscr{P}(u)_{t_{i-1}},\vp\rangle-\langle u_{t_i}-u_{t_{i-1}},\vp\rangle\big]+\bigg[\langle \nabla(\mathscr{P}(v)_{t_i}),\nabla\vp\rangle-2^m\int_{t_{i-1}}^{t_i} ds \, \langle\nabla u_s,\nabla \vp\rangle\bigg]
\end{align*}
and so, using identity \eqref{identity-mscrp}, we end up with the relation
\begin{equation}\label{equa-theta-s-t}
2^m\langle \theta_{t_i}-\theta_{t_{i-1}},\vp\rangle+\langle \nabla \theta_{t_i},\nabla \vp\rangle=-2^m\langle \ga_{t_i}-\ga_{t_{i-1}},\vp\rangle \, .
\end{equation}

Now recall that for every $i=1,\ldots,2^m$, $\theta_{t_i}\in \cs^{(h,L)}$, and therefore this element can be expanded $\theta_{t_i}(x)=\sum_{j=-N+1}^{N-1} \pmb{\theta}^j_{t_i} \Phi_j^h(x)$ for some vector $(\pmb{\theta}^j_{t_i})_{-N+1\leq j\leq N-1}$. With this notation in hand, and since $\theta_0=0$, we can apply Lemma \ref{lem:solu-expli} to the relation \eqref{equa-theta-s-t} and assert that for every $i=1,\ldots,2^m$,
\begin{align*}
\pmb{\theta}_{t_i}&=-\ce_{h,L}^{-1}\sum_{k=1}^i \big[(I+2^{-m}\cm_{h,L})^{-1}\big]^{i-k+1}(\ce_{h,L}^\ast)^{-1}\langle \ga_{t_k}-\ga_{t_{k-1}},\Phi^h \rangle\\
&=-\ce_{h,L}^{-1}\big[(I+2^{-m}\cm_{h,L})^{-1}\big](\ce_{h,L}^\ast)^{-1}\langle \ga_{t_i},\Phi^h \rangle\\
&\hspace{0.5cm}+\ce_{h,L}^{-1}\sum_{k=1}^{i-1} \big\{\big[(I+2^{-m}\cm_{h,L})^{-1}\big]^{i-k}-\big[(I+2^{-m}\cm_{h,L})^{-1}\big]^{i-k+1}\big\}(\ce_{h,L}^\ast)^{-1}\langle \ga_{t_k},\Phi^h \rangle \, ,
\end{align*}
where the second identity is derived from a discrete integration by parts, and the fact that $\ga_0=0$.

\smallskip

Using \eqref{bou-e} and \eqref{bou-exp-discr}, we can first assert that 
$$\big\|\ce_{h,L}^{-1}\big[(I+2^{-m}\cm_{h,L})^{-1}\big](\ce_{h,L}^\ast)^{-1}\langle \ga_{t_i},\Phi^h \rangle\big\|_2 \lesssim \frac{1}{h}\big\|\langle \ga_{t_k},\Phi^h\rangle\big\|_2 \, ,$$
Then, using \eqref{bou-e} and \eqref{bou-deriv-exp-discr}, one has for every $\varepsilon >0$,
\begin{align*}
&\bigg\|\ce_{h,L}^{-1}\sum_{k=1}^{i-1} \big\{\big[(I+2^{-m}\cm_{h,L})^{-1}\big]^{i-k}-\big[(I+2^{-m}\cm_{h,L})^{-1}\big]^{i-k+1}\big\}(\ce_{h,L}^\ast)^{-1}\langle \ga_{t_k},\Phi^h \rangle \bigg\|_2\\
&\lesssim \frac{2^{m\varepsilon}}{h^{1/2}} \sup_{k=1,\ldots,2^m} \big\|(\ce_{h,L}^\ast)^{-1}\langle \ga_{t_k},\Phi^h \rangle \big\|_2\lesssim \frac{2^{m\varepsilon}}{h} \sup_{k=1,\ldots,2^m} \big\|\langle \ga_{t_k},\Phi^h \rangle \big\|_2 \, .
\end{align*}
Therefore, we obtain that
$$\|\pmb{\theta}_{t_i}\|_2 \lesssim \frac{2^{m\varepsilon}}{h}\sup_{k=1,\ldots,2^m}\big\|\langle \ga_{t_k},\Phi^h\rangle\big\|_2 \, .$$
It is now readily checked that 
$$\big\|\langle \ga_{t_k},\Phi^h\rangle\big\|_2^2=\sum_{j=-N+1}^{N-1} \bigg| \int_{x_{j-1}}^{x_{j+1}} dx \, \ga_{t_k}(x)\Phi_j^h(x)\bigg|^2 \lesssim h \, \|\ga_{t_k}\|_{L^2(\Omega_L)}^2$$
for proportional constants independent of $h$ and $L$, and so, using \eqref{encadr-h},
\begin{equation}\label{estim-final-theta-s-t}
\|\theta_{t_i}\|_{L^2(\R)} \leq h^{1/2} \|\pmb{\theta}_{t_i}\|_2 \lesssim 2^{m\varepsilon}\sup_{k=1,\ldots,2^m}\|\ga_{t_k}\|_{L^2(\Omega_L)}\, .
\end{equation}

Going back to the decomposition \eqref{decompo-ga-theta-s-t}, the desired estimate \eqref{estim:galerkin-std-s-t} follows from the combination of \eqref{estim-gamma-s-t} and \eqref{estim-final-theta-s-t}.

\end{proof}

\subsection{The general case}

Let us now extend the result of Proposition \ref{prop:galerkin-std-s-t} to the situation where $u$ does not necessarily vanish on the boundary of $\Omega_L$.

\begin{proposition}\label{prop:gene-galer-s-t}
Fix a smooth compactly-supported function $\rho:\R\to [0,1]$. Then, for all $u\in H^{1,2}_{loc}([0,1]\times \R)$ such that $u_0=0$, all $h,L>0$ such that $N:=\frac{L}{h}\in \mathbb{N}$, all $m\geq 1$ and $\varepsilon >0$, it holds that
\begin{align}
&\sup_{i=1,\ldots,2^m}\big\| \rho \cdot \big\{ u_{t_i}- \cg^{(h,L)}_m(u)_{t_i} \big\} \big\|_{L^2(\R)} \nonumber\\
&\lesssim  \bigg[2^{-m(1-\varepsilon)} \sup_{t\in [0,1]} \|(\partial_t u)_t\|_{L^2(\R)}+2^{m\varepsilon}h^2 \sup_{t\in [0,1]}\|\Delta u_t\|_{L^2(\R)}\bigg]+\sup_{t\in [0,1]} \sup_{x\in \partial\Omega_L}\Big\{ |u_t(x)|+L^{1/2} \big|(\partial_t u)_t(x)\big|\Big\} \, ,\label{desired-galerkin-bound-s-t}
\end{align}
for some proportional constant that does not depend on $h$, $L$ and $m$.
\end{proposition}

\begin{proof}

\smallskip

First, let us introduce a function $u^L\in H^{1,2}_{loc}([0,1]\times \R)$ such that $u^L_s(-L)=u^L_s(L)=0$ and $\Delta u^L_s=\Delta u_s$. To be more specific, we take
\begin{equation}\label{expres-u-l}
u^L_s(x):=u_s(x)-\frac{1}{2L}(u_s(L)-u_s(-L))x-\frac12(u_s(-L)+u_s(L))\, .
\end{equation}
Also, in the sequel, we always consider $L$ large enough so that $\text{supp} \, \rho \subset \Omega_L$. Then one has
\begin{align}
&\sup_{i=1,\ldots,2^m}\big\| \rho \cdot \big\{ u_{t_i}- \cg^{(h,L)}_m(u)_{t_i} \big\} \big\|_{L^2(\R)}\nonumber\\
& \leq \sup_{i=1,\ldots,2^m}\big\|  u^L_{t_i}- \cg^{(h,L)}_m(u^L)_{t_i} \big\|_{L^2(\Omega_L)}+\sup_{t\in [0,1]}\big\| \rho \cdot \big\{ u_t- u^L_t\big\} \big\|_{L^2(\R)}+\sup_{i=1,\ldots , 2^m}\big\|  \cg^{(h,L)}_m(u-u^L)_{t_i} \big\|_{L^2(\Omega_L)} \, ,\label{decompo-strategy}
\end{align}
and we can estimate each of these three quantities separately.

\

To handle the first quantity, observe that the function $u^L$ satisfies the conditions of Proposition \ref{prop:galerkin-std-s-t}, and therefore, thanks to this result, one has for every $\varepsilon >0$ and every $L>0$ large enough 
\begin{align}
&\sup_{i=1,\ldots,2^m}\big\|  u^L_{t_i}- \cg^{(h,L)}_m(u^L)_{t_i} \big\|_{L^2(\Omega_L)}\nonumber\\
&\lesssim 2^{-m(1-\varepsilon)} \sup_{t\in [0,1]} \|(\partial_t u^L)_t\|_{L^2(\Omega_L)}+2^{m\varepsilon}h^2 \sup_{t\in [0,1]}\|\Delta u^L_t\|_{L^2(\Omega_L)}\nonumber\\
&\lesssim \bigg[2^{-m(1-\varepsilon)} \sup_{t\in [0,1]} \|(\partial_t u)_t\|_{L^2(\R)}+2^{m\varepsilon}h^2 \sup_{t\in [0,1]}\|\Delta u_t\|_{L^2(\R)}\bigg]+\sup_{t\in [0,1]}\sup_{x\in \partial\Omega_L} L^{1/2}\big|(\partial_t u)_t(x)\big| \, ,
\end{align}
where the proportional constants do not depend on $h$, $L$ and $m$.

\smallskip

For the second quantity in \eqref{decompo-strategy}, and with expression \eqref{expres-u-l} in mind, it is clear that 
\begin{equation}
\sup_{t\in [0,T]}\big\| \rho \cdot \big\{ u_t- u^L_t\big\} \big\|_{L^2(\R)} \lesssim \sup_{t\in [0,T]} \sup_{x\in \partial\Omega_L} |u_t(x)|\, ,
\end{equation}
where the proportional constant only depends on $\rho$.

\smallskip

As for the third quantity in \eqref{decompo-strategy}, one has, owing to \eqref{charac-g-h-l-m} and Lemma \ref{lem:solu-expli}: $\cg^{(h,L)}_m(u-u^L)_{t_i}=\sum_{j=-N+1}^{N-1} \pmb{\delta}^j_{t_i} \Phi^h_j$, with
$$\pmb{\delta}_{t_i}=\ce_{h,L}^{-1}\sum_{k=1}^i \big[(I+2^{-m}\cm_{h,L})^{-1}\big]^{i-k+1}(\ce_{h,L}^\ast)^{-1}\int_{t_{k-1}}^{t_k}ds\,  \langle \partial_t (u-u^L)_{s}, \Phi^h\rangle\, ,$$
where we have used the fact that $\langle \nabla(u_s-u^L_s),\nabla \vp\rangle=0$ for all $s\in [0,1]$ and $\vp\in \cs^{(h,L)}$.

\smallskip

Using \eqref{encadr-h}, \eqref{bou-e} and \eqref{bou-exp-discr}, we deduce that
\begin{equation}\label{third-term}
\big\|  \cg^{(h,L)}_m(u-u^L)_{t_i} \big\|_{L^2(\Omega_L)}\leq h^{1/2} \|\pmb{\delta}_{t_i}\|_2\lesssim \frac{1}{h^{1/2}} \sup_{s\in [0,1]} \|\langle \partial_t (u-u^L)_s, \Phi^h\rangle \|_2\, .
\end{equation}
Then we can observe that
\begin{align*}
\|\langle \partial_t (u-u^L)_s, \Phi^h\rangle \|_2^2
&\lesssim \bigg|\frac{1}{2L}\big\{(\partial_s u)_s(L)-(\partial_s u)_s(-L)\big\}\bigg|^2 \sum_{j=-N+1}^{N-1} \bigg| \int_{\R} dx \, x\Phi_j^h(x)\bigg|^2 \\
&\hspace{3cm}+\bigg|\frac{1}{2}\big\{(\partial_s u)_s(L)+(\partial_s u)_s(-L)\big\}\bigg|^2 \sum_{j=-N+1}^{N-1} \bigg| \int_{\R} dx \, \Phi_j^h(x)\bigg|^2\\
&\lesssim \bigg(\sup_{x\in \partial \Omega_L}\big|(\partial_s u)_s(x)\big|^2\bigg) \sum_{j=-N+1}^{N-1} \bigg( \int_{\R} dx \, |\Phi_j^h(x)|\bigg)^2\lesssim  h^2 N\sup_{x\in \partial \Omega_L}\big|(\partial_s u)_s(x)\big|^2\, .
\end{align*}
Now remember that $Nh=L$, and therefore, going back to \eqref{third-term}, we deduce
\begin{equation}
\sup_{i=1,\ldots , 2^m}\big\|  \cg^{(h,L)}_m(u-u^L)_{t_i} \big\|_{L^2(\Omega_L)} \lesssim \sup_{s\in [0,1]}\sup_{x\in \partial \Omega_L}L^{1/2}\big|(\partial_s u)_s(x)\big| \, .
\end{equation}

\end{proof}

\subsection{Application to the stochastic problem}

We now intend to apply the result of Proposition \ref{prop:gene-galer-s-t} to our problem \eqref{start-equation-space-discret}, that is to $u= \widetilde{\<Psi>}^{\ka,n}=G\ast \partial_t \partial_x \widetilde{B}^{\ka,n}$. For this application to be relevant, we naturally need, first, to find suitable bounds on the terms involved in the right-hand side of \eqref{desired-galerkin-bound-s-t}. This is the purpose of the next three lemmas.

\begin{lemma}\label{lem:firs-le}
Let $f\in L^{\infty}([0,1]\times \R)$ be a function of the form $f_s(x)=\sum_{k=-K}^{K-1} a^k_s\, \1_{[x_k,x_{k+1})}(x)$, for some $K\geq 1$, some coefficients $a^k\in L^\infty([0,1])$ and some points $x_{-K}\leq \ldots \leq x_{K}$. Setting 
$$M:=\Big(\sup_{k=-K,\ldots,K-1}|x_{k+1}-x_k|\Big)^{-1}$$
and assuming that $M\geq 1$, one has for every $\varepsilon >0$
\begin{equation}\label{appli-f-1-bis}
\sup_{t\in [0,1]} \big\| \partial_t(G\ast f)_t \big\|_{L^2(\R)}+\sup_{t\in [0,1]} \big\| \Delta(G\ast f)_t \big\|_{L^2(\R)} \lesssim \frac{K}{M^{\frac12-\varepsilon}}\sup_{t\in [0,1]} \|f_t\|_{L^\infty(\R)} \, ,
\end{equation}
where the proportional constant does not depend on $K$ and $M$.
\end{lemma}

\begin{proof}
Using Plancherel theorem, we can write
\begin{align*}
\big\| \Delta(G\ast f)_t \big\|_{L^2(\R)}^2=c \int_{\R}d\xi \, |\xi|^4 \big| \cf \big( (G\ast f)_t\big)(\xi)\big|^2  &=c \int_{\R}d\xi \, |\xi|^4 \bigg| \int_0^t ds \, e^{-\xi^2(t-s)} \big(\cf f_s\big)(\xi) \bigg|^2 \\
&\lesssim \int_{\R}d\xi \, |\xi|^{\varepsilon}  \bigg( \int_0^t \frac{ds}{(t-s)^{1-\frac{\varepsilon}{4}}} \big|\big(\cf f_s\big)(\xi)\big| \bigg)^2 \, .
\end{align*}
Then one has, for every $\la\in [0,1]$,
\begin{align*}
\big|\big(\cf f_s\big)(\xi)\big| =\bigg| \sum_{k=-K}^{K-1} a_s^k \int_{x_k}^{x_{k+1}}dx \, e^{-\imath x\xi} \bigg|&\leq \sum_{k=-K}^{K-1} |a_s^k| \bigg|\int_{x_k}^{x_{k+1}}dx \, e^{-\imath x\xi}\bigg|^\la \bigg| \int_{x_k}^{x_{k+1}}dx \, e^{-\imath x\xi} \bigg|^{1-\la}\\
&\lesssim \frac{K}{M^\la |\xi|^{1-\la}}\sup_{t\in [0,1]} \|f_t\|_{L^\infty(\R)} \, ,
\end{align*}
and so
\begin{align*}
&\big\| \Delta(G\ast f)_t \big\|_{L^2(\R)}^2\lesssim \bigg(K \sup_{t\in [0,1]} \|f_t\|_{L^\infty(\R)}\bigg)^2 \bigg[\frac{1}{M^2}\int_{|\xi|\leq 1}d\xi \, |\xi|^{\varepsilon}  + \frac{1}{M^{1-2\varepsilon}}\int_{|\xi|\geq 1} \frac{d\xi}{|\xi|^{1+\varepsilon}} \bigg]\, ,
\end{align*}
which gives the desired bound for $\sup_{t\in [0,1]} \big\| \Delta(G\ast f)_t \big\|_{L^2(\R)}$.

\smallskip

As for $\partial_t(G\ast f)_t $, recall that $\partial_t(G\ast f)_t=\Delta(G\ast f)_t+f_t$, and
$$\sup_{t\in [0,1]} \|f_t\|_{L^2} \leq \bigg(\frac{2K}{M}\bigg)^{1/2}\sup_{t\in [0,1]} \|f_t\|_{L^\infty(\R)}\lesssim \frac{K}{M^{\frac12-\varepsilon}}\sup_{t\in [0,1]} \|f_t\|_{L^\infty(\R)}.$$
\end{proof}

\begin{lemma}\label{lem:sec-le}
Let $f\in L^{\infty}([0,1]\times \R)$ be a function such that $\bigcup_{t\in [0,1]}\text{supp} \, f_t \subset [-2^n,2^n]$, for some $n\geq 1$. Then for all $\beta>0$ and $L>2^n$, it holds that
\begin{equation}\label{appli-f-2}
\sup_{t\in [0,1]} \sup_{x\in \partial\Omega_L}\Big\{ |(G\ast f)_t(x)|+L^{1/2} \big|(\partial_t (G\ast f))_t(x)\big|\Big\} \lesssim \frac{L^{1/2}}{|L-2^n|^\be}\sup_{t\in [0,T]} \|f_t\|_{L^\infty(\R)}\, , 
\end{equation}
where the proportional constant does not depend on $n$ and $L$.
\end{lemma}

\begin{proof}
For $x\in \{-L,L\}$, and for all $\varepsilon\in (0,1)$, $\beta>0$, one has
\begin{align*}
\big|(\partial_t (G\ast f))_t(x)\big|&\lesssim \int_0^t ds \int_{-2^n}^{2^n} dy \, |f_{t-s}(y)|e^{-\frac{(x-y)^2}{2s}} \bigg[ \frac{1}{s^{\frac32}}+\frac{(x-y)^2}{s^{\frac52}}\bigg]\\
&\lesssim \sup_{t\in [0,1]} \|f_t\|_{L^\infty(\R)}\int_0^t \frac{ds}{s^{\frac32-\frac{\beta}{2}}}\int_{-2^n}^{2^n} dy \, \frac{e^{-\varepsilon\frac{(x-y)^2}{2s}}}{|x-y|^{\beta}}\\
&\lesssim \frac{1}{|L-2^n|^{\beta}}\sup_{t\in [0,1]} \|f_t\|_{L^\infty(\R)}\int_0^t \frac{ds}{s^{\frac32-\frac{\beta}{2}}}\int_{\R} dy \, e^{-\varepsilon\frac{(x-y)^2}{2s}}\\
&\lesssim \frac{1}{|L-2^n|^{\beta}}\sup_{t\in [0,1]} \|f_t\|_{L^\infty(\R)}\int_0^t \frac{ds}{s^{1-\frac{\beta}{2}}}\, .
\end{align*}
The quantity $|(G\ast f)_t(x)|$ can then be estimated along similar arguments.

\end{proof}

Based on the estimates \eqref{desired-galerkin-bound-s-t} and \eqref{appli-f-1-bis}-\eqref{appli-f-2}, it remains us to exhibit a bound for the supremum norm of the approximated fractional noise.

\begin{lemma}\label{lem:trois-le}
Fix $(H_0,H_1)\in (0,1)^2$ and for all $\ka>0$, $n\geq 1$, let $\partial_t \partial_x \widetilde{B}^{\ka,n}$ be the approximated fractional noise given by \eqref{interpol-bipara-bis}. Then almost surely, and for every $\varepsilon>0$, it holds that
\begin{equation}\label{pathwise-regu}
\sup_{t\in [0,1]} \big\|(\partial_t \partial_x \widetilde{B}^{\ka,n})_t\big\|_{L^\infty(\R)}\lesssim 2^{n(2-H_0-H_1+\varepsilon)}\, ,
\end{equation}
for some (random) proportional constant that does not depend on $\ka$ and $n$.
\end{lemma}

\begin{proof}
It is essentially a direct consequence of the (almost sure) regularity of the space-time fractional sheet. To be more specific, observe that for all $(s,x),(t,y)\in [0,1]\times \R$, and for every $p\geq 1$,
\begin{align*}
&\mathbb{E} \Big[ \big|B^{\ka,n}_t(y)-B^{\ka,n}_s(y)-B^{\ka,n}_t(x)+B^{\ka,n}_s(x)\big|^{2p}\Big]\\
 &\lesssim \mathbb{E} \Big[ \big|B^{\ka,n}_t(y)-B^{\ka,n}_s(y)-B^{\ka,n}_t(x)+B^{\ka,n}_s(x)\big|^{2}\Big]^p\\
&\lesssim \mathbb{E} \bigg[ \bigg|\int_{|\xi|\leq 2^{2\ka n}}\int_{|\eta| \leq 2^{\ka n}} \widehat{W}(d\xi,d\eta) \, \frac{e^{\imath \xi t}-e^{\imath \xi s}}{|\xi|^{H_0+\frac12}}\frac{e^{\imath \eta y}-e^{\imath \eta x}}{|\eta|^{H_1+\frac12}}\bigg|^{2}\bigg]^p\\
&\lesssim \bigg( \int_{\R}\int_{\R} d\xi d\eta \, \frac{|e^{\imath \xi t}-e^{\imath \xi s}|^2}{|\xi|^{2H_0+1}}\frac{|e^{\imath \eta y}-e^{\imath \eta x}|^2}{|\eta|^{2H_1+1}}\bigg)^p\\
&\lesssim \bigg( \int_{\R}\int_{\R} d\xi d\eta \, \frac{|e^{\imath \xi (t-s)}-1|^2}{|\xi|^{2H_0+1}}\frac{|e^{\imath \eta (y-x)}-1|^2}{|\eta|^{2H_1+1}}\bigg)^p \lesssim |t-s|^{2H_0 p} |y-x|^{2H_1p} \, ,
\end{align*}
where the proportional constants do not depend on $(s,x),(t,y)$, $\ka$ and $n$. Therefore, we can apply for instance \cite[Theorem 3.1]{hu-le} to assert that almost surely, and for all $n\geq 1$, $i,j\in \mathbb{Z}$, $\varepsilon>0$,
$$\big|\! \sq_{ij}^n \! \! B^{\ka,n} \big| \lesssim 2^{-n(H_0+H_1-\varepsilon)} \, ,$$
for some (random) proportional constant that does not depend on $n$ and $i,j$. The claimed estimate \eqref{pathwise-regu} is then an immediate consequence of the definition \eqref{interpol-bipara-bis} of $\partial_t \partial_x \widetilde{B}^{\ka,n}$.

\end{proof}

With the represention
\begin{equation}\label{expansion-noise}
(\partial_t \partial_x \widetilde{B}^{\ka,n})_s(x)=\sum_{\ell=-2^{2n}}^{2^{2n}-1} \bigg(\sum_{k=0}^{2^n-1} 2^{2n}(\sq_{k\ell}^n B^{\ka,n}) \1_{[\frac{k}{2^n},\frac{k+1}{2^n}[}(s)\bigg) \1_{[\frac{\ell}{2^n},\frac{\ell+1}{2^n}[}(x)
\end{equation}
in mind, we can inject the estimates of Lemma \ref{lem:firs-le}, Lemma \ref{lem:sec-le} and Lemma \ref{lem:trois-le} into the result of Proposition \ref{prop:gene-galer-s-t}, which gives successively (for $(h,L)$ such that $\frac{L}{h}\in \mathbb{N}$)

\begin{align}
&\sup_{i=0,\ldots,2^m}\big\| \rho \cdot \big\{ \widetilde{\<Psi>}^{\ka,n}_{t_i}- \cg^{(h,L)}_m\big(\widetilde{\<Psi>}^{\ka,n}\big)_{t_i} \big\} \big\|_{L^2(\R)} \nonumber\\
&\lesssim  2^{-m(1-\varepsilon)} \sup_{t\in [0,1]} \|(\partial_t \widetilde{\<Psi>}^{\ka,n})_t\|_{L^2(\R)}+2^{m\varepsilon}h^2 \sup_{t\in [0,1]}\|\Delta \widetilde{\<Psi>}^{\ka,n}_t\|_{L^2(\R)}\nonumber\\
&\hspace{5cm}+\sup_{t\in [0,1]} \sup_{x\in \partial\Omega_L}\Big\{ |\widetilde{\<Psi>}^{\ka,n}_t(x)|+L^{1/2} \big|(\partial_t \widetilde{\<Psi>}^{\ka,n})_t(x)\big|\Big\}\nonumber\\
&\lesssim \bigg[2^{-m(1-\varepsilon)}  2^{n(\frac32+\varepsilon)}+2^{m\varepsilon}h^2 2^{n(\frac32+\varepsilon)}+\frac{L^{1/2}}{|L-2^n|^\be}\bigg] \cdot \sup_{t\in [0,1]} \|(\partial_t \partial_x \widetilde{B}^{\ka,n})_t\|_{L^\infty(\R)}\nonumber\\
&\lesssim \bigg[2^{-m(1-\varepsilon)}  2^{n(\frac32+\varepsilon)}+2^{m\varepsilon}h^2 2^{n(\frac32+\varepsilon)}+\frac{L^{1/2}}{|L-2^n|^\be}\bigg] \cdot 2^{n(2-H_0-H_1+\varepsilon)} \, ,\label{forthcoming-estimate}
\end{align}
which corresponds to the main estimate of this section.

\begin{proposition}\label{prop:discreti-s-t}
Fix $(H_0,H_1)\in (0,1)^2$ and for all $\ka>0$, $n\geq 1$, let $\widetilde{\<Psi>}^{\ka,n}$ be the solution to the equation \eqref{start-equation-space-discret}, driven by the approximated fractional noise $\partial_t \partial_x \widetilde{B}^{\ka,n}$. Also, fix a smooth compactly-supported function $\rho:\R\to [0,1]$.

\smallskip

Then for all $h>0$, $L>2^n$ such that $\frac{L}{h}\in \mathbb{N}$, for all $\ka>0$, $n\geq 1$, $m\geq 1$, and for all $\beta >0$, $\varepsilon\in (0,1)$, one has almost surely
\begin{align}
&\sup_{i=0,\ldots,2^m}\big\| \rho \cdot \big\{ \widetilde{\<Psi>}^{\ka,n}_{t_i}- \cg^{(h,L)}_m\big(\widetilde{\<Psi>}^{\ka,n}\big)_{t_i} \big\} \big\|_{L^2(\R)}\nonumber\\
&\lesssim  2^{-m(1-\varepsilon)}2^{n(\frac72-H_0-H_1+\varepsilon)}+  2^{m\varepsilon}h^{2}2^{n(\frac72-H_0-H_1+\varepsilon)}+\frac{L^{1/2}}{|L-2^n|^\be}2^{n(2-H_0-H_1+\varepsilon)} \, ,\label{speed-conv-2}
\end{align}
where the proportional constant does not depend on $\ka$, $n$, $h$, $L$ and $m$.

\smallskip

Moreover, if $h:=2^{-2n}$, $L:=2^{n+1}$ and $m:=4n$, then $\cg^{(h,L)}_m\big(\widetilde{\<Psi>}^{\ka,n}\big)=\bar{\<Psi>}^{\ka,n}$ (where we recall that $\bar{\<Psi>}^{\ka,n}$ stands for the approximation process described in Section \ref{subsec:descript-scheme}), and accordingly, for every $(H_0,H_1)\in (0,1)^2$, one has almost surely
\begin{equation}\label{appli-calib}
\sup_{i=0,\ldots,2^m}\big\| \rho \cdot \big\{ \widetilde{\<Psi>}^{\ka,n}_{t_i}-\bar{\<Psi>}^{\ka,n}_{t_i} \big\} \big\|_{L^2(\R)} \lesssim 2^{-\frac{n}{2}} \, ,
\end{equation}
where the proportional constant does not depend on $\ka$ and $n$.
\end{proposition}

\smallskip

\begin{remark}
The estimate \eqref{speed-conv-2} emphasizes a standard feature of the finite-element method (when applied in a parabolic setting), namely the fact that the time mesh size (i.e. $2^{-m}$) must somehow be considered at the same level as the square of the spatial mesh size (i.e. $h^2$) in the discretization procedure. 
\end{remark}

\smallskip

\begin{proof}[Proof of Proposition \ref{prop:discreti-s-t}]
We have already shown the estimate \eqref{speed-conv-2} (see \eqref{forthcoming-estimate}), and therefore it only remains us to check that for the particular choice $h:=2^{-2n}$, $L:=2^{n+1}$ and $m:=4n$, one can identify $\cg^{(h,L)}_m\big(\widetilde{\<Psi>}^{\ka,n}\big)$ with the process $\bar{\<Psi>}^{\ka,n}$ described by relations \eqref{defi-psi-black}.

\smallskip

To this end, recall first that, following the definition of the operator $\cg^{(h,L)}_m$, we can expand $\cg^{(h,L)}_m\big(\widetilde{\<Psi>}^{\ka,n}\big)_{t_i}$ ($i=0,\ldots,2^{4n}$) as $\cg^{(h,L)}_m\big(\widetilde{\<Psi>}^{\ka,n}\big)_{t_i}=\sum_{j=-N+1}^{N-1}\bu^j_{t_i}  \Phi^n_j$, with $N:=\frac{L}{h}=2^{3n+1}$, and where the vector $(\bu_{t_i}^j)$ is given by the following rule (see \eqref{equation-bbu}): $\bu_{0}=0$ and for $i=0,\ldots,2^{4n}-1$,
\begin{equation}\label{identif-u-1}
\big[\ca_{h,L}+2^{-4n}\cb_{h,L}\big] \bu_{t_{i+1}}=\ca_{h,L}\bu_{t_{i}}+\int_{t_{i}}^{t_{i+1}} ds \, \langle (\partial_t \partial_x \widetilde{B}^{\ka,n})_s,\Phi^n\rangle \, .
\end{equation}
Observe that in this situation
$$(\ca_{h,L})_{j,k}=\langle \Phi^n_j,\Phi^n_k\rangle =\frac{2^{-2n+1}}{3}\bigg[\1_{\{j=k\}}+\frac14\1_{\{|j-k|=1\}}\bigg]$$
and
$$(\cb_{h,L})_{j,k}=\langle \nabla\Phi^n_j,\nabla\Phi^n_k\rangle =2^{2n+1}\bigg[\1_{\{j=k\}}-\frac12\1_{\{|j-k|=1\}}\bigg] \, .$$
Therefore, for every $j=-N+1,\ldots,N-1$,
\begin{equation}\label{identif-u-2}
\big(\big[\ca_{h,L}+2^{-4n}\cb_{h,L}\big] \bu_{t_{i+1}}\big)^j=\frac{2^{-2n+1}}{3}\sum_{k=-N+1}^{N-1}\bigg[4\bu_{t_{i+1}}^k\1_{\{j=k\}}-\frac{5}{4}\bu_{t_{i+1}}^k\1_{\{|j-k|=1\}}\bigg] 
\end{equation}
and
\begin{equation}\label{identif-u-3}
\big(\ca_{h,L}\bu_{t_{i}}\big)^j=\frac{2^{-2n+1}}{3}\sum_{k=-N+1}^{N-1}\bigg[\bu_{t_{i}}^k\1_{\{j=k\}}+\frac{1}{4}\bu_{t_{i}}^k\1_{\{|j-k|=1\}}\bigg]
\end{equation}
On the other hand, using the convention $\tilde{i},\dbtilde{j}$ introduced in Section \ref{subsec:descript-scheme}, and with expression \eqref{expansion-noise} in mind, one can write, for every $j=-N+1,\ldots,N-1$,
\begin{align}
&\int_{t_{i}}^{t_{i+1}} ds \, \langle (\partial_t \partial_x \widetilde{B}^{\ka,n})_s,\Phi^n_j\rangle\nonumber\\
&=\sum_{\ell=-2^{2n}}^{2^{2n}-1} \sum_{k=0}^{2^n-1} 2^{2n}(\sq_{k\ell}^n B^{\ka,n}) \int_{t_{i}}^{t_{i+1}} ds \, \1_{[\frac{k}{2^n},\frac{k+1}{2^n}[}(s) \int_{x_{j-1}}^{x_{j+1}} dx \, \1_{[\frac{\ell}{2^n},\frac{\ell+1}{2^n}[}(x) \Phi^n_j(x) \nonumber\\
&=2^{-2n}\sum_{\ell=-2^{2n}}^{2^{2n}-1}\sq^n_{\tilde{i},\ell}B^{\ka,n}\bigg[\1_{\{x_j>\frac{\dbtilde{j}}{2^n}\}}\int_{x_{j-1}}^{x_{j+1}} dx \, \1_{[\frac{\ell}{2^n},\frac{\ell+1}{2^n}[}(x) \Phi^n_j(x)+\1_{\{x_j=\frac{\dbtilde{j}}{2^n}\}}\int_{x_{j-1}}^{x_{j+1}} dx \, \1_{[\frac{\ell}{2^n},\frac{\ell+1}{2^n}[}(x) \Phi^n_j(x)\bigg] \, .\label{appli-sch}
\end{align}
Now
\begin{align*}
\sum_{\ell=-2^{2n}}^{2^{2n}-1}\sq^n_{\tilde{i},\ell}B^{\ka,n}\1_{\{x_j>\frac{\dbtilde{j}}{2^n}\}}\int_{x_{j-1}}^{x_{j+1}} dx \, \1_{[\frac{\ell}{2^n},\frac{\ell+1}{2^n}[}(x) \Phi^n_j(x)&=\1_{\{x_j>\frac{\dbtilde{j}}{2^n}\}}\sq^n_{\tilde{i},\dbtilde{j}}B^{\ka,n}\int_{x_{j-1}}^{x_{j+1}} dx \, \Phi^n_j(x)\\
&=\1_{\{x_j>\frac{\dbtilde{j}}{2^n}\}}2^{-2n}\sq^n_{\tilde{i},\dbtilde{j}}B^{\ka,n}\, ,
\end{align*}
while
\begin{align*}
&\sum_{\ell=-2^{2n}}^{2^{2n}-1}\sq^n_{\tilde{i},\ell}B^{\ka,n}\1_{\{x_j=\frac{\dbtilde{j}}{2^n}\}}\int_{x_{j-1}}^{x_{j+1}} dx \, \1_{[\frac{\ell}{2^n},\frac{\ell+1}{2^n}[}(x) \Phi^n_j(x)\\
&=\1_{\{x_j=\frac{\dbtilde{j}}{2^n}\}} \bigg[\sq^n_{\tilde{i},\dbtilde{j}-1}\!\! B^{\ka,n} \int_{x_{j-1}}^{x_{j}} dx \,  \Phi^n_j(x)+\sq^n_{\tilde{i},\dbtilde{j}}B^{\ka,n} \int_{x_{j}}^{x_{j+1}} dx \, \Phi^n_j(x)\bigg]\\
&=\1_{\{x_j=\frac{\dbtilde{j}}{2^n}\}} 2^{-2n}\bigg[\frac12\sq^n_{\tilde{i},\dbtilde{j}-1}\!\! B^{\ka,n}+\frac12\sq^n_{\tilde{i},\dbtilde{j}}\! B^{\ka,n} \bigg]\, .
\end{align*}
so that, going back to \eqref{appli-sch}, we have
\begin{equation}\label{identif-u-4}
\int_{t_{i}}^{t_{i+1}} ds \, \langle (\partial_t \partial_x \widetilde{B}^{\ka,n})_s,\Phi^n_j\rangle=\1_{\{x_j>\frac{\dbtilde{j}}{2^n}\}}2^{-4n}\sq_{\tilde{i},\dbtilde{j}}^n\! B^{\ka,n} +\1_{\{x_j=\frac{\dbtilde{j}}{2^n}\}}2^{-4n}\bigg[\frac12 \sq_{\tilde{i},\dbtilde{j}-1}^n\!\! B^{\ka,n}+\frac12 \sq_{\tilde{i},\dbtilde{j}}^n\! B^{\ka,n} \bigg] \, .
\end{equation}
Injecting \eqref{identif-u-2}, \eqref{identif-u-3} and \eqref{identif-u-4} into relation \eqref{identif-u-1}, we immediately see that the points $\bu_{t_i}^j$ obey the same iteration rule as the points $\bar{\<Psiblack>}^j_{t_i}$ in \eqref{defi-psi-black}. Therefore, $\bu_{t_i}^j=\bar{\<Psiblack>}^j_{t_i}$ and this achieves to prove the desired identity $\cg^{(h,L)}_m\big(\widetilde{\<Psi>}^{\ka,n}\big)=\bar{\<Psi>}^{\ka,n}$.

\end{proof}

\smallskip

Let us conclude this theoretical analysis with two remarks about our calibration choice $(h,L,m)$ in \eqref{appli-calib} (and therefore in the scheme of Section \ref{subsec:descript-scheme}, leading to the definition of $\bar{\<Psi>}^{\ka,n}$).

\smallskip

\begin{remark}\label{rk:non-optimal-choice}
Observe that in light of \eqref{speed-conv-2}, the calibration $(h,L,m):=(2^{-2n},2^{n+1},4n)$ ensures the convergence of the scheme for all possible values of $(H_0,H_1)\in (0,1)^2$, which, for the sake of clarity, motivated our choice in the scheme proposed in Section \ref{subsec:descript-scheme}. However, for \emph{fixed} $(H_0,H_1)$, one could naturally choose more optimal values for $(h,L,m)$, which would possibly reduce the number of computations in the associated system \eqref{defi-psi-black}.
\end{remark}

\smallskip

\begin{remark}\label{rk:l-2-h}
Due to our use of a Galerkin-type approximation procedure (and Galerkin-type bounding arguments), the estimate result in Proposition \ref{prop:discreti-s-t} is stated in terms of $L^2(\R)$-topology in space, in contrast with the weaker $\ch^{-\al}(\R)$-norm used in Proposition \ref{sto} and Proposition \ref{prop:discreti-noise}, and which is more natural in this rough setting (remember that the solution $\<Psi>$ is a path with values in $\ch^{-\al}(\R)$). By considering the $\ch^{-\al}(\R)$-norm in the left-hand side of \eqref{speed-conv-2}, it might be possible to improve the latter estimate with respect to the parameters $h$, $m$ or $n$ (without changing the topology in our main control \eqref{main-loose-estimate}). In turn, this could allow us to relax the current $(h,L,m)$-calibration of the scheme in Section \ref{subsec:descript-scheme}. Nevertheless, at this point, it is not clear to us how one could adapt the successive arguments of Section \ref{sec:space-discret} in order to - sharply - take negative-order Sobolev norms into account.
\end{remark}

\section{Numerical results and possible improvements}\label{sec:numerical}

We devote this last section to a few details and comments related to the numerical implementation of the algorithm \eqref{defi-psi-ka-n}-\eqref{defi-psi-black}. 

\subsection{Simulation of the scheme}

As described in Section \ref{subsec:descript-scheme}, the simulation of our discretized process $\<Psi>^{\ka,n}$ boils down to the computation of the values $\<Psiblack>_{t_i}^{j}$, along the iterative formula \eqref{defi-psi-black}. As far as randomness is concerned, we are thus left with the implementation of the quantities 
\begin{equation}\label{galerkin-grid}
\delta B^{\ka,n}_{ij}, \quad i=0,\ldots,2^{4n}, \ j=-N+1,\ldots,N-1\, , \ N:=2^{3n+1}.
\end{equation}

\smallskip

To this end, let us briefly recall that any Gaussian sheet can be easily simulated through its mean and covariance formulas. To be more specific, let us fix $t_1<\ldots<t_p$, $x_1<\ldots<x_q$, $p,q\geq 1$, and consider a centered Gaussian field $\{X_t(x), \, t,x\in \R\}$ with covariance given by 
$$\mathbb{E}[X_s(x)X_t(y)]=C_0(s,t)C_1(x,y) \, .$$
Then define the matrices $\mathtt{C_0},\mathtt{C_1}$ along the formulas $\mathtt{C_0}(i,i'):=C_0(t_i,t_{i'})$, $\mathtt{C_1}(j,j'):=C_1(x_j,x_{j'})$, and consider auxiliary symmetric matrices $\mathtt{D_0},\mathtt{D_1}$ such that $\mathtt{D_0}^2=\mathtt{C_0}$ and $\mathtt{D_1}^2=\mathtt{C_1}$ (in the subsequent implementations, $\mathtt{D_0},\mathtt{D_1}$ are computed with the help of the \texttt{sqrtm} function). Now, if $\mathtt{W}$ stands for a random matrix in $\mathcal{M}_{p\times q}(\mathbb{R})$ with independent and $\cn(0,1)$-distributed entries, we set 
\begin{equation}\label{simu-gaussian-sheet}
\mathtt{X}:=\mathtt{D_0}*\mathtt{W}*\mathtt{D_1}, 
\end{equation}
so that, for all $1\leq i,i'\leq p$ and $1\leq j,j'\leq q$,
\begin{align*}
&\mathbb{E}\big[ \mathtt{X}(i,j)\mathtt{X}(i',j')\big]=\sum_{k,k'=1}^p \sum_{\ell,\ell'=1}^q \mathtt{D_0}(i,k)\mathtt{D_0}(i',k')\mathtt{D_1}(\ell,j)\mathtt{D_0}(\ell',j') \mathbb{E}\big[\mathtt{W}(k,\ell) \mathtt{W}(k',\ell') \big]\\
&\hspace{2cm}=\sum_{k=1}^p \mathtt{D_0}(i,k)\mathtt{D_0}(k,i')\sum_{\ell=1}^q \mathtt{D_1}(j,\ell)\mathtt{D_0}(\ell,j') =\mathtt{C_0}(i,i') \mathtt{C_1}(j,j')= \mathbb{E}\big[X_{t_i}(x_j)X_{t_{i'}}(x_{j'})\big] \, .
\end{align*}

\

With this general strategy in mind, let us go back to our approximation $B^{\ka,n}$ of the fractional sheet $B$, for fixed Hurst indexes $H_0,H_1\in (0,1)$. According to the representation \eqref{b-bar-n}, $B^{\ka,n}$ corresponds indeed to a centered Gaussian field with covariance of the form 
\begin{equation}\label{cova-b-ka-n}
\mathbb{E}\big[B^{\ka,n}_s(x)B^{\ka,n}_t(y)\big]=C_0^{\kappa ,n}(s,t) C_1^{\kappa ,n}(x,y),
\end{equation}
with
$$C_0^{\kappa ,n}(s,t) := c_{H_0}^2\int_{|\xi |\leq 2^{2\kappa n}}d\xi  \,
\frac{(e^{i\xi t}-1)(e^{-i\xi s}-1)}{|\xi |^{2H_0+1}},\quad C_1^{\kappa ,n}(x,y):= c_{H_1}^2\int_{   |\eta |\leq 2^{\kappa n}}  d\eta\,
\frac{(e^{i\eta x}-1)(e^{-i\eta y}-1)}{|\eta |^{2H_1+1}} \, .$$
Note that the latter integrals can be more conveniently expanded as 
\begin{align*}
C_0^{\kappa ,n}(s,t)=&c_{H_0}^2\int_{|\xi |\leq 2^{2\kappa n}}d\xi  \,
\frac{\cos(\xi (t-s))-\cos(\xi t)-\cos(\xi s)+1}{|\xi |^{2H_0+1}}\\
=&c_{H_0}^2\, 2^{1-{4H_0\kappa n}}\int_0^{ 1}d\xi \,
\frac{\cos(2^{2\kappa n}\xi (t-s))-\cos(2^{2\kappa n}\xi t)-\cos(2^{2\kappa n}\xi s)+1}{|\xi |^{2H_0+1}},
\end{align*}
with a similar expression for $C_1^{\kappa ,n}(x,y)$, which in turn allows us to approximate $C_0^{\kappa ,n},C_1^{\kappa ,n}$ through a standard Riemann-sum procedure, i.e. as
\begin{align*}
C_{0}^{\kappa ,n}(s,t)\approx c_{H_0}^2\frac{2^{1-{4H_0\kappa n}}}{M_0}\sum_{m=1}^{M_0}  
\frac{\cos(2^{2\kappa n}\frac{m}{M_0} (t-s))-\cos(2^{2\kappa n}\frac{m}{M_0} t)-\cos(2^{2\kappa n}\frac{m}{M_0} s)+1}{|\frac{m}{M_0} |^{2H_0+1}}\, ,
\end{align*}
\begin{align*}
C_{1}^{\kappa ,n}(x,y)\approx c_{H_1}^2\frac{2^{1-{2H_1\kappa n}}}{M_1}\sum_{m=1}^{M_1}  
\frac{\cos(2^{\kappa n}\frac{m}{M_1} (x-y))-\cos(2^{\kappa n}\frac{m}{M_1} x)-\cos(2^{\kappa n}\frac{m}{M_1} y)+1}{|\frac{m}{M_1} |^{2H_1+1}}\, ,
\end{align*}
with $M_0,M_1$ large enough. 

\

As a consequence of decomposition \eqref{cova-b-ka-n}, the values of $B^{\ka,n}_{\frac{i}{2^n}}(\frac{j}{2^n})$ can now be easily simulated through the above-described method, i.e. using \eqref{simu-gaussian-sheet}, which immedialy provides us with the set of increments 
\begin{equation}\label{noise-grid}
\sq^n_{ij}\!\! B^{\ka,n}, \quad i=0,\ldots,2^n, \ j=-2^{2n},\ldots,2^{2n},
\end{equation}
involved in the scheme. Observe that, following \eqref{involvement-b-ka-n-scheme}, each quantity $\delta B_{i,j}^{\kappa ,n}$ in \eqref{defi-psi-black} is in fact computed from the pair $(\sq^n_{\tilde{i},\dbtilde{j}} B^{\ka,n},\sq^n_{\tilde{i},\dbtilde{j}-1} B^{\ka,n})$, where $\tilde{i}:=\lfloor i2^{-3n} \rfloor$ and $\dbtilde{j}:=\lfloor j2^{-n} \rfloor$.

\

Once endowed with the (renormalized) quantities $\beta (j,i):=\frac{3}{2^{2n+1}} \delta B_{i,j}^{\kappa ,n}$, the simulation of \eqref{defi-psi-black} merely relies on the consideration of the two matrices
\begin{align*}
\mathtt{A_1}:=\begin{pmatrix}
4 & -\frac{5}{4} & 0 & \cdots & 0 \\ 
-\frac{5}{4} & 4 & \ddots & \ddots & \vdots \\ 
0 & \ddots & \ddots & \ddots & 0 \\ 
\vdots & \ddots & \ddots & 4 & -\frac{5}{4} \\ 
0 & \cdots & 0 & -\frac{5}{4} & 4
\end{pmatrix}\quad
\textrm{ and }
\quad
\mathtt{A_2}:=\begin{pmatrix}
1 & \frac{1}{4} & 0 & \cdots & 0 \\ 
\frac{1}{4} & 1 & \ddots & \ddots & \vdots \\ 
0 & \ddots & \ddots & \ddots & 0 \\ 
\vdots & \ddots & \ddots & 1 & \frac{1}{4} \\ 
0 & \cdots & 0 & \frac{1}{4} & 1
\end{pmatrix} .
\end{align*} 
Namely, setting $\phi (j,i):=\<Psiblack>_{t_i}^{j}$, formula \eqref{defi-psi-black} can be more efficiently recast into the iterative scheme
\begin{align}\label{simul-phi}
\mathtt{A_1}\phi(\cdot , 1)=\beta (\cdot , 1), \quad \mathtt{A_1}\phi(\cdot , i+1)=\mathtt{A_2}\phi(\cdot , i)+\beta (\cdot , i) \quad \text{for all} \ i>1,
\end{align}  
the implementation of which becomes an easy task (see for instance Figure \ref{Fig2} for a simulation with $n=3$, $\kappa=1$, $M_0=10000$, $M_1=1000$ and $H_0=H_1=\frac14$).

\begin{figure}[htbp]
  \centering
  \includegraphics[scale=0.25]{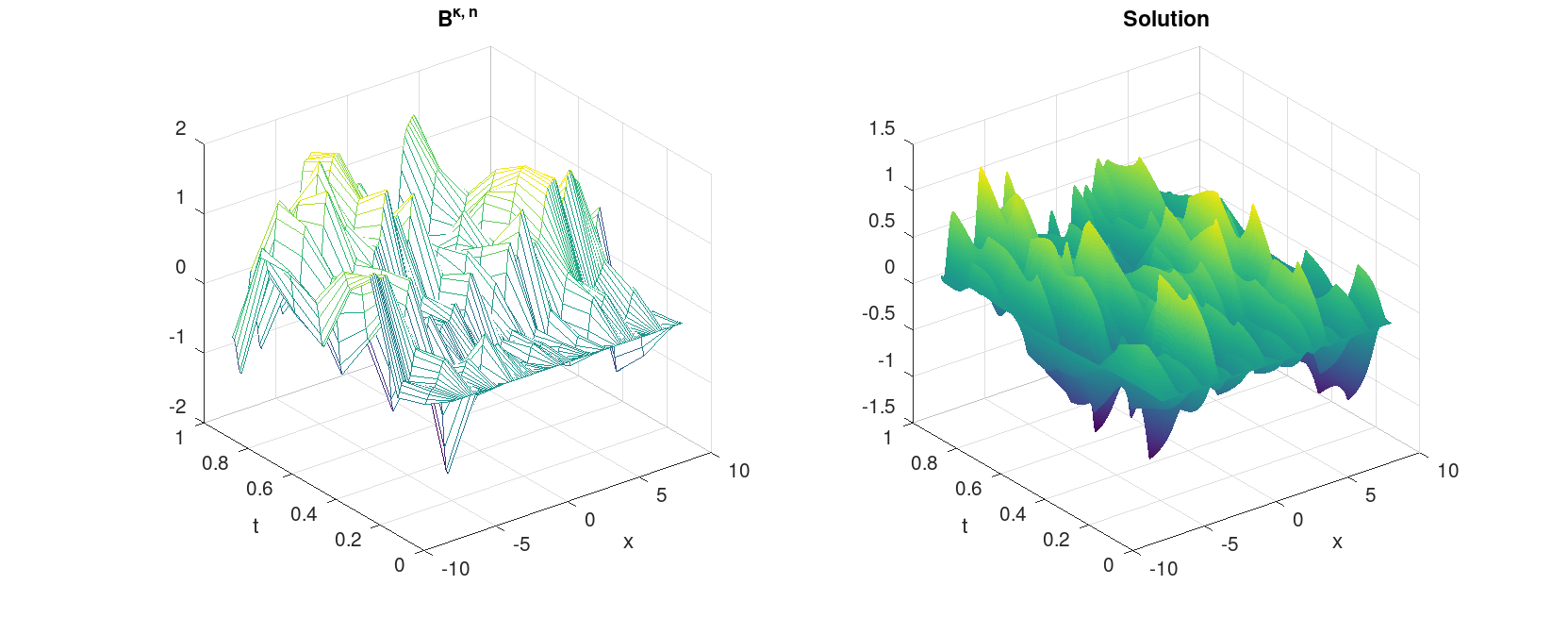}
   \caption{$H_0=H_1=\frac14$}
\label{Fig2}
  \end{figure}

\

\begin{remark}
We cannot hide the fact that, due to our consideration of a grid with extremely fine mesh (namely $2^{-4n}$ in time, $2^{-3n+1}$ in space) and growing support (in space), the simulation of the above scheme soon reveals to be highly demanding as $n$ increases, and we have actually not been able to implement the algorithm for $n\geq 4$. As a consequence of this computational restriction, the process $B^{\ka,n}$, with $\ka> 0$ small enough (following the result of Theorem \ref{main-theo}) and $1\leq n\leq 3$, can only be seen as a coarse approximation of $B$, which explains the relative smoothness of the simulated sheets in Figure \ref{Fig2}. 

\end{remark}

\subsection{Open issues}

Let us conclude the study with two natural open questions raised by the above simulation procedure, and which could motivate possible future improvements of our theoretical results.

\

\noindent
\textit{Replacing $B^{\ka,n}$ with $B$.} 
As we emphasized it in Remark \ref{rk:ka}, our restriction on the parameter $\ka$ (in $B^{\ka,n}$) plays an important technical role in the proof of the convergence property \eqref{speed-converg-1}, but we cannot firmly assert that the result of Theorem \ref{main-theo} (or some similar convergence statement toward $\<Psi>$) would fail for larger values of $\ka>0$, or even for $\ka=\infty$, which corresponds to replacing $B^{\ka,n}$ with $B$ in the algorithm.

\smallskip

Figure \ref{Fig3} corresponds to a simulation of such a modified scheme (where $\ka=\infty$ , $H_0=H_1=\frac14$ and $n=3$), and thus the resulting sheet might represent a more faithful approximation of $\<Psi>$.

\begin{figure}[htbp]
  \centering
  \includegraphics[scale=0.25]{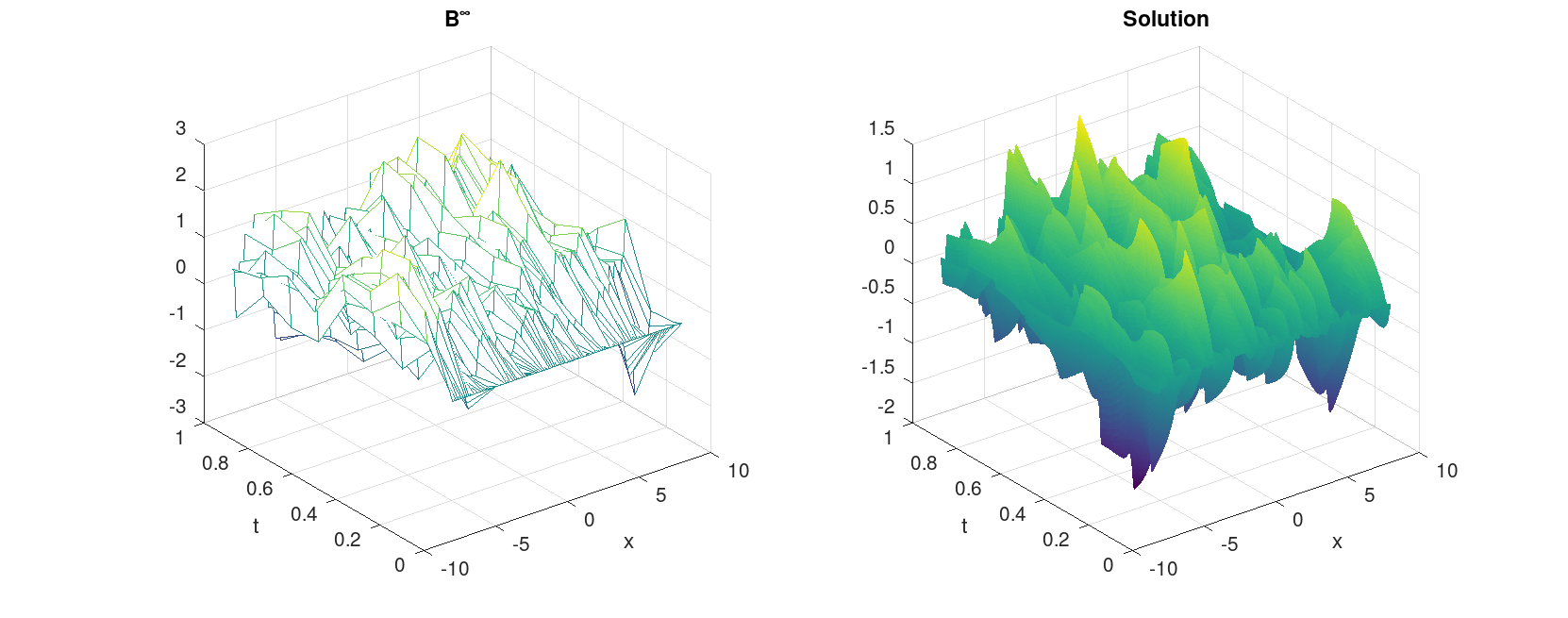}
   \caption{$H_0=H_1=\frac14$}
\label{Fig3}
  \end{figure}

\

\

\noindent
\textit{Grid synchronization.} Another natural question arising from our scheme is to know whether the convergence in Theorem \ref{main-theo} (or some similar property) would hold if one discretized the noise and the heat operator \emph{over the same grid} (say $t_i=\frac{i}{2^n},x_j=\frac{j}{2^n}$).

\smallskip

Recall that the strong discrepancy between the \enquote{Galerkin grid} in \eqref{galerkin-grid} and the \enquote{noise grid} in \eqref{noise-grid} is - at least partially - due to our treatment of $B^{\ka,n}$ as a $L^\infty$-function throughout Section \ref{sec:space-discret}. As we mentionned it in Remark \ref{rk:l-2-h}, we expect some more direct analysis in $\ch^{-\al}$ to provide sharper estimates with respect to the stochastic perturbation, which in turn could allow us to - at least partially - fill the gap between the two grids.

\smallskip

In this setting, Figure \ref{Fig4} (where $n=5$, $H_0=H_1=\frac14$) accounts for the simulation of the corresponding \enquote{synchronized} scheme over the common grid $ t_i=\frac{i}{2^n}$, $x_j=\frac{j}{2^n}$. We have also provided a simulation of this scheme in a more regular situation for which $2H_0+H_1 >1$ (see Figure \ref{Fig5}, where $H_0=H_1=\frac34$ and $n=5$): the two figures \ref{Fig4} and \ref{Fig5} thus offer a clear contrast between the regular \enquote{functional} case, and the rough \enquote{distributional} regime.

\

\begin{figure}[h]
  \includegraphics[scale=0.25]{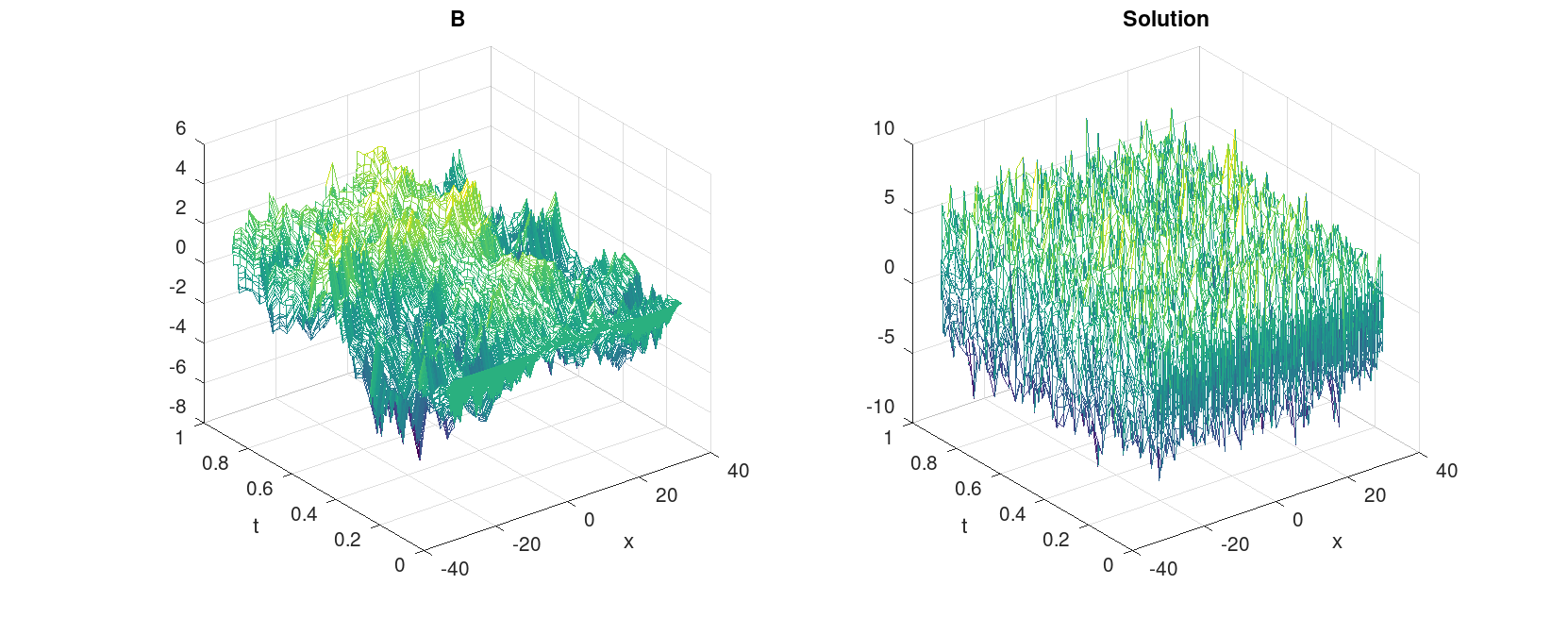}
   \caption{$H_0=H_1=\frac14$}
\label{Fig4}
  \end{figure}

\begin{figure}[h]
  \includegraphics[scale=0.25]{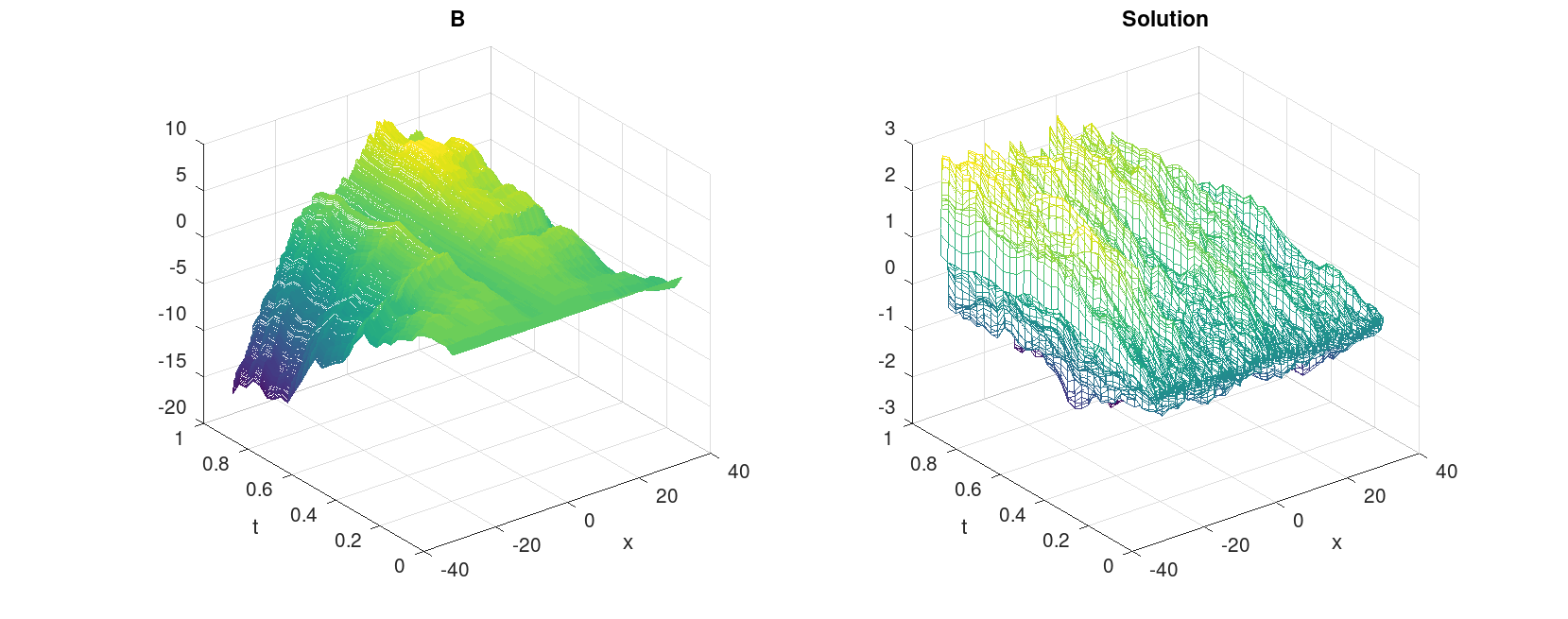}
   \caption{$H_0=H_1=\frac34$}
\label{Fig5}
  \end{figure}

\end{document}